\documentclass[pdflatex,sn-mathphys-num,twocolumn]{sn-jnl}
\usepackage{graphicx}%
\usepackage{multirow}%
\usepackage{amsmath,amssymb,amsfonts}%
\usepackage{amsthm}%
\usepackage{mathrsfs}%
\usepackage[title]{appendix}%
\usepackage{xcolor}%
\usepackage{textcomp}%
\usepackage{manyfoot}%
\usepackage{booktabs}%
\usepackage{algorithm}%
\usepackage{algorithmicx}%
\usepackage{algpseudocode}%

\theoremstyle{thmstyleone}%
\newtheorem{theorem}{Theorem}%  meant for continuous numbers
\newtheorem{proposition}[theorem]{Proposition}% 

\usepackage{comment}
\usepackage{bbm}
\usepackage{subcaption}

\theoremstyle{thmstyletwo}%
\newtheorem{remark}{Remark}%

\theoremstyle{thmstylethree}%
\newtheorem{definition}{Definition}%
\newtheorem{lemma}{Lemma}

\DeclareMathAlphabet\gothic{U}{euf}{m}{n}
\DeclareMathOperator{\SE}{SE}
\DeclareMathOperator{\SO}{SO}

\newcommand{\ul}{\mathbf}
\newcommand{\R}{\mathbb{R}}
\newcommand{\bp}{\mathbf{p}}
\newcommand{\bq}{\mathbf{q}}
\newcommand{\bx}{\mathbf{x}}
\newcommand{\by}{\mathbf{y}}
\newcommand{\bn}{\mathbf{n}}
\newcommand{\cF}{\mathcal{F}}

\DeclareRobustCommand{\indbb}[1]{\ensuremath{\mathbbm{1}_{#1}}}

\begin{document}
\title[Article Title]{Sub-Riemannian Snakes on the Projective Line Bundle 
%$\R^2 \times P^1$ 
with Applications to Segmentation of SEM Images}


\author*[1,2]{\fnm{Leanne} \sur{Vis}}\email{l.vis@tue.nl}

\author[1,2]{\fnm{Maxim} \sur{Pisarenco}}\email{m.pisarenco@tue.nl}

\author[1]{\fnm{Bart M. N.} \sur{Smets}}\email{b.m.n.smets@tue.nl}

\author[3]{\fnm{Fons} \spfx{van der} \sur{Sommen}}\email{fvdsommen@tue.nl}

\author[1]{\fnm{Remco} \sur{Duits}}\email{r.duits@tue.nl}


\affil*[1]{\orgdiv{Department of Mathematics and Computer Science, CASA \& EAISI}, \orgname{Eindhoven University of Technology}, \orgaddress{\street{Groene Loper 3}, \city{Eindhoven}, \country{the Netherlands}}}

\affil[2]{\orgname{ASML Research}, \orgaddress{\street{De Run 6501}, \city{Veldhoven}, \country{the Netherlands}}}

\affil[3]{\orgdiv{Department of Electrical Engineering, ARIA \& EAISI}, \orgname{Eindhoven University of Technology}, \orgaddress{\street{Groene Loper 3}, \city{Eindhoven}, \country{the Netherlands}}}


\abstract{
    Geodesic tracking on the projective line bundle $\R^2 \times P^1 $ has many uses, including the segmentation of objects in images. 
    However, global tracking requires expensive distance map computations. 
    %by fast marching and/or solving the eikonal PDE
    We provide a practical solution to this problem by introducing a snake model on $\R^2 \times P^1$, where we only compute the distance map where needed. 
    Our method introduces a geometric criterion for switching between fast spatial snakes and computing minimizing geodesics of a new projective line bundle model.
    The new pseudo-distance underlying our geometric model is both symmetric and cusp-free, in contrast to previous geodesic sub-Riemannian models on $\R^2 \times P^1$. 
    Our pseudo-distance satisfies the triangle inequality on a large set that we characterize, and includes a connected-component-informed cost function, which is highly advantageous in applications.
    Experiments on Scanning Electron Microscopy (SEM) images demonstrate our method's robust, automatic segmentation of overlapping electronic structures.
    % We apply our tracking method to a variety of Scanning Electron Microscopy (SEM) images of electronic devices, where we demonstrate automatic robust and correct segmentation of overlapping structures.
}




\maketitle

\section{Introduction}\label{sec1}
Segmentation of overlapping structures in two dimensional images (with correct tracing of edges) is a notorious problem in image processing \cite{BekkersSIAM,citti2006cortical,cohen2001multiple,Benmansour,chen2017global,Chen2024,duitsmeestersmirebeauportegies,FRANCESCHIELLO201955}.

% add PhD thesis Erik Bekkers, %and Da Chen here! 
%Also recent paper %https://link.springer.com/article%/10.1007/s11263-023-01881-z
% %
% and %https://www.sciencedirect.com/sci%ence/article/pii/S0926224518302195
This paper presents a novel distance map on the projective line bundle and uses it for geodesic tracking and segmentation of overlapping structures in images. 
%As we explain in details in the following, the 3 main contributions of this article are: 
The distance map we propose provides a new geodesic tracking model with several advantages, including: 

% 1. In contrast with previous work [], The geodesics defined by the distance maps do not have cusps, which are undesirable artefacts that do not occur in the edges of images.  

% 2. It integrates the connected component algorithm of \citet{berg2024connectedcomponentsliegroups} and includes a criterion to switch from a geodesic tracking to a quick snake optimization. 

% 3. It shows improved performance in automatic segmentation and grouping of overlapping structures in Scanning Electron Microscopy (SEM) images of electronic devices. 

\begin{itemize}
    \item 
    In contrast to previous work on tracking on the projective line bundle $\R^2 \times P^1$ by \citet{BekkersGSI}, 
    our geodesics do not exhibit cusps after projecting them on the image.
    Cusps are undesirable artifacts \cite{Boscain2,duitsmeestersmirebeauportegies} that do not occur in the edges found in most types of images. 
    Dealing with them on $\R^2\times P^1$ is challenging. 
    
    %There are also other significant advantages of our model which we %mathematically 
    %analyze in this article. %This is one of the differences and advantages with respect to the previous model as we analyze in our main theorem. 
    \item 
    It integrates the connected component algorithm of \citet{berg2024connectedcomponentsliegroups}, which enables it to better deal with overlapping structures.
    \item 
    It includes a criterion to switch between accurate \emph{geodesic tracking} and quick \emph{snake optimization} to keep the computational cost low. 
\end{itemize}
Together, this leads to an improved tracking and segmentation method, compared to the previous work of \cite{BekkersGSI}. We demonstrate this in our experiments, which show improved performance in automatic segmentation and grouping of overlapping structures in Scanning Electron Microscopy (SEM) images of electronic devices where overlapping structures are the norm. 

We first provide a literature review on \emph{snakes} and \emph{geodesic tracking} models in image analysis over the past decades. 
Then we highlight the relevance of curve optimization on the projective line bundle, and introduce the new model. 
At the end of this introduction we explain our main contributions and novelties of this article in detail. 

\subsection{Previous Work on Geodesic Tracking and Snakes}

%Active contour modeling (`snakes') and geodesics active contours have become popular in around 1990. 
In 1988, \citet{kass1988snakes} introduced active contour models. 
The idea is that a parameterized curve (called a `snake') is fitted to the local differential structure of an image, such as smooth edges and lines.
A snake is an energy-minimizing spline that minimizes three forces, namely (1) the internal forces of the spline generated by the resistance of deformation, (2) the image forces that pull the snake toward the local differential structures present in the image, and (3) the external constraint forces specified by the user. 

The energy minimizing procedure is an iterative process, where an initial contour (manually specified or generated automatically) is moved iteratively towards the local differential structure.
Active contour modeling requires the adjustment of the smoothness parameter, a parameter which balances the resistance to deformation of the spline and the image forces. %To reduce user input, active contour models may possibly be enriched with statistical template matching \cite{cootes1995active}. 
%It can, for example, be used to generate an initial contour. 

New geometric models (e.g., level sets associated with mean curvature flows \cite{chan1999active,caselles1993geometric,malladi1995shape}) were introduced based on geometric partial differential equations; cf.~Sapiro \cite{Sapirobook} and the references therein. 

\citet{caselles1997geodesic} introduced the \emph{geodesic} contour model. They showed how a curve evolution 
(in this case the geodesic curve flow instead of the mean curvature flow) 
can be derived from the classical snake model.
The geodesic contour model builds on formal Riemannian geometry \cite{Jost} where stationary curves of the Riemannian distance are called geodesics. If a geodesic is a global minimizer of the Riemannian distance connecting two given boundary points, it is called a \emph{minimizing} geodesic. 
The geodesic contour model \citet{caselles1997geodesic} constrains the snake to be a minimizing geodesic of a data-driven Riemannian metric. 

Geodesic contour or tracking modeling is a two-step procedure; first, a distance map is computed, which is followed by a gradient descent step for computation of \emph{minimizing} geodesics. 
%Recall, that geodesics are stationary curves of the Riemannian metric and need not be mini

Computations of the level sets of viscosity solutions \cite{Evansbook} of PDEs considered in \cite{caselles1997geodesic}, and of related eikonal PDEs in \cite{osher2001level} are generally expensive, despite their vast advantages in image analysis, as shown in \cite{caselles1997geodesic} and \cite{osher2001level}. 
Therefore, effective fast-marching algorithms to solve these PDEs were developed \cite{sethian2000fast,osher1988fronts} based on Tsiklis' method \cite{tsitsiklis1995efficient}, with generalizations to general manifolds \cite{kimmel1998computing}.
For details on how we compute the geodesics, see Section~\ref{ch:track}.

%In the last decade, 
Several works have advocated the use of such PDE-methods on the space of positions and orientations $\R^2 \times S^1$ instead of the position space $\R^2$
\cite{Benmansour,BekkersSIAM,chen2017visual,chen2017global,duitsmeestersmirebeauportegies}. In these methods, images are lifted to $\R^2 \times S^1$, allowing one to appropriately deal with complex structures, such as overlapping structures, crossing lines \cite{BekkersSIAM}, and even %(non-smooth) 
bifurcations \cite{duitsmeestersmirebeauportegies}.
%Geodesic contour model is a two-step procedure; first a distance map is computed (as a viscosity solution of the eikonal PDE on this higher-dimensional homogeneous space) followed by a gradient descent for computation of the (in this case, sub-Riemannian or, even more generally, sub-Finslerian) geodesics. 

For computing the distance map on $\R^2\times S^1$, there are two main approaches: 

\begin{enumerate}
    \item 
    advanced and efficient anisotropic fast marching methods \cite{mirebeau_anisotropic_2014}, following a semi-Lagrangian or Hamiltonian approach \cite{mirebeau2019hamiltonian},
    
    \item  iterative (eikonal) PDE solvers \cite{BekkersSIAM,berg2025crossing}, which can be truly sub-Riemannian. 
    %and therefore more accurate. 
    %Although these methods are more computationally demanding, they are easy to parallelize, which offsets this disadvantage on modern hardware.
\end{enumerate}

Initially, \citet{SanguinettiFM} showed that the difference in error between both approaches is small.
However, they advocated anisotropic fast marching methods because they were much faster. 

Recently, GPU-based methods have made the more accurate PDE methods a viable, flexible, and fast alternative, as shown in \cite{berg2025crossing}. %via TaiChi \cite{hu2019taichi}
%For the numerical methods in this paper 
We used the latter GPU-approach. 
%mainly because of the flexible, fast, and simple structure of the algorithm.  

%These methods typically act on the homogeneous space of positions and orientations \mbox{$\R^2 \times S^1$} diffeomorphic to the Lie group $\SE(2)=\R^2 \rtimes \SO(2)$ under identification
%\[R_{\theta} \leftrightarrow \theta \in \R/(2\pi \mathbb{Z})\leftrightarrow \bn(\theta)=(\cos\theta, \sin \theta) \in S^1.\]

\subsection{A Motivation for using the Projective Line Bundle \texorpdfstring{\unboldmath $\R^2 \times P^1$}{R2 x P1}}

\begin{figure}[ht]
    \centering
    \begin{subfigure}[T]{0.45\textwidth}
        \centering
        \includegraphics[width=0.67\textwidth]{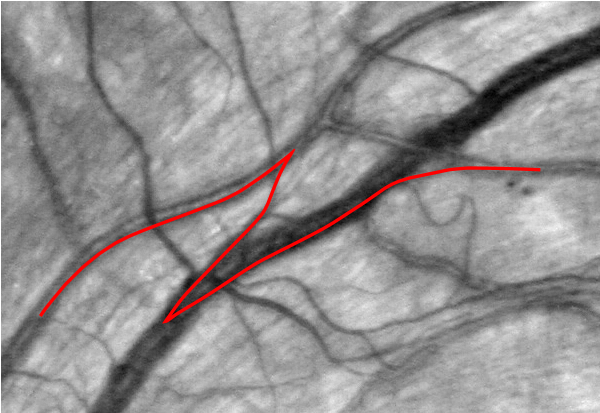}
        \caption{Result of the previous geodesic tracking model on the projective line bundle \cite{BekkersGSI} between two points on a blood vessel. Note that this model can exhibit cusps.}
        \label{fig:previous_tracking}
    \end{subfigure}
    \begin{subfigure}[T]{0.45\textwidth}
        \centering
        \includegraphics[width=0.67\textwidth]{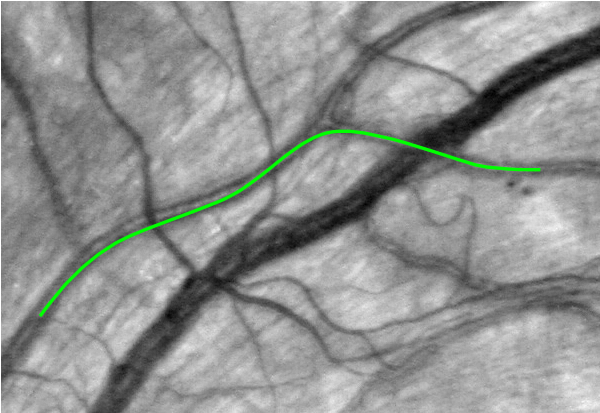}
        \caption{Result of the new cusp-free geodesic tracking model on the projective line bundle,  Eq.~\eqref{dmon}.}
        \label{fig:new_tracking}
    \end{subfigure}
    \caption{Visualization of the benefits of defining a cusp-free geodesic tracking model on the projective line bundle. The cusp in the Fig.~\ref{fig:previous_tracking} arises after spatial projection of the smooth geodesic in $\R^2 \times P^1$ of the model as proposed in \cite{BekkersGSI}.}
    \label{fig:medical_images}
\end{figure}

\begin{figure}[ht]
    \centering
    \begin{subfigure}[T]{0.45\textwidth}
        \centering
        \includegraphics[width=0.9\textwidth]{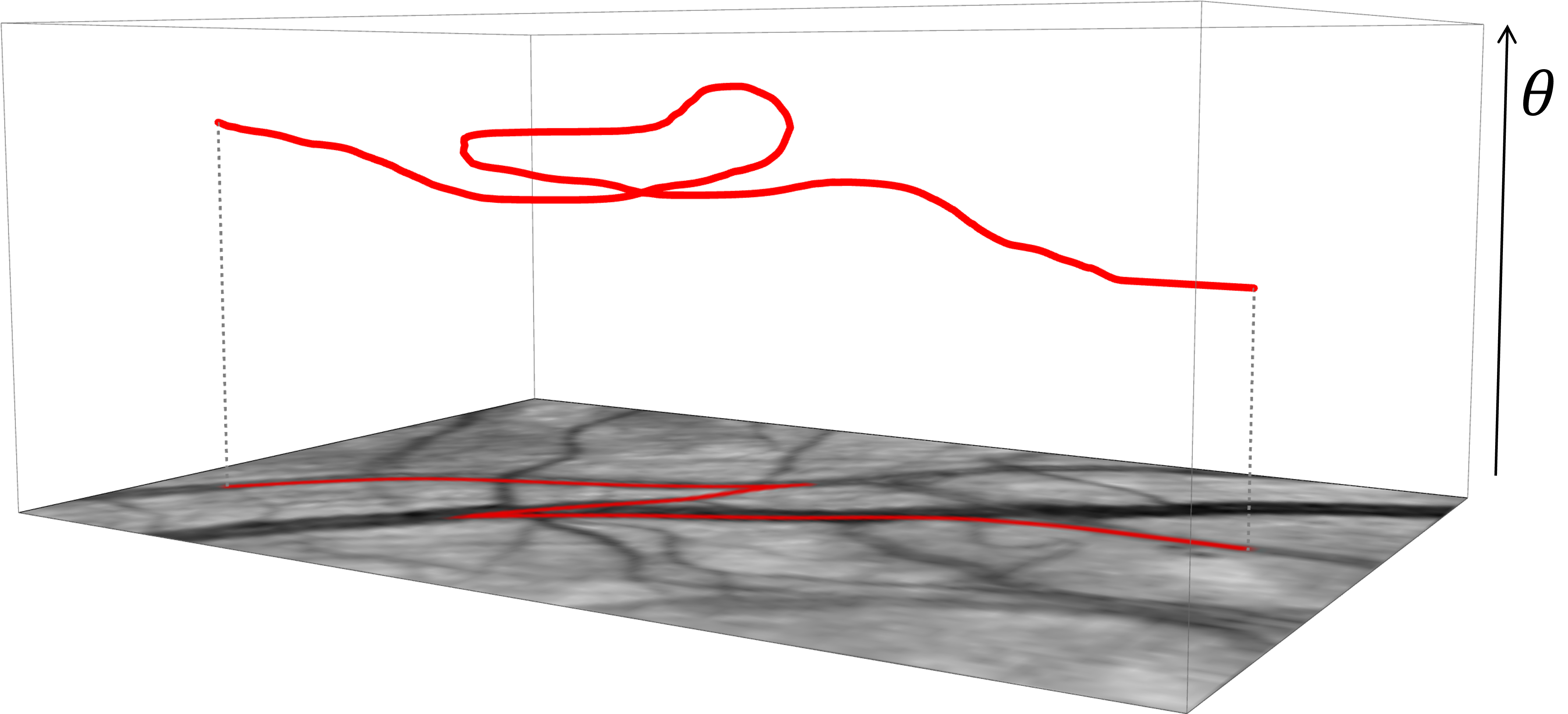}
        \caption{The geodesic found by the previous tracking model on the projective line bundle, cf. \cite{BekkersGSI}. Note that the geodesic is smooth in $\R^2 \times P^1$. The cusps arises after its spatial projection, cf. Fig.\ref{fig:previous_tracking}.}   
        \label{fig:previous_tracking2}
    \end{subfigure}
    \begin{subfigure}[T]{0.45\textwidth}
        \centering
        \includegraphics[width=0.9\textwidth]{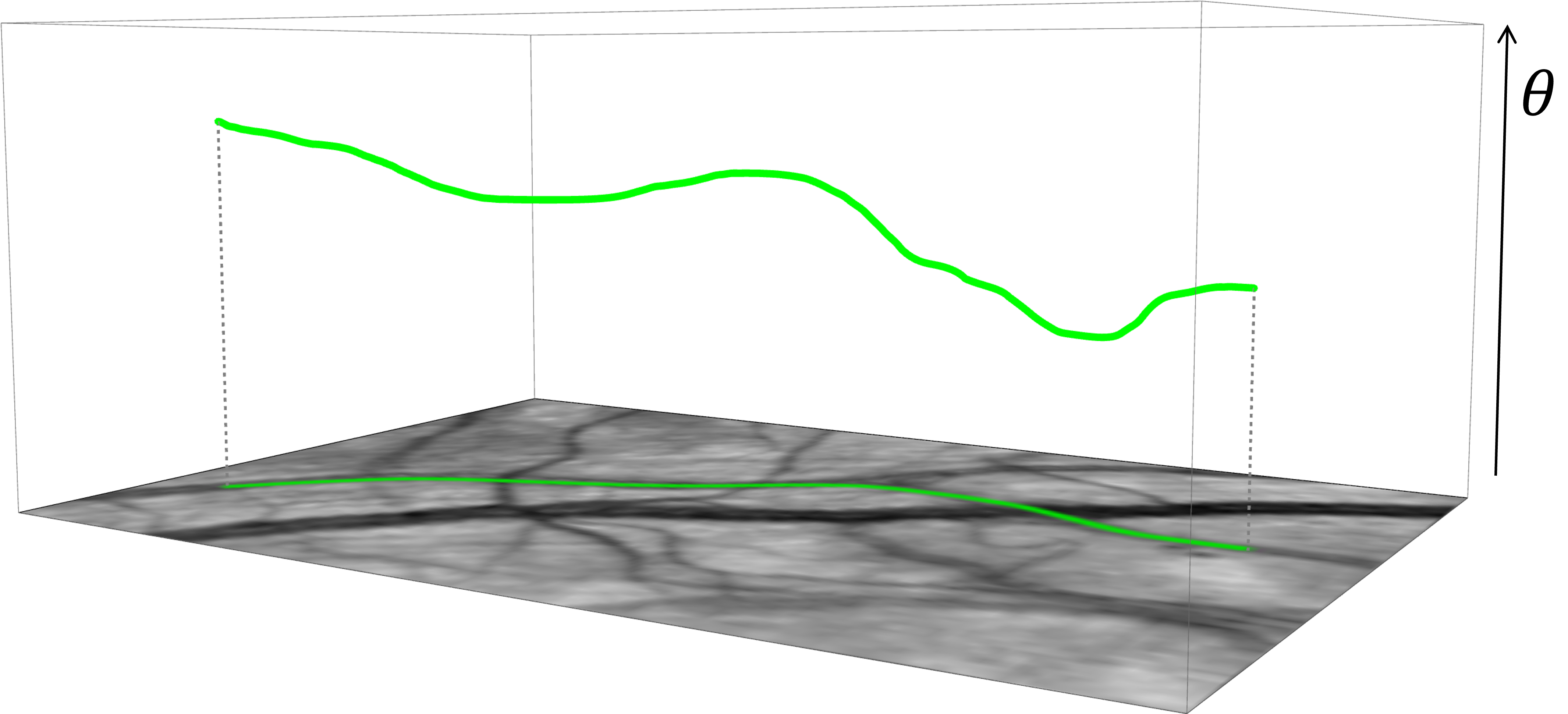}
        \caption{The geodesic found by the new cusp-free geodesic tracking model on the projective line bundle as described in Eq.~\eqref{dmon}. The spatial projection of this new geodesic does not exhibit a cusp cf. Fig.\ref{fig:new_tracking}.} 
        \label{fig:new_tracking2}
    \end{subfigure}
    \caption{Visualization of the smooth geodesics of Fig.~\ref{fig:medical_images} in $\R^2 \times P^1$, where $\theta(t) \in [0,\pi)$ represents the ``orientation'' of the geodesic given by the set $\{\bn(t),-\bn(t)\}$, where $\bn(t)=(\cos \theta(t), \sin \theta(t))$.
    \label{fig:medical_images2}}    
\end{figure} 

Consider a piecewise smooth edge in a two-dimensional image.
A regular point on the edge can be described by its position and the ``orientation'' \emph{along} the edge at this point.
One might decide to encode the orientation using a unit vector $\bn \in S^1$.
However, the antipodal vector $-\bn$ can also be used to describe its orientation.
Given this ambiguity, we simply decide to encode the ``orientation'' of an edge as the set $\{\bn, -\bn\}$.
More formally, we can describe the ``orientation'' as a point in the \emph{projective line} $P^1 := S^1 / \sim$, where $\sim$ is the equivalence relation that identifies antipodal unit vectors with each other, that is, $\bn_1\sim \bn_{2} \iff \bn_1=\pm \bn_2$.

Given the above, we decide to describe a local edge as a point in the \emph{projective line bundle} $\R^2 \times P^1$, that is, as the equivalence class $[\bp]=\{\bp,\overline{\bp}\}$ where $\bp:=(\bx,\bn) \in \R^2 \times S^1$ and $\overline{\bp}:=(\bx,-\bn)$.

In this article, we will study geodesic tracking models on the projective line bundle $\R^2\times P^1$, which %rely on 
minimize over
multiple curve optimizations on \mbox{$\R^2\times S^1$}.
Eight pairs could be made in $\R^2 \times S^1$ out of the two points on the projective line bundle, $[\bp_0]=\{\bp_0,\overline{\bp}_0\}$ and $[\bp_1]=\{\bp_1,\overline{\bp}_1\}$, namely:
\begin{equation}\label{eq:8 pairs}
\begin{array}{cc}
    \text{from }\bp_0 \text{ to } \bp_1,\text{ from } 
    \bp_0 \text{ to } \overline{\bp}_1,\\
    \text{from } \overline{\bp}_0 \text{ to } \bp_1, \text{ from }
    \overline{\bp}_0 \text{ to } \overline{\bp}_1,\\
    \text{from }\bp_1 \text{ to } \bp_0,\text{ from } 
    \bp_1 \text{ to } \overline{\bp}_0,\\
    \text{from } \overline{\bp}_1 \text{ to } \bp_0, \text{ from }
    \overline{\bp}_1 \text{ to } \overline{\bp}_0.
    \end{array}
\end{equation}
A priori, we need to consider all eight pairs, since the %asymmetric Finslerian 
length (as we will see in Def.~\ref{def:ASFD}) of the geodesic from $\bp$ to $\bq$ need not be equal to the length of the geodesic from $\bq$ to $\bp$ for all $\bp, \bq \in \R^2 \times S^1$.
% All eight pairs must be considered, since the geodesic tracking models used to connect two points are not necessarily symmetric.
%Later, when we explain the geometric models, we recognize symmetries that allow to reduce these $8$ cases to either $4$ or $2$ depending on the model.  

% \begin{definition}[Antipodal points]\label{def:eq}
%     A point $[\bp] \in \R^2\times P^1$ is given by the equivalence class $[\bp]=\{\bp,\overline{\bp}\}$ 
%     with $\bp:=(\bx,\bn) \in \R^2 \times S^1$ and $\overline{\bp}:=(\bx,-\bn) \in \R^2\times S^1$. 
% \end{definition}

In \citet{BekkersGSI}, a symmetric %(sub-Riemannian) 
geodesic tracking model on $\R^2 \times S^1$ is proposed to connect the points $[\bp_0]$ and $[\bp_1]$ in the projective line bundle $\R^2 \times P^1$. 
We will refer to this model as $d_{\mathrm{proj}}$, where `proj' is short for projective line bundle. 
The details of the model $d_{\mathrm{proj}}$ are given in Section~\ref{sec:dproj}.
The model $d_{\mathrm{proj}}$ is well-posed on $\R^2 \times P^1$, and builds on a data-driven version \cite{BekkersSIAM} of the sub-Riemannian model on $\R^2 \times S^1$ by Citti \& Sarti \cite{citti2006cortical}.

By symmetry arguments in the model $d_{\mathrm{proj}}$ \cite[Eq.~4]{BekkersGSI}, Bekkers et al. showed that the 8 pairs in \eqref{eq:8 pairs} can be reduced to only 2 representative pairs:
%\begin{equation*}
$    \text{from }\bp_0 \text{ to } \bp_1,\text{ and from } 
    \overline{\bp}_0 \text{ to } \bp_1.
    $
%\end{equation*}
%where, again due to symmetries, the final pair can be replaced by the pair $\bp_0 \text{ to } \overline{\bp}_1$.

However, the model $d_{\mathrm{proj}}$ sometimes resulted in geodesics with so-called \emph{cusps}, cf.~\citet{Boscain2}. 
An example of a cusp is visualized in Figs.~\ref{fig:previous_tracking},~\ref{fig:previous_tracking2}.
The cusp-problem is a fundamental problem, which we aim to tackle via a new well-posed model on the projective line bundle  $\R^2\times P^1$. 
The definition of a cusp is given in Section~\ref{ch:cuspscontrol}.

Intuitively, if we were to compare a geodesic with the trajectory of a car, a cusp can be viewed as reversing the motion direction of a car (switching from a forward to a backward motion or the other way around).
In this article, we will overcome this problem by setting up a new model on the projective line bundle $\R^2\times P^1$. 
% ; see Figs.~\ref{fig:new_tracking},~\ref{fig:new_tracking2}.

\subsection{A Cusp-Free Model on \texorpdfstring{\unboldmath $\R^2 \times P^1$}{R2 x P1}}

\begin{figure*}[t]
\centering
    \hfill
    \begin{subfigure}[t]{0.3\textwidth}
        \centering
        \includegraphics[width=3.3cm, trim={75 75 75 75}, clip]{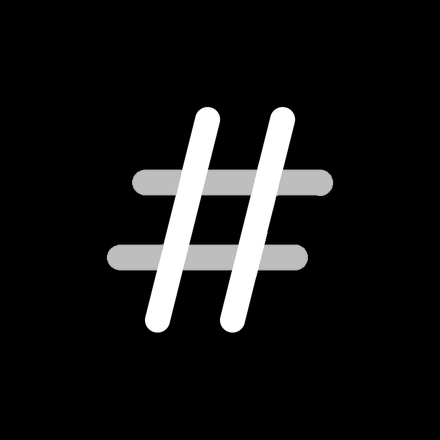}
        \caption{Example image with overlapping structures. The background structures %in the background 
        have lower intensity.}
        \label{fig:CC_input}
    \end{subfigure}    
    \hfill
    \begin{subfigure}[t]{0.3\textwidth}
        \centering
        \includegraphics[width=3.3cm, trim={75 75 75 75}, clip]{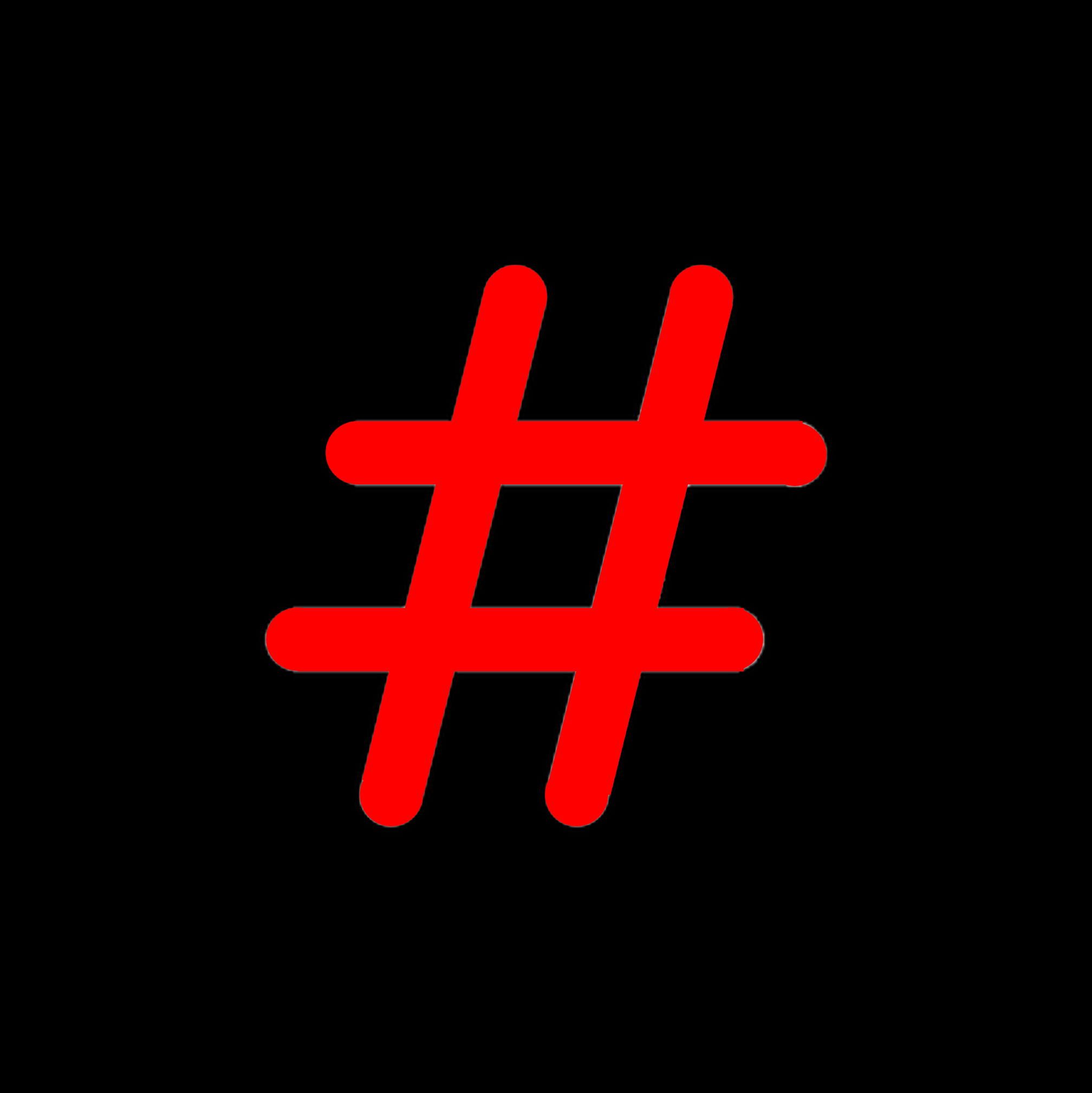}
        \caption{Output of a Connected Component algorithm in $\R^2$.}
        \label{fig:CC_R2}
    \end{subfigure}
    \hfill
    \begin{subfigure}[t]{0.3\textwidth}
        \centering
        \includegraphics[width=3.3cm, trim={75 75 75 75}, clip]{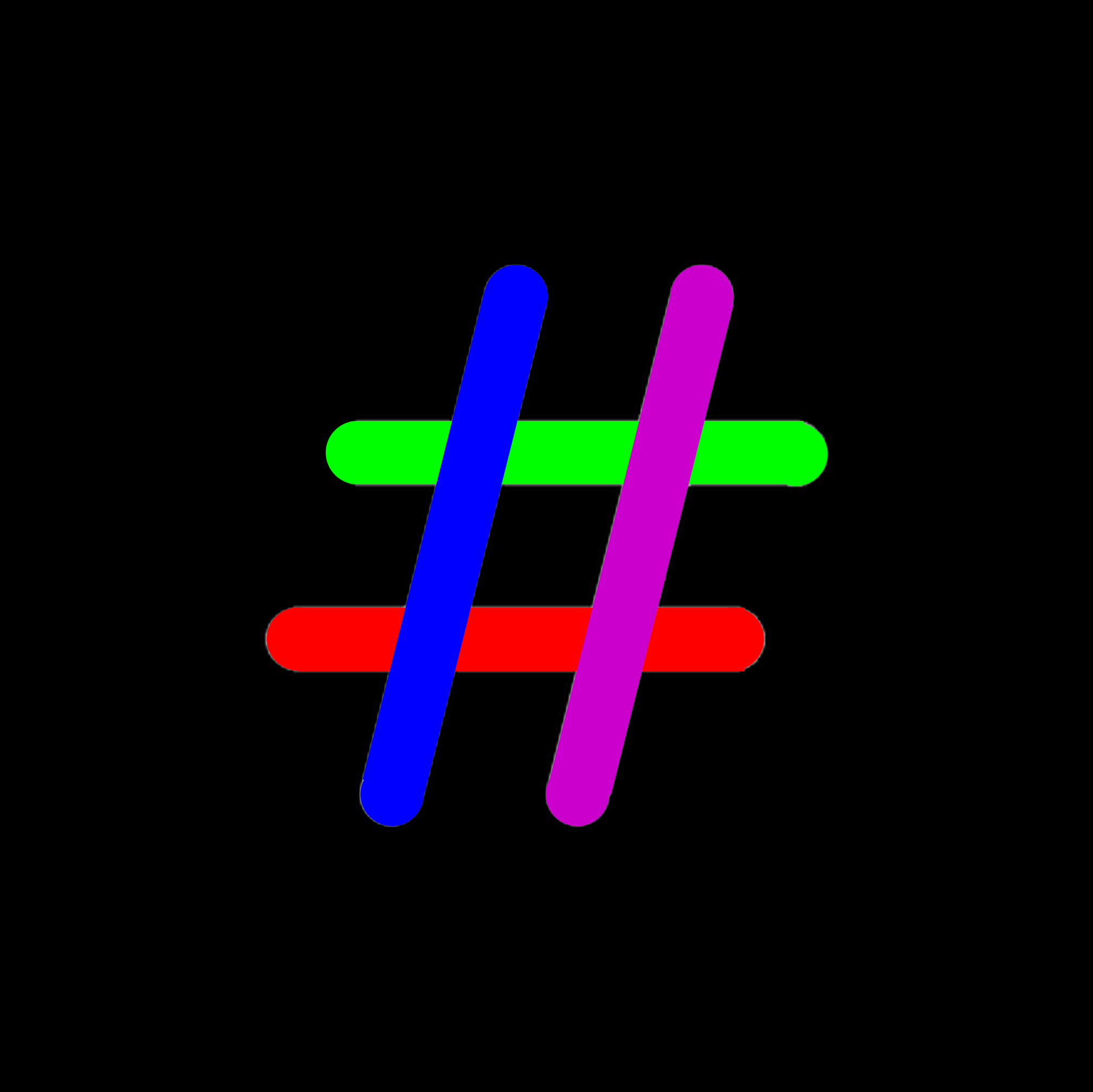}
        \caption{Output of a connected Component algorithm in $\R^2 \times S^1$ as proposed in \cite{berg2024connectedcomponentsliegroups}.}
        \label{fig:CC_M2}
    \hfill~
    \end{subfigure}
    \caption{Visualization of the advantages of the connected component algorithm from \cite{berg2024connectedcomponentsliegroups} in $\R^2\times S^1$. The central image illustrates that a standard connected component algorithm in  $\R^2$ failed to correctly group the four components, while the method of \cite{berg2024connectedcomponentsliegroups} successfully identified them, as shown by the four colors.}\label{fig: Example CC}
\end{figure*}

\begin{figure*}[t]
\centering
\hfill
    \begin{subfigure}[T]{0.45\textwidth}
        \centering
        \includegraphics[width=3.3cm, trim={75 75 75 75}, clip]{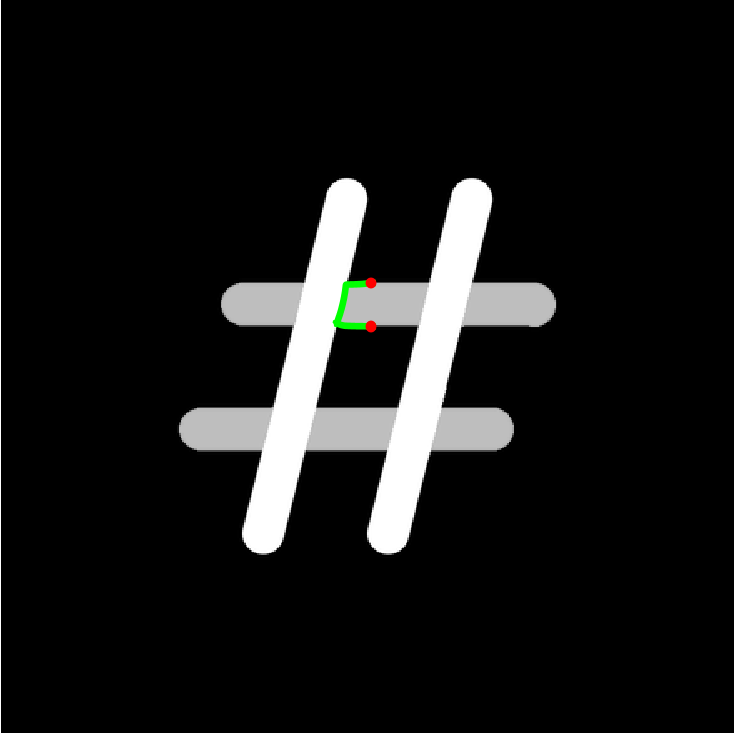}
        \caption{Output of the new geodesic tracking model with the standard cost function $\mathcal{C}$ which is \emph{not} informed by the connected component algorithm of \cite{berg2024connectedcomponentsliegroups}. The green line represents the minimizing geodesic between the two red dots.}
        \label{fig:tracking_not_informed}
    \end{subfigure}
    \hfill
    \begin{subfigure}[T]{0.45\textwidth}
        \centering
        \includegraphics[width=3.3cm, trim={75 75 75 75}, clip]{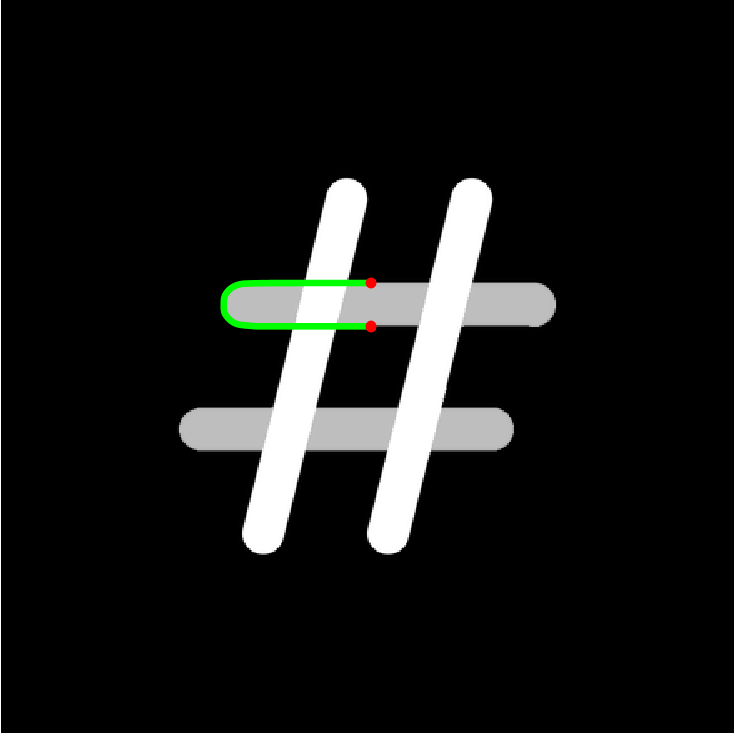}
        \caption{Output of the new geodesic tracking model with the cost function $\mathcal{C}_{\mathrm{new}}$ \eqref{CNEW} which is informed by the connected component algorithm of \cite{berg2024connectedcomponentsliegroups}. The green line represents the minimizing geodesic between the two red dots.}
        \label{fig:tracking_informed}
    \end{subfigure}
    \hfill~
    \caption{Visualization of the advantages of informing the cost function with the connected component algorithm of \cite{berg2024connectedcomponentsliegroups}. By informing the cost function, we reduce the risk of producing a geodesic that switches between edges of multiple components.}
    \label{fig: Example grouping of cost function}
    
\end{figure*}

We introduce a new model on $\R^2 \times P^1$ that yields cusp-free contours which we refer to as $d_{\mathrm{c}}$ (the $c$ referring to contours).
It works by connecting the 8 instances in \eqref{eq:8 pairs} using a different geodesic tracking model on $\R^2 \times S^1$ than in the model $d_{\mathrm{proj}}$ of \citet{BekkersGSI}.
The new model's formal introduction, and the intuition behind it, will follow in Section~\ref{ch:dc}.

The main differences between the  models $d_{\mathrm{c}}$ and $d_{\mathrm{proj}}$ on the projective line bundle $\R^2 \times P^1$ are:
\begin{enumerate}
    \item The model $d_{\mathrm{c}}$ is cusp-free, whereas the previous model $d_{\mathrm{proj}}$ may produce geodesics with cusps. 
    This is because the new model $d_{\mathrm{c}}$ relies on a cusp-free geodesic tracking model on $\R^2 \times S^1$, as proposed by \citet{duitsmeestersmirebeauportegies}. 
That model restricts the car to use only its forward gear, thereby eliminating cusps that occur when switching gears. 
    Figs.~\ref{fig:previous_tracking},~\ref{fig:previous_tracking2} compared to Figs.~\ref{fig:new_tracking},~\ref{fig:new_tracking2}, clearly illustrates this difference.
    \item Our model $d_{\mathrm{c}}$ relies on the asymmetric (Finslerian) geodesic tracking model of \cite{duitsmeestersmirebeauportegies} on $\R^2 \times S^1$, whereas the model $d_{\mathrm{proj}}$ relies on a symmetric (sub-Riemannian) model on $\R^2 \times S^1$. 
    Consequently, the 8 instances in \eqref{eq:8 pairs} cannot be reduced to the same 2 representative pairs as for $d_{\mathrm{proj}}$. 
    Instead, we will reduce $d_{\mathrm{c}}$ 
    %may be reduced 
    to 4 instances. 
    \item Another main advantage of our new model $d_{\mathrm{c}}$ over $d_{\mathrm{proj}}$ in \cite{BekkersGSI} is that it includes a connected-component-informed cost function. 
    The models $d_{\mathrm{c}}$ and $d_{\mathrm{proj}}$ become data-driven via a cost function $\mathcal{C}$.
    This cost function puts low costs on edges/lines present in the image and high costs everywhere else, and we use a standard choice 
    %as described in 
\cite{BekkersGSI,berg2024geodesic,duitsmeestersmirebeauportegies,berg2025crossing}.
    Details will follow in Section~\ref{Section: connected-component-informed cost function}.
    
    In the new model $d_{\mathrm{c}}$, we %want to 
    inform the cost function $\mathcal{C}$ with the connected component algorithm on $\R^2 \times S^1$ of \cite{berg2024connectedcomponentsliegroups}. 
   % By lifting images to $\R^2\times S^1$, \citet{berg2024connectedcomponentsliegroups} developed an algorithm that can group overlapping components. 
   That algorithm can correctly group overlapping structures, see Fig.~\ref{fig: Example CC}, and by including it in our tracking model (by replacing $\mathcal{C}$ with $\mathcal{C}_{\mathrm{new}}$, cf.~Section~\ref{Section: connected-component-informed cost function}) our model will track the correct connected component, see Fig.~\ref{fig: Example grouping of cost function}.  
\end{enumerate}

% In contrast to the model of \cite{duitsmeestersmirebeauportegies} on $\R^2\times S^1$, the new model $d_{\mathrm{c}}$ on $\R^2 \times P^1$ is symmetric. 

% \begin{remark}[Computation minimizing geodesics] \label{rem:2}
%     The minimizing geodesic(s) in $\R^2 \times S^1$ follows from backtracking on a distance map $d(\bp_0,\cdot)$ starting at $\bp_1$ by gradient descent, cf. \cite[Thm.4]{duitsmeestersmirebeauportegies}. Various methods exist for fast computation of the distance map  \cite{mirebeau_anisotropic_2014,mirebeau2019hamiltonian,berg2025crossing}. Details and our choice of algorithm will follow in Section~\ref{ch:track}. Examples of backtracking output in Medical_Images are visualized in Fig.~\ref{fig:medical_images}.
% \end{remark}

\begin{figure}[t]
    \centering
    \includegraphics[width=0.9\linewidth]{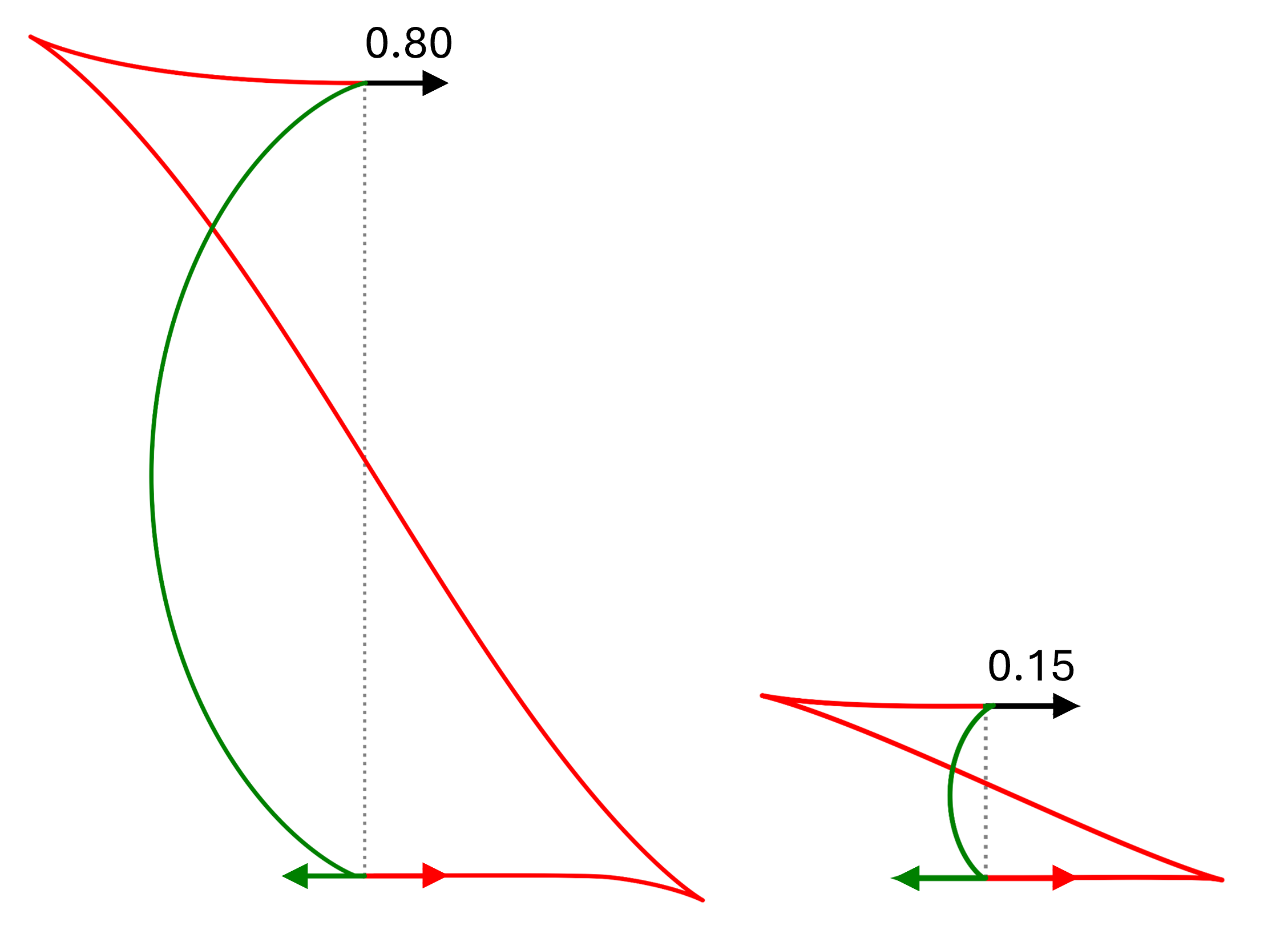}
    \caption{In red the projective line bundle minimizing (spatially projected) geodesic of $d_{\mathrm{proj}}([\bp_0],[\bp_1^\varepsilon])$ for $\bp_1^\varepsilon=((0,\varepsilon),(1,0))$ with $\varepsilon=0.80, 0.15$, \mbox{$\mathcal{C}=\xi=1$}. In green the cusp-free minimizing geodesic in the new model in $d_{\mathrm{c}}([\bp_0],[\bp_1^\varepsilon])$. When  the end condition (arrow in black) approaches the initial condition along the vertical axis (the dotted line) the red geodesic's length (\ref{Eq: finsler distance}) vanishes, whereas the green geodesic's length converges to $\pi$, see Prop.~\ref{corr:tobewritten}.}\label{fig:lat}
\end{figure} 

In this article, we advocate for the use of $d_{\mathrm{c}}$ instead of $d_{\mathrm{proj}}$ for our image analysis applications, and we mathematically analyze both the similarities and the differences of the two models. 

Later, in Theorem~\ref{th:main}, we will show that the new \emph{cusp-free} model $d_{\mathrm{c}}$ is not a distance in the strict sense, as it only satisfies the inequality on a large subset $\R^2 \times P^1$, where the models $d_{\mathrm{c}}$ and $d_{\mathrm{proj}}$ coincide.
%Nevertheless, as we show in the same theorem, it is equal to the projective line bundle distance $d_{\mathrm{proj}}$ on a large set $\tilde{\mathcal{Q}}_{\xi}$, which will be defined later. As $d_{\mathrm{proj}}$ satisfies the triangle inequality, we can deduce that $d_{\mathrm{c}}$ often satisfies the triangle inequality, as we will clarify in our Theorem~\ref{th:main}.
In practice, the key differences between $d_{\mathrm{c}}$ and $d_{\mathrm{proj}}$ arise when the end condition $[\bp_1]$ approaches the initial condition $[\bp_0]$ from the side, cf. Fig.~\ref{fig:lat}. 
There the green curves (new geodesics) are much more suitable for our applications than the red curves (previous geodesics). 
This will be formalized in Prop.~\ref{corr:tobewritten}. 
Although these green geodesics violate the triangle inequality, their behavior is desirable.
% Furthermore, we will explain how the Maxwell set simplifies and reduces under the new model (cf.~Figs. \ref{fig:Maxwellmon},  \ref{fig:Maxwellproj}).

The practical advantages of model $d_{\mathrm{c}}$ over model $d_{\mathrm{proj}}$ are qualitatively evident in Fig.~\ref{Fig: results ERIK}.
Several quantitative advantages are also demonstrated in Section~\ref{section: experiment}.

\subsection{Applications to SEM Images}
Scanning Electron Microscopy (SEM) has become a widely used tool in semiconductor metrology. Due to its high resolution, an electron microscope is a suitable tool for measuring, for example, the width and diameter of structures such as lines or holes in electronic devices on silicon wafers. 

In recent years, structures have shrunk in size, but have expanded vertically. These vertical structures, such as FinFETs (Fin Field-Effect Transistors) and gate-all-around transistors, are now key in the semiconductor industry \cite{timhouben}.
%(for an overview, see ~\cite{WikipediaMultigateDevice}), 
The commonality between these electronic devices is the overlapping structures (by inspection from above). 

Fig.~\ref{fig: SEM example} shows an SEM image of a FinFET, where the horizontal structures are located underneath the vertical structures. The image analysis task is to correctly segment such overlapping structures. 

Our segmentation algorithm is performed in the higher-dimensional space $\R^2 \times P^1$ to adequately deal with overlapping structures, and is visualized in the flowchart of Fig.~\ref{fig: flowchart}. 
The first two steps are a visualization of the connected component algorithm of \citet{berg2024connectedcomponentsliegroups}, upon which this algorithm is built. 
The first step shows the lifting of the data to $\R^2 \times P^1$, followed by a visualization of the output of the connected component algorithm in $\R^2 \times P^1$, where every component has a distinct color. 
The initial contour for our snake model can be derived from the connected component algorithm and is shown in Step 3 of Fig.~\ref{fig: flowchart}. Its placement will be refined in later steps. 

We have two algorithms for the refinement of the initial contour. For straight edges (blue in Step 3 of Fig.~\ref{fig: flowchart}), we use a standard edge detection operator in $\R^2$, as visualized in Step 4b. 
Our choice is the scale selection algorithm of \cite{lindeberg1998edge} combined with the edge-focusing algorithm \cite{Sjberg1988ExtractionOD,BartTeHaarRomenyFront} based on the norm of the Gaussian gradient. 

However, this efficient method fails on highly curved edges (red in Step 3 of Fig.~\ref{fig: flowchart}). 
For these parts, we use the newly proposed geodesic tracking model in $\R^2 \times P^1$. 
Since this space has a higher dimensionality than $\R^2$, computations in $\R^2 \times P^1$ are %inherently 
more %computationally 
demanding.
By integrating a switching criterion into the framework, computations in \mbox{$\R^2 \times P^1$} are performed only when necessary.

Step 4a of Fig.~\ref{fig: flowchart} visualizes the results of the new geodesic tracking model. 
The final contours are presented in the last step of Fig.~\ref{fig: flowchart}, where each component %has a distinct 
is color-coded.

\begin{figure}[t]
\centering
\includegraphics[width=0.5\linewidth
]{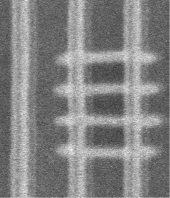}
\caption{SEM image with %containing 
overlapping %electronic 
structures. 
%This image depicts a processing step in the manufacturing of a FinFET.
}\label{fig: SEM example}
\end{figure}

\begin{figure*}[t]
    \centering
    \includegraphics[width=\linewidth]{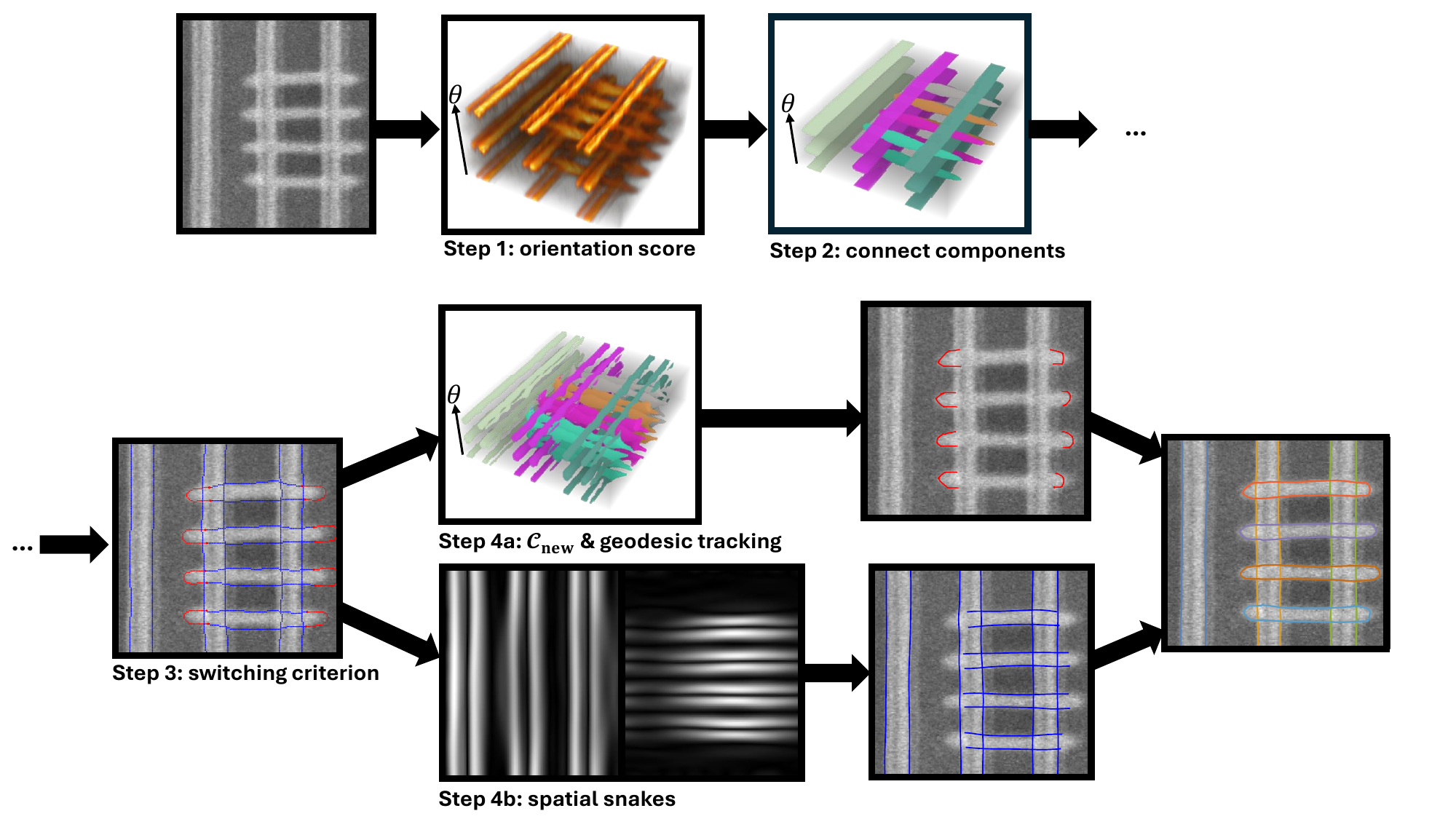}
    \caption{The flowchart of our algorithm applied to a SEM image (of a FinFET). The first two steps visualize 
    the connected component output in the lifted space of positions and orientations of the algorithm in \cite{berg2024connectedcomponentsliegroups}. Step 3 shows the switching criterion based on the curvature of the edge, after which the algorithm splits in two; An accurate algorithm (4a) and fast algorithm (4b). Step 4a visualized the connected-component-informed cost function $\mathcal{C}_{\mathrm{new}}$. Step 4b shows the scale selection (\cite{lindeberg1998edge}) and edge refinement method in the lateral direction. The results of both algorithms are combined in the final step yielding correctly segmented overlapping electronic structures in the SEM image.}\label{fig: flowchart}
\end{figure*}

\subsection{Contributions}
In this paper, we propose a new snake model on the projective line bundle to segment overlapping structures in SEM images. 
More specifically, our main contributions are categorized as follows: 
\begin{itemize}
\item 
\textbf{Theoretical: }
The geodesic tracking part of our model is, to the best of our knowledge, the first symmetric, cusp-free geometric model on the projective line bundle. We prove in Theorem~\ref{th:main} that it coincides with the previous model if the boundary conditions can be connected with a cusp-free geodesic.
We also visualize the important differences between the models if this is not the case (Fig.~\ref{fig:monprojoutcone}), and show the differences in the development of the Maxwell set (Figs. \ref{fig:Maxwellmon}, \ref{fig:Maxwellproj}). In Prop.~\ref{corr:tobewritten} we show that the new cusp-free model is globally controllable but not locally controllable.  
\item \textbf{Algorithmical:} We provide a computational model for snakes in the projective line bundle where we use a switching criterion between geodesic tracking and a simple edge detection model to avoid unnecessary computations.
It also includes a connected-component-informed cost function. Implementations of the full pipeline can be downloaded from \url{https://github.com/LeanneVis/Snakes-on-the-Projective-Line-Bundle}.
\item \textbf{Experimental: }
We show in Section~\ref{section: experiment} that our method performs well qualitatively (Figs.~\ref{fig:qualitative_results},~\ref{Fig: results ERIK}) and quantitatively (Figs.~\ref{Fig: MASD distance},~\ref{Fig: Hausdorff distance}), compared to the approach of \cite{BekkersGSI}.
Finally, we show the advantage of the connected component
informed cost function (Fig.~\ref{Fig: results ungrouped cost function}). 
Output examples of our method applied to various SEM images are shown in Fig.~\ref{fig:introexpgen}, where we see correct identification of the overlapping structures.

\begin{figure}[t]
    \centering
    \includegraphics[width=0.85\linewidth]{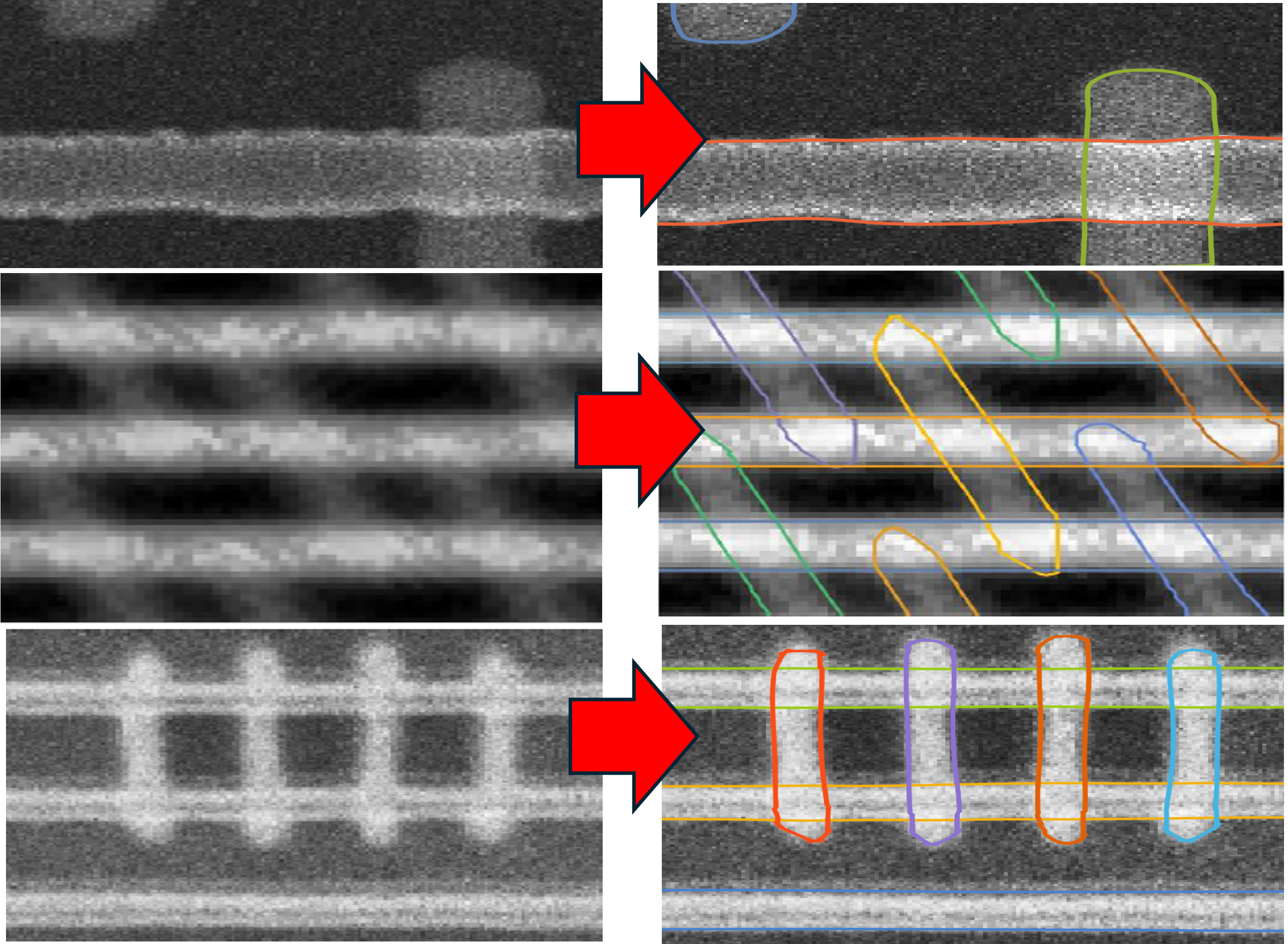}
    \caption{Output of our method on various SEM images. 
    For common parameters, see Section~\ref{section: experiment}.
\label{fig:introexpgen}}
\end{figure}
%and the switching criterion Tab.~\ref{} 
\end{itemize}
\begin{comment}
Regarding the experimental part, we show promising feasibility studies. 
As illustrated in Fig.~\ref{fig:introexpgen} our method produces good results on various SEM images from different data sets. 
An evaluation (with clear improvements to existing crossing-preserving tracking methods \cite{BekkersGSI} on $\R^2\times P^1$, and with clear improvements through the connected-component-informed cost function) is presented in Section~\ref{section: experiment}.
A wider scale comparison (also to geometric deep learning methods) is beyond the scope of this mathematically oriented article and left for a future more 
application-oriented follow-up article. 
\end{comment}

\subsection{Structure of the Article}
In Sec.~\ref{ch:newmodel} we formally  introduce the new metric model $d_{\mathrm{c}}$ on the projective line bundle $\R^2\times P^1$, and the previous metric model $d_{\mathrm{proj}}$.

In Sec.~\ref{section: preliminaries} we cover preliminary theory on the space of the 2D positions and orientations.
This theory supports
the subsequent analysis of $d_{\mathrm{c}}$ and the numerical computations of the geodesics.

In Sec.~\ref{sec:cuspfreemodel} we analyze, visualize and compare the new geodesic tracking model $d_{\mathrm{c}}$ with the previous model $d_{\mathrm{proj}}$ 
%as suggested in 
\cite{BekkersGSI}, and provide our main theoretical results (Thm.~\ref{th:main}, Prop.~\ref{corr:tobewritten}).
In Sec.~\ref{Section: connected-component-informed cost function} we provide the algorithm to inform the cost function $\mathcal{C}$ with the connected component algorithm of \cite{berg2024connectedcomponentsliegroups}, resulting in a connected-component-informed cost function $\mathcal{C}_{\mathrm{new}}$.
In Sec.~\ref{section: switching criterion} we provide a criterion for our snake model to switch between the new geodesic tracking model and a faster 2D edge detection. 

Then all established differential geometric tools are included in the overall snake algorithm on $\R^2 \times P^1$ (depicted in Fig.~\ref{fig: flowchart}). 
Experiments of Sec.~\ref{section: experiment}, show the advantages of the connected-component-informed cost function. They also show improved 
%qualitative and quantitative 
results of the task of contouring overlapping structures in SEM-images, compared to the approach of \citet{BekkersGSI}. 
Finally, we summarize and conclude in Sec.~\ref{section: conclusion}.

\section{The Model \texorpdfstring{\unboldmath $d_{\mathrm{c}}$}{dc} on the Projective Line Bundle \label{ch:newmodel}}
In this section, we introduce our new metric model on the projective line bundle $\R^2\times P^1$ which avoids cusps in spatial projections of geodesics. This model takes the minimum over multiple curve optimizations in $\R^2\times S^1$, relating to the 8 instances in Eq.~\eqref{eq:8 pairs}. In those curve optimizations, we will rely on an existing (cusp-free) model $d_{\mathcal{F}_0^+}$ on $\R^2 \times S^1$ \cite{duitsmeestersmirebeauportegies}.   

First, in Subsection~\ref{ch:cuspscontrol} we will give a formal definition of cusps (and control variables) and define the space of piecewise $C^1$-curves over which we optimize in $d_{\mathcal{F}_0^+}$.

Secondly, in Subsection~\ref{Section: Finslerian distance} we define $d_{\mathcal{F}_0^+}$. This metric is asymmetric, but does satisfy an intriguing antipodal symmetry, as we will prove in Lemma~\ref{Lemma: antipodal symmetry}.

Thirdly, in Subsection~\ref{ch:dc}
we introduce the model $d_{\mathrm{c}}$ on $\R^2\times P^1$ and show that it is symmetric, in contrast to $d_{\mathcal{F}_0^+}$. 
Here we use  Lemma~\ref{Lemma: antipodal symmetry} to minimize over four (rather than eight) instances.
In Subsection~\ref{ch:minimalcontours} we provide an intuitive viewpoint on $d_{\mathrm{c}}$, where optimization takes place over lifted contours rather than curves. 
This viewpoint will be formalized later in Section~\ref{sec:cuspfreemodel} (Prop.~\ref{Corollary: sign-semidefinite distance}).

Finally, in Subsection~\ref{sec:dproj} we point out some of the key differences between our proposed $d_{\mathrm{c}}$ model and the model proposed in \citet{BekkersGSI}.

\subsection{Curves, Controls and Cusps \label{ch:cuspscontrol}}

Throughout this article, we will restrict ourselves to certain curves in $\R^2 \times S^1$ for the curve optimization that connects two given distinct points $\bp_0,\bp_1 \in \R^2 \times S^1$ with a geodesic of minimal length. 

\begin{definition}[Candidate Curves]
    \label{def:Gamma}    
    Let $\Gamma(\bp_0,\bp_1)$ denote the set of piecewise $C^1$-curves $\gamma(\cdot):[0,1] \to \R^2 \times S^1$ with $\gamma(0)=\bp_0$, $\gamma(1)=\bp_1$.\\
    Within this set of curves we define the subsets
    \begin{align*}
        \Gamma_0(\bp_0,\bp_1) = \big\{
        &\gamma(\cdot)=(\bx(\cdot), \bn(\cdot)) \in \Gamma(\bp_0,\bp_1) \mid
        \\
        &\dot{\bx} \in \textrm{span}\{\bn\} \text{ a.e.} \big\}, 
        \\
        \Gamma_0^{+}(\bp_0,\bp_1) =
        \big\{&\gamma \in \Gamma_0(\bp_0,\bp_1) \mid
        \dot{\bx}\cdot \bn \geq 0
        \text{ a.e.}
        \big\},  
    \end{align*}
    with constraints on the spatial velocity $\dot{\bx}$ of  $\gamma$. 
    Curves in $\Gamma_0(\bp_0,\bp_1)$ are  called `horizontal curves'.
\end{definition}

\begin{definition}[Velocity Controls]\label{def:controlu1}
Let $\gamma \in \Gamma_0(\bp_0,\bp_1)$. Then there exist piecewise $C^1$ functions $u^1,u^3:[0,1] \to \R$ so that
\[
\dot{\gamma}=(\dot{\bx},\dot{\bn})=
\left(u^1 \bn, u^{3}\hat\bn\right),
\]
where $\bn(t)=(\cos \theta (t), \sin \theta(t))^T$ and $\hat\bn(t)=(-\sin \theta (t), \cos \theta(t))^T$. 
We refer to $u^1$ as the \emph{spatial control} and $u^3$ as the \emph{angular control}.
We can express these controls as follows
    \begin{equation} 
    \begin{array}{l}
        u^1(t)=\dot{\bx}(t) \cdot \bn(t), \\
        u^{3}(t)= \dot{\theta}(t), \textrm{ for all }t \in [0,1].
        \end{array}
    \end{equation}
\end{definition}
If $\gamma \in \Gamma^{+}_0(\ul{p}_0,\ul{p}_1)$ then the control $u^1(t)\geq 0$ along the curve. 
If $\gamma \in \Gamma_0(\ul{p}_0,\ul{p}_1)$ then the spatial control $u^1$ may have a zero-crossing and cusps arise in the spatial projection of the geodesic:
\begin{definition}[Cusp]\label{def:cusp}
    Let $\gamma \in \Gamma_0(\bp_0,\bp_1)$.
    A \emph{cusp} is given by a zero-crossing of the spatial control $u^1(t)$ at some time $t=t_0$: 
    \begin{equation} \label{eq:cusp}
        u^1(t_0) =0 \; \text{and} \; \dot{u}^1(t_0)\neq 0.
    \end{equation}    
\end{definition}

\subsection{\texorpdfstring{Distance Model \unboldmath $d_{\cF_0^+}$}{Cusp-Free Distance Model} on \texorpdfstring{\unboldmath $\R^2 \times S^1$}{R2 x S1}}\label{Section: Finslerian distance}
%\subsection{A Cusp-free Model \texorpdfstring{$d_{\mathrm{c}}$ on \unboldmath $\R^2 \times P^1$}{dc on R2 X P1}}

% If we compare a smooth path $t \mapsto \gamma(t)=(\bx(t),\bn(t))$ 
% with $\dot{\bx}/\|\dot{\bx}\|= \pm \bn$ 
% with the trajectory of a car, then the car has its spatial location $\bx(t)$ and its direction of motion $\bn(t)$. Intuitively, a cusp can be compared to changing the car's gears (i.e., switching from a forward to a backward motion or the other way around). For example, consider the car driving along the red curve in Fig.~\ref{fig:medical_images}a it switches from forward to backward gear and back. The mathematical definition of a cusp is given as follows.

The cusp-free model on $\R^2 \times S^1$, the so-called `Reeds-Shepp car without reverse gear', as proposed in \citet{duitsmeestersmirebeauportegies}, is an asymmetric Finslerian version of the sub-Riemannian distance model \cite{citti2006cortical,Petitot2017Neurogeometry,duits2014association,bellaard2023analysis}. It imposes a positive spatial control $u^1\geq 0$ on the set of curves ($\Gamma_0^+$ instead of $\Gamma_0$) over which we optimize. There, one minimizes on $\R^2 \times S^1$:
\begin{definition} [Finsler distance $d_{\cF_0^+}$]
\label{def:ASFD}
The \emph{asymmetric} Finsler distance \cite{duitsmeestersmirebeauportegies} is denoted by $d_{\cF_0^+}$ and equals:
\begin{equation} \label{Eq: finsler distance forward}
    \!d_{\cF^{+}_0}(\bp_0,\bp_1) = \!\!\inf_{\gamma \in \Gamma_0^+(\bp_0, \bp_1)}
    \int \limits_0^1\cF_0(\gamma(t),\dot{\gamma}(t)) \mathrm{d}t,
\end{equation}
with the set of curves $\Gamma^+_0(\bp_0, \bp_1)$ defined in Def.~\ref{def:Gamma}. 
The corresponding Finsler function $\cF_0$ is given by 
\begin{equation}\label{Eq: finsler function}
\begin{array}{l}
   \cF_0(\gamma(t),\dot{\gamma}(t)) \\
   \ \ = \mathcal{C}(\gamma(t))\sqrt{\xi^2|u^1(t)|^2 + |u^{3}(t)|^2}, 
    \end{array}
\end{equation}
with an a priori given continuously differentiable cost function $\mathcal{C}(\cdot)\geq \delta>0$ with $\xi>0$. 
\end{definition}

The parameter $\xi>0$ of \eqref{Eq: finsler function} regulates the \emph{bending stiffness} of the minimizing geodesics, i.e. the curve(s) $\gamma$ minimizing \eqref{Eq: finsler distance forward}.

The minimizing geodesics follow from backtracking on the distance map $d_{\cF_0^+}(\bp_0,\cdot)$ starting at $\bp_1$ by gradient descent, cf. \cite[Thm.4]{duitsmeestersmirebeauportegies}. Various methods exist for fast computation of the distance map  \cite{mirebeau_anisotropic_2014,mirebeau2019hamiltonian,berg2025crossing}. Details and our choice of algorithm will follow in Section~\ref{ch:track}. Examples of backtracking output in medical images are visualized in Fig.~\ref{fig:medical_images}.
% \begin{comment}
% and with  velocity
% given by
% \[
% \dot{\gamma}=(\dot{\bx},\dot{\bn})=
% \left(u^1 \bn, u^{3}\frac{\dot{\bn}}{\|\dot{\bn}\|}\right),
% \]
% where $\bn(t)=(\cos \theta (t), \sin \theta(t))$, and where the spatial velocity control $t \mapsto u^1(t)$ is given by \eqref{controlu1} and with angular velocity control is defined by $t \mapsto u^{3}(t)=\dot{\theta}(t)$.
% \end{comment}

\begin{remark}[Choice of cost function]
We rely either on the cost function $\mathcal{C}\neq 1$, given in \cite{BekkersSIAM,duitsmeestersmirebeauportegies,BekkersGSI} that we provide in Section~\ref{Section: connected-component-informed cost function} in Eq.~(\ref{Eq:CostFunction}), or a connected-component-informed cost function $\mathcal{C}_{\mathrm{new}}$, that we provide in Section~\ref{Section: connected-component-informed cost function} in Eq.~(\ref{CNEW}).
\end{remark}

\begin{definition}[Sub-Riemannian distance $d_{\cF_0}$]\label{Def: sub-Riemannian distance}
    The standard \emph{symmetric} sub-Riemannian distance \cite{BekkersSIAM,BoscainESAIM,citti2006cortical,Petitot2017Neurogeometry,bellaard2023analysis}
        is denoted by $d_{\cF_0}$ and equals:
    \begin{equation} \label{Eq: finsler distance}
    \begin{array}{rl}
    d_{\cF_0}(\bp_0,\bp_1) &= \inf \limits_{\gamma\in\Gamma_0(\bp_0, \bp_1)}
    L(\gamma)\ ,
    \end{array}
    \end{equation}
    with length $L(\gamma) =\int_0^1\cF_0(\gamma(t),\dot{\gamma}(t)) \mathrm{d}t$.
\end{definition}
%A quick solution to regain symmetry is to take the minimum over $d_{\cF_0^+}(\bp_0,\bp_1)$ and $d_{\cF_0^+}(\bp_1,\bp_0)$, but this will have serious consequences as we will see later.

\begin{remark}[Symmetry $d_{\cF_0}$ and asymmetry $d_{\cF_0^+}$]  
    \label{rem:as} 
    The distance $d_{\cF_0}$ is symmetric because of the following property: if $\gamma \in \Gamma_0(\bp_0,\bp_1)$ then $\gamma^\text{rev}(t)=\gamma(1-t) \in \Gamma_0(\bp_1,\bp_0)$, 
    with equal length $L(\gamma)=L(\gamma^\text{rev})$. 
    The distance $d_{\cF_0^+}$ is clearly asymmetric because of the positive control constraint in $\Gamma_0^+$. Consider for example the points $\bp_0=((0,0),(1,0))$ and $\bp_1=((1,0),(1,0))$). 
    Indeed the symmetry argument for $d_{\mathcal{F}_0}$ %(reversing the curve) 
    no longer applies to model  $d_{\mathcal{F}_0^+}$.
    % $d_{\cF^{+}_0}$ is \emph{not} %necessarily symmetric as $\Gamma_0^+$ does not have this property     
    %Indeed, it is not difficult to find examples where $d_{\cF^{+}_0}(\bp_0,\bp_1) \neq d_{\cF^{+}_0}(\bp_1,\bp_0)$.
\end{remark} 

Despite its  asymmetry, the model $d_{\mathcal{F}_0^+}$ does admit an antipodal symmetry:

%Once the model $d_{\cF_0}$ and $d_{\cF_0^+}$ is set the minimal curves are computed as follows.
%including the Finsler function $\cF_0$ as given in Eq.~{\eqref{Eq: finsler function}}.

%distances have %an antipodal %symmetry %property.
\begin{lemma}[Antipodal symmetry]\label{Lemma: antipodal symmetry}
Let the cost function $\mathcal{C}:\R^2 \times S^1 \to [\delta,\infty)$ be well-defined on $\R^2\times P^1$, i.e. $\mathcal{C}(\bp)=\mathcal{C}(\overline{\bp})$ for all $\bp \in \R^2 \times S^1$.
Then one has (antipodal) symmetries: 
%for $d_{\cF_0}$ and $d_{\cF_0^+}$:
    \begin{equation*}        
    \begin{array}{ll}
        d_{\cF_0^+}(\bp_0,\bp_1) &=d_{\cF_0^+}(\overline{\bp}_1,\overline{\bp}_0), \\[5pt] d_{\cF_0}(\bp_0,\bp_1) &=d_{\cF_0}(\overline{\bp}_1,\overline{\bp}_0),     
    \end{array}
    \end{equation*}
    for all $\bp_0,\bp_1 \in \R^2 \times S^1$. 
\end{lemma}
\begin{proof}
See Appendix~\ref{app:MCP}.
%Essentially, one both reverses and takes the antipodal-points over the curves over which is optimized.  
\end{proof}

\subsection{New Cusp-Free and Symmetric Model \texorpdfstring{\unboldmath $d_{\mathrm{c}}$}{dc} on \texorpdfstring{\unboldmath $\R^2\times P^1$}{R2 x P1}}\label{ch:dc}

We propose the following symmetric cusp-free model \emph{on the projective line bundle} $\R^2\times P^1$. We do this
by selecting the minimum distance $d_{\cF_0^+}$ over the equivalence classes:

\begin{equation}\label{dmon}
% \label{Eq:dc_newnotation}
\boxed{
d_{\mathrm{c}}([\bp_0], [\bp_1]) := 
\min_{\substack{
    \bq_0 \in [\bp_0]  \\
    \bq_1 \in [\bp_1] 
}} 
\{d_{\cF^{+}_0}(\bq_0, \bq_1)\}.}
%d_{\cF^{+}_0}(\bq_1, \bq_0)\}.
\end{equation}
The new model $d_{\mathrm{c}}$ is well-defined on $\R^2\times P^1$. 
% The model $d_{\mathrm{c}}$ includes the new cost function $\mathcal{C}_{\mathrm{new}}$ instead of $\mathcal{C}$ in (\ref{Eq: finsler function}) that we will explain in detail later in Section~\ref{Section: connected-component-informed cost function}.

The model $d_{\mathrm{c}}$ is \emph{cusp-free} (i.e. spatial projections of geodesics will not exhibit cusps) as it selects the minimizer over 4 cusp-free geodesics. 
Recall that we constrain $u^1\geq 0$ in model $d_{\mathcal{F}_0^+}$ via $\Gamma_0^+$ (Def.~\ref{def:Gamma}), ensuring that cusps (\ref{eq:cusp}) cannot occur. 

The model $d_{\mathrm{c}}$ is \emph{symmetric}. 
The symmetry is not straightforward as $d_{\mathcal{F}_0^+}$ is asymmetric.  
To this end, we note that 
in \eqref{dmon} one could have taken the minimum over 
$\{d_{\cF^{+}_0}(\bq_0, \bq_1),d_{\cF^{+}_0}(\bq_1, \bq_0)\}$, covering all 8 instances of \eqref{eq:8 pairs}. 
The distance $d_{\cF_0^+}$ is asymmetric, but it does satisfy $d_{\cF_0^+}(\bq_0,\bq_1)=d_{\cF_0^+}(\overline{\bq}_1,\overline{\bq}_0)$ by Lemma~\ref{Lemma: antipodal symmetry}. 
Thus, $d_{\mathrm{c}}$ is symmetric, and the more efficient definition \eqref{dmon}, minimizing over only 4 instances, suffices.

In the formal analysis of $d_{\mathrm{c}}$ in Section~\ref{sec:cuspfreemodel}, we will see that the minimization over the 4 instances has major impact on properties of the symmetric model $d_{\mathrm{c}}$ (e.g. development spheres, Maxwell sets etc.) which are clearly different from the 
properties of the previous model $d_{\mathrm{proj}}$ \cite{BekkersGSI}, and very different from the asymmetric model $d_{\mathcal{F}_0^+}$ \cite[Fig. 9C]{duitsmeestersmirebeauportegies}. 
%More concretely one has:sub %\cite{du*its2014association,duitsmeestersmirebeauportegies}.
%\begin{equation}\label{dmon}
%    \begin{array}{l}
%        d_{\mathrm{c}}([\bp_0], [\bp_1]) = 
%       \min\left\{ 
%        d_{\cF^+_0}(\bp_0, \bp_1),\right.\\
%        \left.
%        d_{\cF^+_0}(\overline{\bp}_0, \bp_1),
%        d_{\cF^+_0}(\bp_0, \overline{\bp}_1),
%        d_{\cF^+_0}(\overline{\bp}_0, %\overline{\bp}_1) 
%        \right\}. 
%    \end{array}
%\end{equation}

\subsubsection{Minimal Lifted Contours}\label{ch:minimalcontours}
% The problem $d_{\cF_0^+}$ \cite{duits2014association,duitsmeestersmirebeauportegies} aims to find a minimal \emph{curve} on $\R^2\times S^1$.
In model $d_{\mathrm{c}}$ we find a geodesic $\gamma$ in $\R^2\times S^1$ (connecting the minimal pair in \eqref{eq:8 pairs}). 
We will prove in Prop.~\ref{Corollary: sign-semidefinite distance} (Section~\ref{sec:cuspfreemodel}) that this curve has a spatial control $u^1$ given by Def.~\ref{def:controlu1}, which is non-negative or non-positive everywhere. 
Then $\gamma=(\ul{x},\ul{n})$ with $\dot{\bx} \in \mathrm{span}\{\bn\}$ is \emph{either} the lifted curve of its spatial projection $\ul{x}(\cdot)$ \emph{or} the lifted curve of its time reversed spatial projection. We call such a curve $\gamma$ a \emph{lifted contour}, because of this directional property.
The c in $d_{\mathrm{c}}$ stands for `contour'.  

Thus, intuitively the model $d_{\mathrm{c}}$ selects the lifted contour with minimal length. 
Later we will provide a formal mathematical justification of this intuition %(and the symmetry of $d_{\mathrm{c}}$) 
in the analysis of Section~\ref{sec:cuspfreemodel} (in Prop.~\ref{Corollary: sign-semidefinite distance}).

\subsection{Model \texorpdfstring{\unboldmath $d_{\mathrm{c}}$}{dc} vs. Model \texorpdfstring{\unboldmath $d_{\mathrm{proj}}$}{dproj}}\label{sec:dproj}

The cusp free model $d_{\mathrm{c}}$ in \eqref{dmon} 
differs significantly from the model $d_{\mathrm{proj}}$ suggested in \citet{BekkersGSI} on %the projective line bundle 
$\R^2 \times P^1$, which is based on the distance

\begin{equation}\label{Eq:dproj_newnotation}
d_{\mathrm{proj}}([\bp_0], [\bp_1]) := 
\min_{\substack{
    \bq_0 \in [\bp_0]  \\
    \bq_1 \in [\bp_1] 
}} 
\{d_{\cF_0}(\bq_0, \bq_1)\},
%,d_{\cF_0}(\bq_1, \bq_0)\},
\end{equation}
where cusps can still arise, cf.~Figs.~\ref{fig:medical_images}-\ref{fig:lat}. 
%Recall that one has to take the minimum over eight instances \eqref{eq:8 pairs}. %Due to the antipodal symmetry in $d_{\cF^+_0}$, $d_{\mathrm{c}}$ %\eqref{Eq:dc_newnotation} can be reduced to 4 instances, as will follow by %Lemma~\ref{Lemma: antipodal symmetry}:
%\begin{equation}\label{dmon}
%    \begin{array}{l}
%        d_{\mathrm{c}}([\bp_0], [\bp_1]) := 
%        \min\left\{ 
%        d_{\cF^+_0}(\bp_0, \bp_1),\right.\\
%        \left.
%        d_{\cF^+_0}(\overline{\bp}_0, \bp_1),
%        d_{\cF^+_0}(\bp_0, \overline{\bp}_1),
%        d_{\cF^+_0}(\overline{\bp}_0, \overline{\bp}_1) 
%        \right\}. 
%    \end{array}
%\end{equation}
Since $d_{\cF_0}(\bq_0,\bq_1)=d_{\cF_0}(\bq_1,\bq_0)
%$ and $d_{\cF_0}%(\bq_0,\bq_1)
=d_{\cF_0}(\overline{\bq}_1,\overline{\bq}_0)$, as shown in Lemma~\ref{Lemma: antipodal symmetry}, the distance model 
$d_{\mathrm{proj}}$  can be expressed by taking a minimum over 2 instances:
% , as shown in~\cite[Eq.~4]{BekkersGSI}
\begin{equation}
\label{dproj}
\boxed{
\begin{array}{l}
d_{\mathrm{proj}}([\bp_0],[\bp_1]]) = \\[5pt]
\min\left\{d_{\cF_0}(\bp_0, \bp_1), d_{\cF_0}(
\bp_0, \overline{\bp}_1)\right\}
\end{array}
}
\end{equation} 
This is in contrast to $d_{\mathrm{c}}$ which uses the asymmetric distance $d_{\cF_0^+}$, and which required \emph{4 instances} (\ref{dmon}). 

Further analysis of the previous model $d_{\mathrm{proj}}$ and the new model $d_{\mathrm{c}}$, and computation of their geodesics, requires some mathematical background that we cover next.

\begin{comment}
\subsection{Model $d_{\mathrm{c}}$ yields Minimal Lifted Contours}\label{ch:minimalcontours}
% The problem $d_{\cF_0^+}$ \cite{duits2014association,duitsmeestersmirebeauportegies} aims to find a minimal \emph{curve} on $\R^2\times S^1$.
In model $d_{\mathrm{c}}$, which is well defined on $\R^2\times P^1$, we find a geodesic $\gamma$ in $\R^2\times S^1$ (connecting the minimal pair \eqref{eq:8 pairs}). 
We will prove in Prop.~\ref{Corollary: sign-semidefinite distance} (Section~\ref{sec:cuspfreemodel}) that this curve has a spatial control $u^1$ given by Def.~\ref{def:controlu1}, which is non-negative or non-positive everywhere. 
Then $\gamma=(\ul{x},\ul{n})$ is \emph{either} the lifted curve of its spatial projection $\gamma=(\ul{x}, \ul{n})$ with $\bx \in \mathrm{span}\{\bn\}$ \emph{or} the lifted curve of its time reversed spatial projection. We call such a curve $\gamma$ a \emph{lifted contour}, because of this directional property. 
The c in $d_{\mathrm{c}}$ stands for `contour'.  
Thus, intuitively the model $d_{\mathrm{c}}$ selects the lifted contour with minimal length. 

Later we will provide a formal justification of this intuition %(and the symmetry of $d_{\mathrm{c}}$) 
in the analysis of Section~\ref{sec:cuspfreemodel} (in Prop.~\ref{Corollary: sign-semidefinite distance}).
\end{comment}

\section{Mathematical Background}\label{section: preliminaries}
In this section, we address some general mathematical background needed for:
\begin{itemize}
\item All steps in Fig.~\ref{fig: flowchart} (Step 1: lifting, Step 2: connected components, Step 3:  switch criterion, Step 4a: geodesic tracking, Step 4b: spatial snakes).
\item Geometric foundation of new model $d_{\mathrm{c}}$ and previous model $d_{\mathrm{proj}}$ and their geodesics, using sub-Riemannian geometrical tools on \mbox{$\R^2 \times S^1$}.
\end{itemize}
We start off by introducing the space of positions and orientations $\R^2 \times S^1$, followed by the Lie group of roto-translations $\mathrm{SE}(2)$, which acts on $\R^2 \times S^1$.
In Section \ref{Sec: orientation score processing}, we elaborate on how to lift the data with the orientation score transform to $\R^2 \times S^1$. Lifting an image to this higher-dimensional space $\R^2 \times S^1$, allows one to appropriately deal with complex structures, such as crossing lines. 

The two subsequent sections explain (sub)-Riemannian and Finslerian distances, which are included in various steps of the snake model. 
Finally, we introduce the geodesic tracking algorithm in $\R^2 \times S^1$. 
% Finally, the concept of morphological dilation is explained, which is used for the grouping of the cost function. 

\subsection{Lifted Space of Positions and Orientations \texorpdfstring{\unboldmath $\R^2 \times S^1$}{R2 x S1}}

Let $\R^2 \times S^1$ be the homogeneous space of positions and orientations on which the Lie group $\SE(2)$ acts transitively.
Elements of the homogeneous space $\R^2 \times S^1$ will be denoted by $(\bx,\bn) \in \R^2 \times S^1$, where $\bx$ represents the position in $\R^2$ and $\bn=(\cos \theta, \sin \theta)$ represents the orientation in $S^1$. Let $\bp_0 = (\mathbf{0},(1,0))$ be our reference element in $\R^2 \times S^1$.

Let $\SE(2) = \R^2 \rtimes \SO(2)$ be the Lie group of roto-translations. 
The group elements $g$ of $\SE(2)$ are given by $(\bx,R_\theta)\in \R^2 \rtimes \SO(2)$, where $\bx \in \R^2 $ is the translation vector and where $R_\theta \in \SO(2)$ is the rotation matrix with the (counterclockwise) angle $\theta$. 
We will use the concise notation $(x,y,\theta)$ for $g$, due to the isomorphism $\SO(2) \cong \R/(2\pi\mathbb{Z})$.

The group product on $\SE(2)$ is defined by:
\begin{align*}
    g_1 g_2:= (\bx_1+R_{\theta_1}\bx_2,\ \theta_1 + \theta_2 \mod 2 \pi)
\end{align*}
for $g_1=(\bx_1,R_{\theta_1})$ and $g_2=(\bx_2,R_{\theta_2})\in \SE(2)$. The identity element is $e=(\ul{0},R_0)\in\SE(2)$.

For each $g = (\by,R_\phi) \in \SE(2)$ the left transitive action of roto-translations $L_g: \R^2 \times S^1 \rightarrow \R^2 \times S^1$ on the homogeneous space $\R^2 \times S^1$ is given by
\begin{equation} \label{groupaction}
L_g(\bx,\bn) := g \cdot (\bx,\bn) = (\by+R_\phi \bx, R_\phi\bn),
\end{equation}
with $(\bx,\bn) \in \R^2 \times S^1$.
The bijection between a position-orientation and a roto-translation becomes apparent if we apply a roto-translation $g=(\bx,R_\theta)$ to reference element $\bp_0$: 
\begin{align*}
L_g(\mathbf{0},(1,0))&:= (\bx,R_\theta)\cdot(\mathbf{0},(1,0)) \\
&= (\bx,(\cos \theta, \sin \theta)).
\end{align*}
Henceforth, we will identify $\R^2 \times S^1$ with $\SE(2)$ and we can use the same concise notation for element of $\R^2 \times S^1$, namely $(\bx,\bn) \equiv (x,y,\theta)$. 

The roto-translation group $\SE(2)$ also acts on functions on $\R^2$ via the representation $\mathcal{U}: \SE(2) \rightarrow B(\mathbb{L}_2(\R^2))$ given by
\begin{equation*}
(\mathcal{U}_g f)(\by) := f\left(R_\theta^{-1}(\by - \bx)\right).
\end{equation*} 

Similarly, the group $\SE(2)$ acts on functions \mbox{$U:\R^2 \times S^1 \to \mathbb{C}$} via the representation $\mathcal{L}: \SE(2) \rightarrow B(\mathbb{L}_2(\R^2 \times S^1))$ given by
\begin{equation} \label{eq:calL_representation}
\mathcal{L}_gU(h) := U(g^{-1} h), 
\end{equation} 
for all $h,g \in G$.
% This is a shift-twist transformation.

\subsection{Orientation Score Processing}\label{Sec: orientation score processing}
This section explains how the data is lifted to the space $\R^2\times S^1$ using the \emph{orientation score transform}, as visualized in Step 1 of Fig.~\ref{fig: flowchart}.

\begin{definition}[Orientation Score Transform]
    \label{def: orientation score}
     We define the invertible orientation score transform $W_\psi: \mathbb{L}_2(\R^2) \rightarrow \mathbb{L}_2( \R^2\times S^1)$, with an anisotropic wavelet $\psi: \mathbb{L}_1( \R^2) \cap \mathbb{L}_2( \R^2) $, by
    \begin{align}\label{eq:OS}
        W_\psi f(\bx,\theta) &:= \int_{\R^2} \overline{\psi\left(R_\theta^{-1}(\by-\bx)\right)}f(\by)\mathrm{d}\by 
    \end{align}
    We will use the shorthand notation $U_f := W_\psi f$.
\end{definition}

The orientation score is thus constructed by taking the convolution with a wavelet $\psi$ rotated counterclockwise over all orientations $\theta \in \R/(2\pi\mathbb{Z})$.
For the wavelet $\psi$ we used the cake wavelet, as proposed in \cite[Part I, Fig.~3]{DuitsAMS}.

\begin{remark}[Motivation usage of cake wavelets $\psi$]
    Under conditions on $\psi$ the orientation score transform $W_{\psi}$ is invertible, i.e. the image $f$ can be reconstructed in a stable way from the orientation score $W_{\psi} f$ \cite[Thm.~4]{Duits2007a}, \cite[Part I, App.~A]{DuitsAMS}. 
    Furthermore, $f$ can also be recovered approximately simply by summing or integrating over all orientations \cite[Eq.~21]{Duits2007a}. 
    Finally, cake wavelets produce a vanishing minimal %$SE(2)$-
    uncertainty gap \cite[Thm.~1]{Sherry2025}.
\end{remark}

The orientation score transform $W_\psi$ is \emph{equivariant}, which means that: 
\begin{equation} \label{intertwine}
W_\psi \circ \mathcal{U}_g = \mathcal{L}_g \circ W_\psi 
\end{equation} 
for all $g \in \SE(2)$. Thereby orientation score processing must be left-invariant (i.e. must commute with $\mathcal{L}$) for Euclidean equivariant image processing, see \cite[Thm.21]{DuitsRthesis}.

Such processing requires key ingredients such as
left-invariant %(Sub-)Riemannian 
metric tensor fields and distance models that we explain next.
% Equivariance is a desirable property: if we act on the input of an equivariant operator, its output is transformed accordingly. 
% In this example, roto-translation of the input image $f$, 
% yields roto-translation of the output orientation score $U_f$ by a shift-twist transformation.

% It can be shown that the operator in the image domain is left-invariant if and only if the corresponding operator on the invertible orientation score commutes with $\mathcal{L}$ \cite[Thm.21]{DuitsRthesis}. 

\subsection{Left-Invariant %(Sub-)Riemannian 
Metric Tensor Fields and Distance Models}

Next, we introduce a left-invariant frame, metric tensor field, and distance on $\R^2 \times S^1$ to ensure equivariant processing on the orientation score. 

\begin{definition}[Left-invariant vector fields]
    \label{def:left_invariant_vector_fields}
    %Let $\mathcal{A} \in \sections(T(\R^2 \times S^1))$, with $\sections(T(\R^2 \times S^1))$ the smooth sections of the tangent bundle $T(\R^2 \times S^1)$, be a vector field. 
    Vector field $\mathcal{A}$ on $\R^2 \times S^1$ is called \emph{left-invariant} if 
    \begin{equation*}\label{eq:left_invariant_vector_field}
    (L_g)_* \mathcal{A} \vert_{\bp} = \mathcal{A} \vert_{g \bp}
    \end{equation*}
    for all $g \in \SE(2)$ and $\bp \in \R^2 \times S^1$. The push-forward $(L_g)_*$ of $L_g$ given by
    \begin{equation*}\label{eq:push_forward}
    ((L_g)_* \mathcal{A} \vert_{\bp}) (U) := \mathcal{A} \vert_{\bp} (U \circ L_g), 
    \end{equation*}
    for all smooth functions $U: \R^2 \times S^1 \rightarrow \mathbb{C}$.
\end{definition}

\begin{comment}
One typically constructs left-invariant vector fields by pushing forward tangent vectors at the reference point $\bp_0$, e.g.
for any $A \in T_{\bp_0}(\R^2 \times S^1)$ we can define a left-invariant vector field %$\mathcal{A} \in \sections(T(\R^2 \times S^1))$ 
by
\begin{equation*}
\mathcal{A} \vert_{g \bp_0} := (L_g)_* A ,
\end{equation*}
with $g \in \SE(2)$. In particular, 
\end{comment}

We can construct the \emph{left-invariant frame} $\{\mathcal{A}_1, \mathcal{A}_2, \mathcal{A}_3\}$ by pushing forward the static coordinate frame $\{\partial_x, \partial_y, \partial_\theta\}$ at  reference point $\bp_0$:
\begin{equation}\label{eq:left_invariant_frame}
\begin{split}
    \mathcal{A}_1 \vert_{\bp}  &= \cos\theta\, \partial_x \vert_{\bp} + \sin\theta\,  \partial_y \vert_{\bp},\\ 
    \mathcal{A}_2 \vert_{\bp} &= -\sin\theta\,  \partial_x \vert_{\bp} + \cos\theta \, \partial_y \vert_{\bp},  \; \text{and} \\
    \mathcal{A}_3 \vert_{\bp} &= \partial_\theta \vert_{\bp},
\end{split}
\end{equation}
with $\bp = (x, y, \theta)$.
Accordingly, its dual frame $\{\omega^1, \omega^2, \omega^3\}$, defined by $\omega^i(\mathcal{A}_j)=\delta^i_j$, equals: 
\begin{align*}
    \omega^1 \vert_{\bp}  &= \cos\theta\,  \mathrm{d}x \vert_{\bp} + \sin \theta\, \mathrm{d}y \vert_{\bp},\\ 
    \omega^2 \vert_{\bp} &= -\sin \theta\,  \mathrm{d}x \vert_{\bp} + \cos \theta\, \mathrm{d}y \vert_{\bp},  \; \text{and} \\
    \omega^3 \vert_{\bp} &= \mathrm{d}\theta \vert_{\bp}.
\end{align*}

\begin{definition}[Left-invariant metric tensor field]
    Metric tensor field $\mathcal{G}$ on $\R^2 \times S^1$ is left-invariant iff 
    %and only if 
    \begin{equation*}    \mathcal{G}_{g\bp}\left((L_g)_*\dot{\bp},(L_g)_*\dot{\bp}\right) = \mathcal{G}_{\bp}\left(\dot{\bp},\dot{\bp}\right) 
    \end{equation*}
    for all $\bp\in\R^2 \times S^1, \dot{\bp} \in T_{\bp}(\R^2 \times S^1)$ and $g\in \SE(2)$.
\end{definition}

Let $\mathcal{G}$ be a left-invariant metric tensor field on $\R^2 \times S^1$. Then there exists a unique constant matrix $[g_{ij}] \in \R^{3 \times 3}$ such that
\begin{equation}\label{Eq: sub-Riemannian metric tensor}
    \mathcal{G} = \sum_{i,j=1}^3g_{ij}\; \omega^i \otimes \omega^j.
\end{equation}
We will restrict ourselves to diagonal matrices $g_{ij}=g_{ii}\delta^i_j$ with $g_{ii}>0$.
%The left-invariant 
A left-invariant metric tensor field $\mathcal{G}$ induces a left-invariant Riemannian metric $d_{\mathcal{G}}: (\R^2 \times S^1) \times (\R^2 \times S^1) \rightarrow \R_{\geq 0}$ given by
\begin{equation}\label{Eq: Riemannian distance}
    d_{\mathcal{G}}(\bp_0,\bp_1) = \hspace{-1em} \inf_{\gamma \in \Gamma(\bp_0,\bp_1)}\int_0^1\sqrt{\mathcal{G}_{\gamma(t)}(\dot{\gamma}(t),\dot{\gamma}(t))}\mathrm{d}t,
\end{equation}
where $\Gamma(\bp_0,\bp_1)$ is defined in Def.~\ref{def:Gamma}. 
% Note that, the distance $d_\mathcal{G}$ does not depend on the variable time variable $t$, meaning that we are allowed to re-parametrize every curve between $0$ and $1$.

\subsubsection{Sub-Riemannian Distance\texorpdfstring{ \unboldmath $d_{\cF_0}$}{}}
The distance $d_{\cF_0}$ \eqref{Eq: finsler distance} restricts geodesics to the set $\Gamma_0$ of horizontal curves, recall Def.~\ref{def:Gamma}. Such geodesics $\gamma(\cdot)=(\bx(\cdot),\bn(\cdot))$ satisfy $\dot{\bx} \in \textrm{span}\{\bn\}$.
%Recall, that these geodesics $t \mapsto \gamma(t)=(\bx(t),\bn(t))$ %should satisfy: $\dot{\bx}(t)= u^1(t)  \bn(t)$ with $u^1(t) \in \R$.

As shown in \cite[Thm.~2]{duitsmeestersmirebeauportegies}, the  distance $d_{\cF_0}$ is a limit of the Riemannian distance $d_{\mathcal{G}}$ \eqref{Eq: Riemannian distance} by setting $g_{11}=\xi^2$, $g_{22}\rightarrow \infty$, $g_{33}=1$, cf.~\cite[Thm.~2]{duitsmeestersmirebeauportegies}. By $g_{22} \rightarrow \infty$ `sideway motion' becomes forbidden. 

From \eqref{eq:left_invariant_frame} we can derive that horizontal curves satisfy $\dot{\gamma}(t)\in \operatorname{span}\{\left.\mathcal{A}_1\right|_{\gamma(t)}, \left.\mathcal{A}_3\right|_{\gamma(t)}\}$ for all $t\in [0,1]$.
The distance $d_{\cF_0}$ is indeed a sub-Riemannian distance in a well-posed geodesically complete sub-Riemannian manifold. 
%with a $2$-bracket generating distribution. 
% \begin{remark}[Relation between sub-Riemannian and symmetric Finslerian distance]
% \label{Remark: relation finsler sub riemannian}
%     The symmetric Finsler distance $d_{\cF_0}$ \eqref{Eq: finsler distance} equals the  sub-Riemannian distance arising from \eqref{Eq: Riemannian distance} when $g_{11}=\xi^2$, $g_{22}\rightarrow \infty$ and $g_{33}=1$.
% \end{remark}
\begin{remark}[Sub-Riemannian manifold]\label{rem:sub-riemannian manifold} 
    The sub-Riemannian manifold underlying the sub-Riemannian distance $d_{\cF_0}$ is given by 
    \begin{equation*}
        (\R^2 \times S^1, \Delta, \mathcal{G}|_{\Delta \times \Delta})
    \end{equation*}
    where $\Delta_{\bp}=\textrm{span}\{\left.\mathcal{A}_1\right|_{\bp},\left.\mathcal{A}_{3}\right|_{\bp} \}$. Each pair of points is connectable by a 
    sub-Riemannian 
    geodesic, as $\Delta$ with its Lie brackets generates the tangent space: \mbox{$\Delta+ [\Delta,\Delta] = T(\R^2 \times S^1)$.}
\end{remark}

The sub-Riemannian model $d_{\cF_0}$ can be considered as a Reeds-Shepp car \cite{BoscainESAIM,duitsmeestersmirebeauportegies}. Smooth minimizing geodesics in $\R^2 \times S^1$ can exhibit cusps in their spatial projections, as it is sometimes cheaper to use of the reverse gear of the car, as can be seen in Fig.~\ref{fig:previous_tracking}.

\subsection{Geodesic Tracking 
%in \texorpdfstring{\unboldmath $\R^2 \times S^1$}{R2 X S1} 
by Steepest Descent on %Finslerian 
Distance~Map \label{ch:track}}

Geodesic contour modeling is a two-step procedure; first, the distance map $d_{\cF_0}$ or $d_{\cF_0^+}$ is computed as a viscosity solution of the eikonal PDE and secondly, the minimizing geodesic $\gamma$ from $\bp_0$ to $\bp_1$ is computed by a backtracking algorithm from $\bp_1$. 
%This section presents an iterative method for solving the eikonal PDE on $\R^2 \times S^1$, together with a %n 
%backtracking algorithm. 

A Finslerian distance map $W(\bp):=d_{\cF}(\bp,\bp_0)$ is the viscosity solution of eikonal PDE system:
\begin{equation}\label{eq:eikonal_pde}
    \begin{array}{l}           
    \cF^*(\bp, {\rm d} W(\bp)) = 1, \text{for } \bp \neq \bp_0, \\[5pt]
    W(\bp_0) = 0,
    \end{array}
\end{equation}
with dual Finsler function $\cF^*$ defined by \eqref{eq:dual_finsler_function}.
Now for our algorithms, we either need to apply this to the Finsler function $\cF=\cF_0$ \eqref{Eq: finsler function} or $\cF=\cF_0^+$ given by 
\begin{equation}\label{F0plus}
    \cF_0^+(\bp,\dot{\bp})=
    \begin{cases}
    \cF_0(\bp, \dot{\bp}) & \textrm{ if } \dot{\bx} = u^1 \bn, u^1 \geq 0 \\
    \infty &\textrm{else}.
    \end{cases}
\end{equation}
%\begin{remark}
%For technical details on dual Finsler functions and our (numeric) %approximation of the eikonal PDE system, see Appendix~\ref{app:B}. 
%\end{remark}
In this section, we provide the eikonal PDE system and the backtracking for the relevant cases $\cF=\cF_0$ and $\cF=\cF_0^+$. 
For the converging numerical scheme for computing the distance map, see Appendix~\ref{app:B}.
The eikonal PDE system for $\cF=\cF_0$ is:
\begin{equation}
    \begin{array}{l}           
    %\|\nabla W(\bp)\|= 
    \frac{1}{\mathcal{C}(\bp)}\sqrt{
    \xi^{-2} |\mathcal{A}_1W(\bp)|^2 +
    |\mathcal{A}_{3}W(\bp)|^2
    }= 1, \\
    \text{for } \bp \neq \bp_0,\\
    W(\bp_0) = 0,
    \end{array}   
\end{equation}
where we recall the left-invariant frame \eqref{eq:left_invariant_frame}. 

The eikonal PDE system for $\cF=\cF_0^{+}$ is:
\begin{equation} \label{eikF0plus}
    \begin{array}{l}
    \frac{1}{\mathcal{C}(\bp)}\sqrt{\xi^{-2} |(\mathcal{A}_1W(\bp))_+|^2 +|\mathcal{A}_{3}W(\bp)|^2}  = 1, \\
    \text{for } \bp \neq \bp_0, \\
    W(\bp_0) = 0,
    \end{array}
\end{equation}
where we use the notation $(a)_+=\max\{a,0\}$.

% BENEDEN NAAR APP B
%R\begin{equation}\label{F0plusb}
%\cF_0^+%(\bp,\dot{\bp})=
%\begin{cases}
%\cF_0(\bp, \dot{\bp}) & %\textrm{ if } \dot{\bx} / %\|\dot{\bx}\| = \bn, \\
%\infty &\textrm{else},
%\end{cases}
%\end{equation}
The sub-Riemannian geodesic in $d_{\cF_0}$ is computed by steepest descent backtracking:
\begin{equation} \label{eq:backtracking}
    \begin{array}{rrrl}
        \dot{\gamma}(t)=& -\frac{L}{\mathcal{C}^2(\gamma(t))} \big(
        \frac{1}{\xi^2}&\mathcal{A}_1 W(\gamma(t))&\mathcal{A}_1|_{\gamma(t)}\\
        & +&\mathcal{A}_{3}W(\gamma(t))&\mathcal{A}_3|_{\gamma(t)}\big), \\
        \gamma(0) =&\bp_1,
        \gamma(1) =\bp_0,
    \end{array}
\end{equation}
with $L:=d_{\cF_0}(\bp_0,\bp_1)$ and $t \in [0,1]$. 
The sub-Riemannian geodesic with +-control in $d_{\cF_0^+}$ is computed by steepest descent backtracking:
\begin{comment}
\begin{align} \label{eq:backtracking}
    \begin{array}{l}
            \dot{\gamma}(t)= -L {\rm d}_{\hat{\bp}}(\cF_0^+)^{*}(\gamma(t), {\rm d}W(\gamma(t))) \\
            \gamma(0)=\bp_1, \gamma(1)=\bp_0 
    \end{array}     
\end{align}
for all $t \in [0,1]$, with $L:=d_{\cF_0^+}(\bp_0,\bp_1)$ and ${\rm d}_{\hat{\bp}}(\cF_0^{+})^{*}$ denotes the differential of the dual
Finsler function $(\cF_0^+)^*$ with respect to the second variable
$\hat{\bp}$, cf. \cite[Eq.~23]{duitsmeestersmirebeauportegies}).
More concretely \eqref{eq:backtracking} can be expressed in the left-invariant basis as
\end{comment}

\begin{equation} \label{eq:backtrackingplus}
    \begin{array}{rrl}
        \dot{\gamma}(t) =& - \frac{L}{\mathcal{C}^2(\gamma(t))}\Big(\frac{1}{\xi^2}& \big( \mathcal{A}_1 W(\gamma(t)) \big)_+\mathcal{A}_1|_{\gamma(t)}\\
    	&+&\mathcal{A}_{3}W(\gamma(t))\mathcal{A}_3|_{\gamma(t)}\Big),\\
        \gamma(0)=& \bp_1, \gamma(1) =\bp_0,      
    \end{array}
\end{equation}
    with $L:=d_{\cF_0^+}(\bp_0,\bp_1)$ and $t \in [0,1]$.
%\end{align}

These algorithms show us how to compute the geodesics of the models $d_{\cF_0}$ and $d_{\cF_0^+}$, needed for the models $d_{\mathrm{c}}$ \eqref{dmon} and $d_{\mathrm{proj}}$ \eqref{dproj} which we will analyze and compare next.

\section{Analysis of the Symmetric Distance Model on \texorpdfstring{\unboldmath $\R^2 \times P^1$}{R2 x P1}}\label{sec:cuspfreemodel}
In this section, we analyze the new cusp-free, symmetric distance model on the projective line bundle $\R^2\times P^1$.

Intuitively, a cusp occurs when the optimal Reeds-Shepp car (traveling over a geodesic) switches gear (from forward to backward or vice versa). 
Recall the definition of a cusp in Def.~\ref{def:cusp}.
The proposed model, denoted by $d_{\mathrm{c}}$, restricts the Reeds-Shepp car to 
%only use a single gear; either 
either move
forward or to move backward, as we will see in Prop.~\ref{Corollary: sign-semidefinite distance} below. %show in the next proposition. %(Eq.~\eqref{Eq: finsler distance} with constraints $u^1 \geq 0$ and $u^1 \leq 0$, respectively).

% The minimizing geodesic of $d_{\mathrm{c}}$ can be directed in either direction. 
% This property is particularly desirable for contour modeling. We therefore refer to our model as $d_{\mathrm{c}}$, with ``c'' standing for contour.  

\begin{proposition}[Cusp-free model \texorpdfstring{$d_{\mathrm{c}}$ on $\R^2\times P^1$}{dc on R2 x P1}]\label{Corollary: sign-semidefinite distance}
Let the cost function be well-defined on $\R^2\times P^1$, i.e. $\mathcal{C}(\bp)=\mathcal{C}(\overline{\bp})$ for all $\bp \in \R^2 \times S^1$.
The model $d_{\mathrm{c}}: (\R^2 \times P^1) \times (\R^2 \times P^1) \to \R_{\geq 0}$ \eqref{dmon} is symmetric and can be expressed as the %following 
curve optimization problem:
\begin{equation} \label{FD}
    d_{\mathrm{c}}([\bp_0],[\bp_1]) = 
    \inf_{\gamma \in \Gamma^c([\bp_0],[\bp_1])} L(\gamma),
\end{equation}
with length $L(\gamma)=\int_0^1\cF_0(\gamma(t),\dot{\gamma}(t))  {\rm d}t$, and 
where $\Gamma^c([\bp_0],[\bp_1])$ is the set of all piecewise $C^1$-curves $\gamma : [0,1] \to \R^2 \times S^1$ with $\gamma(0) \in [\bp_0]$, $\gamma(1) \in [\bp_1]$, $\dot{\bx} = u^1 \bn$, and $u^1$ having a constant sign, i.e., it is either nonnegative or nonpositive everywhere. 

% By this `fixing of the gear of the Reeds-Shepp car', model $d_{\mathrm{c}}$ only produces minimizing geodesics whose spatial projections do not exhibit cusps.
\end{proposition}
\begin{proof}
See Appendix \ref{app:MCP}.
Essentially, imposing a forward gear ($\textrm{sign}(u^1)\geq 0$) on the Reeds-Shepp car along the curve $\gamma$ on $\R^2\times S^1$, connecting $\bp_0$ and $\bp_1$, is the same as imposing a backward gear ($\textrm{sign}(u^1)\leq 0$) on the Reeds-Shepp car along the curve $\overline{\gamma}$, connecting $\overline{\bp}_0$ and $\overline{\bp}_1$. Then symmetry follows by Lemma~\ref{Lemma: antipodal symmetry}.
\end{proof}
\begin{remark}[Minimal lifted contour]
The set $\Gamma^{c}([\ul{p}_0],[\ul{p}_1])$ is the set of lifted contours (recall Section~\ref{ch:minimalcontours}), and by Prop.~\ref{Corollary: sign-semidefinite distance} our model $d_{\mathrm{c}}$ selects the lifted contour(s) with minimal %(sub-Riemannian) 
length.
\end{remark}

%\end{remark}
Recall that the new model $d_{\mathrm{c}}$ and the %previous 
model $d_{\cF_0^+}$ are related by \eqref{dmon}. Both models are 
cusp-free, but can have other remarkable points: `keypoints' as we explain next.  
\begin{definition}[Keypoints]\label{def: keypoint}
A keypoint is a point $\tilde{\bx}$ on the spatial projection of a geodesic $t \mapsto \gamma(t)=(\bx(t),\bn(t))$, where
    \begin{equation*}
        \exists 0 \leq t_0 < t_1 \leq 1, \forall t \in [t_0,t_1]: \bx(t) = \tilde{\bx}, \; \dot{\bn}(t) \neq 0.
    \end{equation*}
\end{definition}
The purely angular motion of a geodesic $\gamma$ at a keypoint is referred to as an \emph{in-place rotation}.

\begin{remark}[Keypoints at start and end for $\mathcal{C}=1$] 
% The models of $d_{\cF_0}$ and $d_{\mathrm{proj}}$ may give rise to cusps.
The model $d_{\cF_0^+}$ and $d_{\mathrm{c}}$ does not have cusps, but its produced geodesics usually have keypoints (Def.~\ref{def: keypoint}) at the boundaries \cite[Thm.~3]{duitsmeestersmirebeauportegies} for $\mathcal{C}=1$ \eqref{Eq: finsler function}. 
Essentially, in $d_{\cF_0^+}$ one rotates the spherical part at the boundaries in a minimal way, so that the rotated boundary conditions can be connected with a cusp-free sub-Riemannian geodesic. This is illustrated by the green geodesics of Fig.~\ref{fig:lat}. 
\end{remark}

% \begin{remark}[Allowing for corners and bifurcations] 
% In \cite[Fig. 14]{duitsmeestersmirebeauportegies} it is %experimentally 
% shown that keypoints are useful for the tracking of blood vessels, as they %typically 
% occur at bifurcations for the model $d_{\cF_0^+}$ \eqref{Eq: finsler distance forward} %when a %nontrivial 
% %cost function 
% when using data-driven cost $\mathcal{C}$ in \eqref{Eq: finsler function}. 
% The ability of our $d_{\mathrm{c}}$ model to handle corners (by  key points) in geodesic tracking is also beneficial for SEM applications%(Fig.~\ref{fig:introexpgen})
%   where electronic structures may have corners.
% \end{remark}

In the next subsection, we will analyze similarities and differences in the behavior of the geodesics of the new model $d_{\mathrm{c}}$ and the previous model $d_{\mathrm{proj}}$ \cite{BekkersGSI} on the projective bundle $\R^2\times P^1$.

\subsection{Analysis of \texorpdfstring{\unboldmath $d_{\mathrm{c}}$ for $\mathcal{C}=1$}{dc for C=1}}\label{section: analysis dmon}

In the analysis of $d_{\mathrm{c}}([\bp_0],[\bp_1])$, we set our cost function to be constant, say $\mathcal{C}=1$, so that the Finsler function $\cF_0$, \eqref{Eq: finsler function} depends solely on the bending stiffness parameter $\xi>0$. 

We will need to distinguish between points $[\bp_1] \in \R^2\times P^1$ that lie in a specific cone field (relative to $[\bp_0]$) and points outside of this cone field. 
Next, we specify this relevant cone field $\tilde{\mathcal{Q}}_\xi$, depicted on a discrete grid in Fig.~\ref{fig: set Q}.

\begin{definition}[Cusp-free sets]\label{Def: Cusp-free sets}
The set of end points in $\R^2 \times P^1$ reachable by cusp-free minimizing geodesics of $d_{\mathrm{proj}}([\bp_0],\cdot)$ equals
\begin{align}\label{Eq: Tilde Qxi}
\begin{split}
    \tilde{\mathcal{Q}}_\xi:=\big\{ &[\bp_1] \in \R^2 \times P^1 \mid \\
    &\exists \text{ cusp-free 
    \emph{minimizing} geodesic  }\gamma_{\mathrm{min}}  \\
    &\text{of } d_{\mathrm{proj}} 
     \textrm{ that connects a pair of }\\
     &\text{elements from } [\bp_0]=\{\bp_0,\overline{\bp}_0\}
    \\
    & \textrm{and }[\bp_1]=
    \{\bp_1,\overline{\bp}_1\} \big\}.
\end{split}
\end{align}
The set of end points in $\R^2 \times S^1$ reachable by cusp-free minimizing geodesics of $d_{\cF_0}(\bp_0, \cdot)$ is given by
\begin{align}\label{Eq: Rxi}
\begin{split}
    \mathcal{R}_\xi:=\big\{ &\bp_1 \in \R^2 \times S^1 \mid \\
    &\exists \text{ cusp-free }\gamma_{\mathrm{min}} \in \Gamma_0(\bp_0, \bp_1) \\
    & \text{that is a minimizer of } d_{\cF_0}\big\}.
\end{split}
\end{align}
Finally, the set of end points in $\R^2 \times P^1$ that can be reached with a (not necessarily minimizing) cusp-free geodesic of the model $ d_{\mathrm{proj}}$ is given by
\begin{align}\label{Eq: Qxi-pre}
%\hspace{-0.3cm}
\begin{split}
   \mathcal{Q}_\xi:=\big\{ &[\bp_1] \in \R^2 \times P^1 \mid \\
    &\exists \text{ cusp-free  geodesic $\gamma$  of } d_{\mathrm{proj}} \textrm{ that }
    \\ &\text{connects a pair of}
    \\ & \text{elements from }[\bp_0]=\{\bp_0,\overline{\bp}_0\}
    \\ &\text{and }[\bp_1]=
    \{\bp_1,\overline{\bp}_1\} \big\}.
\end{split}
\end{align}
\end{definition}
\noindent
For an explicit analytic description  of $\mathcal{R}_\xi$, cf. \mbox{\cite[Thms.~9,3]{duits2014association}.} 
As shown in \cite[Prop.~3]{BekkersGSI} one has
\begin{equation}\label{Eq: Qxi}
\begin{array}{ll} 
\mathcal{Q}_\xi &=  
\{ \pm \{\bp,\overline{\bp}\} \;|\; \bp \in \mathcal{R}_{\xi} \} \\[6pt]
 & \equiv \mathcal{R}_\xi \cup \overline{\mathcal{R}}_\xi \cup -\mathcal{R}_\xi \cup -\overline{\mathcal{R}}_\xi,
\end{array}
\end{equation}
where we use the notations for sets $A\subset \R^2\times S^1$:
\[
\begin{array}{rl}
-A &= \{(-\bx, -\bn) \mid (\bx, \bn) \in A\}= \{-\bp \mid \bp \in A \}, \\
\overline{A} &=\{(\bx, -\bn) \mid (\bx, \bn) \in A\}= \{\overline{\bp} \mid \bp \in A \}. 
\end{array}
\]
and where in the final identification in \eqref{Eq: Qxi} we flatten the equivalence classes.
%Note that $\mathcal{Q}_{\xi} \subset \R^2 \times P^1$, as $\bp_1 \in \mathcal{Q}_{\xi} \Rightarrow \overline{\bp}_1 \in %\mathcal{Q}_{\xi}$. 
%\end{definition}

\begin{remark}[Scaling homothety]\label{rem:sh}
Let $\xi>0$. 
The sets $\mathcal{R}_{\xi}$, $\mathcal{Q}_\xi$ and $\tilde{\mathcal{Q}}_{\xi}$ follow from 
$\mathcal{R}_1, \mathcal{Q}_1,\tilde{\mathcal{Q}}_1$ by spatial zooming:
\[
\begin{array}{l}
\mathcal{R}_{\xi}= S_{\xi} \mathcal{R}_1 \textrm{ and }
\mathcal{Q}_{\xi}= S_{\xi} \mathcal{Q}_1\ , \\[5pt] \textrm{with scaling } S_{\xi}(\bx,\bn)=(\xi^{-1} \bx, \bn).
\end{array}
\]
This corresponds to a scaling homothety; \cite[Rem 1.2]{duits2014association},~ \cite[Rem 4.1]{BoscainESAIM}. 
\end{remark}

See Fig.~\ref{fig: set Q} for a double-sided cone visualization of $\tilde{\mathcal{Q}}_1 \subset \mathcal{Q}_1$ with color coded components \eqref{Eq: Qxi}.

%There the points in $\mathcal{R}_{\xi}$, $ \overline{\mathcal{R}}_{\xi}$, $-\mathcal{R}_{\xi}$ and $-%\overline{\mathcal{R}}_{\xi}$ are visualized %in black, blue, red, and green, respectively, 
%on a discrete grid. 
%in Fig.~\ref{fig: set Q}. 

\begin{lemma}[$\tilde{\mathcal{Q}}_\xi$ is a strict subset of $\mathcal{Q}_{\xi}$]\label{Lemma: Tilde Qxi minus Qxi}
The set $\tilde{\mathcal{Q}}_\xi$ % (cf.~\eqref{Eq: Tilde Qxi}) 
is a subset of $\mathcal{Q}_{\xi}$ %(cf.~\eqref{Eq: Qxi}),
and 
\begin{align*}
\mathcal{Q}_{\xi} \setminus \tilde{\mathcal{Q}}_\xi:= \big\{ &[\bp_1] \in\mathcal{Q}_{\xi} \mid \gamma_{\min} \text{ of } d_{\mathrm{proj}} \text{ s.t }\\
& \gamma_{\min}(0) = \bp_0, \gamma_{\min}(1) = \bp_1 \\
& \text{and } \gamma_{\min} \text{ has a cusp} \big\} \neq \emptyset.
\end{align*}
\end{lemma}
\begin{proof}
Let $[\bp_1] \in \mathcal{Q}_\xi$ be given by \eqref{Eq: Qxi-pre}. 
By \eqref{Eq: Qxi} and by symmetry considerations we assume, without loss of generality, that $\bp_1 \in \mathcal{R}_{\xi}$.  
By \cite[Thm.~10]{duits2014association} one has $\bp_1\in\mathcal{R}_{\xi} \Rightarrow \overline{\bp}_1\notin\mathcal{R}_{\xi}$.

We now distinguish between two cases:
\begin{enumerate}
\item
 If $d_{\cF_0}(\bp_0,\bp_1) \leq d_{\cF_0}(\bp_0,\overline{\bp}_1)$,  the minimizer of $d_{\mathrm{proj}}$ is cusp-free by ~\eqref{Eq: Rxi}. This implies that $[\bp_1] \in \tilde{\mathcal{Q}}_{\xi}$,
 \item
  If $d_{\cF_0}(\bp_0,\bp_1) > d_{\cF_0}(\bp_0,\overline{\bp}_1)$, the minimizer of $d_{\mathrm{proj}}$ does have a cusp, since $\overline{\bp}_1 \notin \mathcal{R}_{\xi}$. This implies $[\bp_1] \notin \tilde{\mathcal{Q}}_\xi$.
  \end{enumerate}
  The second case is rare, but can happen:
  Fig.~\ref{fig:lat} shows that the unique minimizing geodesic (in red) between $[\bp_0]$ and $[\bp_1]$ of $d_{\mathrm{proj}}$ has 2 cusps. This implies $[\bp_1] \notin \tilde{\mathcal{Q}}_{\xi}$. However, $[\bp_1]\in \mathcal{Q}_{\xi}$, due to the existence of a cusp-free geodesic (in green) in $d_{\mathrm{proj}}$, connecting $[\bp_0]$ and $[\bp_1]$. So $\mathcal{Q}_{\xi} \setminus \tilde{\mathcal{Q}}_{\xi} \neq \emptyset$. 
\end{proof}

% By \cite[Thm.~10]{duits2014association} we have that if $\bp_1\in\mathcal{R}_{\xi}$, then $\overline{\bp}_1\notin\mathcal{R}_{\xi}$. Thus, only one of the two geodesics (or neither) of the minimization problem $d_{\mathrm{proj}}$ \eqref{dproj} is cusp-free. We can redefine $\tilde{\mathcal{Q}}_\xi$ as the set of all points $[\bp_1]\in \mathcal{Q}_\xi$, such that there exists a cusp-free geodesic of $d_{\mathrm{proj}}$  and its length is smaller than the length of the geodesic with a cusp.

Recall from Prop.~\ref{Corollary: sign-semidefinite distance} that the model $d_{\mathrm{c}}$ is well-posed and practically beneficial on $\R^2 \times P^1$, as it only produces cusp-free geodesics.
The next theorem relates $d_{\mathrm{c}}$ to $d_{\mathrm{proj}}$ and shows that the nice practical properties of $d_{\mathrm{c}}$ also come with a price in terms of the triangle inequality.

\begin{figure}
\centering 
    \includegraphics[width=\linewidth]{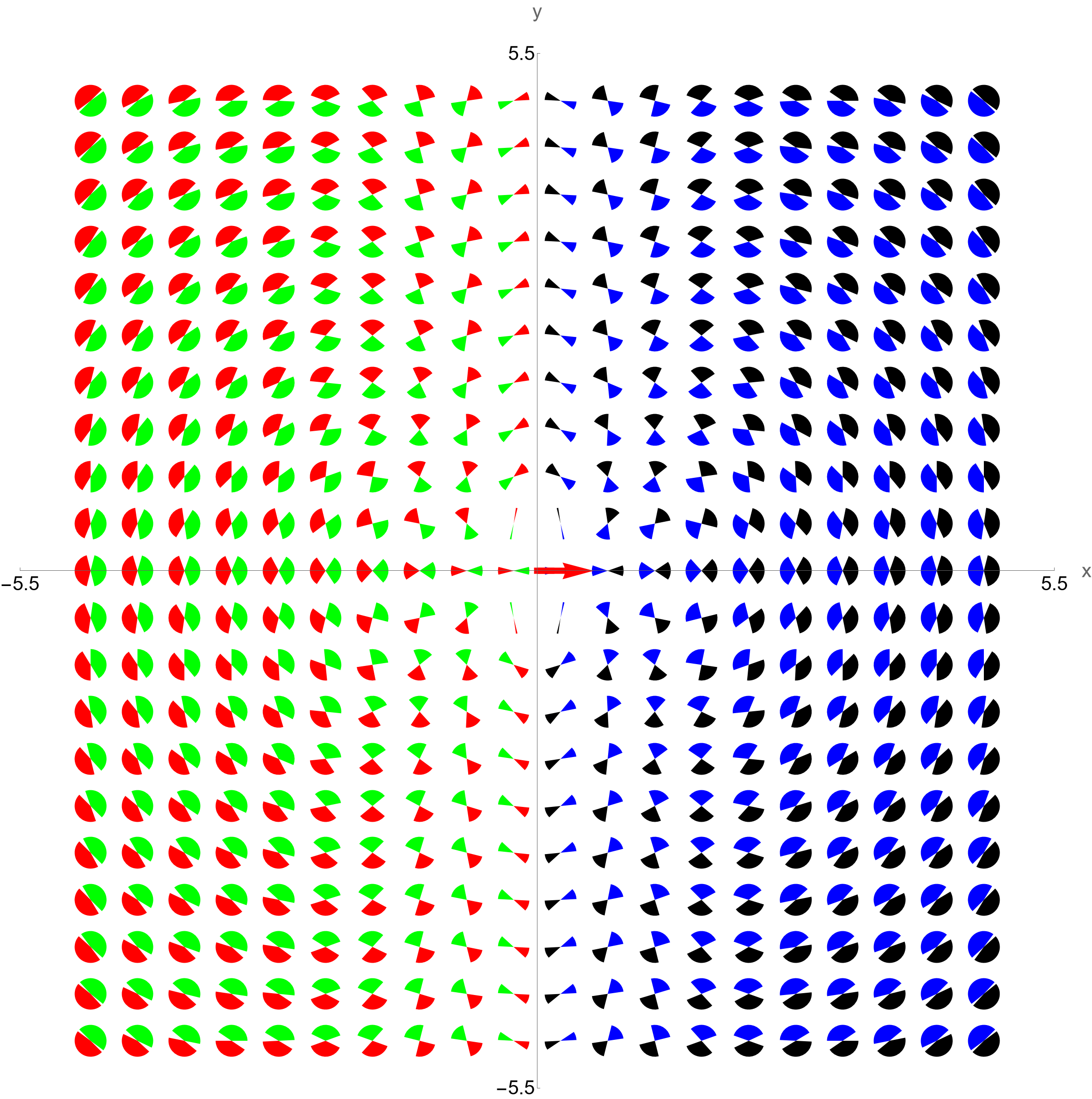}              
    \caption{Cone field visualization of the set $\tilde{\mathcal{Q}}_{\xi}$ 
    where models $d_{\mathrm{c}}$ and  $d_{\mathrm{proj}}$ coincide, recall \eqref{Eq: Tilde Qxi}. 
    In black: $\mathcal{R}_{\xi}$.
    In blue:
    $\overline{\mathcal{R}}_{\xi}$. In red:
    $-\mathcal{R}_{\xi}$. In green:
    $-\overline{\mathcal{R}}_{\xi}$. Here we set $\xi=1$. The general case $\xi>0$ follows by the scaling homothety, see Remark.~\ref{rem:sh}. 
    }\label{fig: set Q}
\end{figure} 

\begin{theorem}[Comparison $d_{\mathrm{c}}$ and $d_{\mathrm{proj}}$]\label{th:main}
Set $[\bp_0]$ as the reference element in $\R^2 \times P^1$.
Consider the set $\tilde{\mathcal{Q}}_{\xi}$ given by Eq.~\eqref{Eq: Tilde Qxi}.
\bigskip
\newline
If $[\bp_1] \in \tilde{\mathcal{Q}}_{\xi}$, then it can be connected by a minimizing geodesic departing from $[\bp_0]$ whose spatial projection is cusp-free.
%which is a minimizer %of %$d_{\cF_0}$. 
%PREVIOUS_ASSUMPTION_KEEP
%
%Assume\footnote{In %practice this %assumption is %obsolete see %Remark~\ref{rem:assum%ptioncanbedropped}} 
%it is also a %minimizer of %$d_{\mathrm{proj}}$ 
Then
\begin{equation}\label{thm: main in Q}
    d_{\mathrm{c}}([\bp_0],[\bp_1])= d_{\mathrm{proj}}([\bp_0],[\bp_1])
\end{equation}
and the triangle equality is satisfied:
\begin{equation} \label{dho}
\begin{array}{l}
d_{\mathrm{c}}([\bp_0],[\bp_1]) \leq \\
d_{\mathrm{c}}([\bp_0],[\bp_2]) +
d_{\mathrm{c}}([\bp_2],[\bp_1])
\end{array}
\end{equation}
for all $[\bp_2] \in \R^2 \times P^1$.
\bigskip
\newline
If $[\bp_1] \notin \tilde{\mathcal{Q}}_{\xi}$, then $
d_{\mathrm{c}}([\bp_0],[\bp_1]) >
d_{\mathrm{proj}}([\bp_0],[\bp_1])$ and the triangle inequality is violated; 
\newline
Either there exists  $[\bp_{\mathrm{cusp}_1}] \in \R^2\times P^1$ such that 
\begin{equation}\label{thm: main not in Q, 1 cusp}
\begin{array}{l}
    d_{\mathrm{c}}([\bp_0],[\bp_{\mathrm{cusp}_1}]) + d_{\mathrm{c}}([\bp_{\mathrm{cusp}_1}],[\bp_1]) \\
    < d_{\mathrm{c}}([\bp_0],[\bp_1]),
    \end{array}
\end{equation} 
or
 there exist $[\bp_{\mathrm{cusp}_1}],[\bp_{\mathrm{cusp}_2}] \in \R^2\times P^1$ such that
\begin{equation}\label{thm: main not in Q, 2 cusps}
\begin{array}{l}
    d_{\mathrm{c}}([\bp_0],[\bp_{\mathrm{cusp}_1}]) + d_{\mathrm{c}}([\bp_{\mathrm{cusp}_1}],[\bp_{\mathrm{cusp}_2}]) \\
    + \; d_{\mathrm{c}}([\bp_{\mathrm{cusp}_2}],[\bp_1])
    < d_{\mathrm{c}}([\bp_0],[\bp_1]).
    \end{array}
\end{equation}
\end{theorem}

\begin{proof}
%Assume that the minimizer of $d_{\mathrm{proj}}([\bp_0],[\bp_1])$ has no cusps in its spatial projection \eqref{definiton cusp}. Note that assumption is justified by the result of \cite{BekkersGSI} where geodesic fronts from $\bar{\bp}_0$ only reach geodesic with a cusps and furthermore geodesic without cusps are always globally minimizingComplete the proof s&}

For all $[\bp_1] \in \tilde{\mathcal{Q}}_\xi \subset \mathcal{Q}_{\xi}$, we have by \eqref{Eq: Qxi} that $\bp_1, \overline{\bp}_1 \in \mathcal{Q}_\xi$. Now $\bp_1$ is an element of $\mathcal{R}_\xi, \overline{\mathcal{R}}_\xi, -\mathcal{R}_\xi$ or $-\overline{\mathcal{R}}_\xi$. Each of these 4 cases is treated separately: 
\begin{enumerate}
    \item[\textbf{1)}] For $\bp_1 \in \mathcal{R}_\xi$, there exists a cusp-free minimizing geodesic $\gamma_{\mathrm{min}}(t)$ with $t\in[0,1]$ of the model $d_{\cF_0}$, such that $\gamma_{\mathrm{min}}(0)=\bp_0$ and $\gamma_{\mathrm{min}}(1)=\bp_1$. Then:
    %We can derive the following (in)equalities:
    \begin{equation} \label{usetildeQxi}
     \bp_1 \in \mathcal{R}_\xi \Rightarrow
    \left\{
    \begin{array}{ll}
        d_{\cF^+_0}(\bp_0,\bp_1) &= d_{\cF_0}(\bp_0,\bp_1) \\
        &\leq d_{\cF_0}(\bp_0,\overline{\bp}_1) \\
        &\leq  d_{\cF^+_0}(\bp_0,\overline{\bp}_1)
        \end{array}
        \right.
    \end{equation}
    The first equality follows from $\gamma_{\mathrm{min}}$ being a cusp-free geodesic. The first inequality follows the definition of $\tilde{\mathcal{Q}}_{\xi}$ given in \eqref{Eq: Tilde Qxi}. The second inequality follows from the additional constraint of the +-control in Eq.~\eqref{Eq: finsler distance forward}. 
    
   Furthermore, one has $d_{\cF_0}(\bp_0,\overline{\bp}_1)=d_{\cF_0}(\overline{\bp}_0,\bp_1) \leq d_{\cF^+_0}(\overline{\bp}_0,\bp_1)$, where the equality follows from antipodal symmetry (Lemma~\ref{Lemma: antipodal symmetry}) and the inequality follows from the constraint of the +-control. 
  Similarly one has $d_{\cF_0}(\bp_0,\bp_1)=d_{\cF_0}(\overline{\bp}_0,\overline{\bp}_1) \leq d_{\cF^+_0}(\overline{\bp}_0,\overline{\bp}_1)$. Combining all results, we get that:
    \begin{align}\label{eq:dc is dF_0^+}
        d_{\cF_0^+}(\bp_0,\bp_1) \leq \begin{cases}
            d_{\cF^+_0}(\bp_0,\overline{\bp}_1)\\
            d_{\cF^+_0}(\overline{\bp}_0,\bp_1)\\
            d_{\cF^+_0}(\overline{\bp}_0,\overline{\bp}_1).
        \end{cases}
    \end{align}
    Hence, we obtain by the definition of $d_{\mathrm{c}}$, Eq.~\eqref{dmon}, \eqref{eq:dc is dF_0^+}, \eqref{usetildeQxi}, and by definition of $d_{\mathrm{proj}}$, Eq.~\eqref{dproj} that
    \[
    \begin{array}{ll}
    d_{\mathrm{c}}([\bp_0],[\bp_1]) &=d_{\cF_0^+}(\bp_0,\bp_1)=d_{\cF_0}(\bp_0,\bp_1) \\
     &=d_{\mathrm{proj}}([\bp_0],[\bp_1]).
    \end{array}
    \]
      
    %, we have that:
    %\begin{equation*} 
    %    d_{\mathrm{proj}}([\bp_0],[\bp_1])=d_{\cF_0}(\bp_0,\bp_1)=d_{\cF_0^+}(\bp_0,\bp_1),
    %\end{equation*}
    %from which we can conclude %that indeed $d_{\mathrm{c}}%([\bp_0],%[\bp_1])=d_{\mathrm{proj}}%([\bp_0],[\bp_1])$.

    \item[\textbf{2)}] For $\bp_1 \in \overline{\mathcal{R}}_\xi$, we use the following symmetry with the geodesic $\gamma_{\mathrm{min}}$ of the previous case \textbf{1)}:
    \[ \gamma^{\mathrm{new}}_{\mathrm{min}}(t)=\overline{\gamma}_{\mathrm{min}}(1-t), 
    \]
    for all $t \in [0,1]$. 
    The lengths of geodesics $\gamma_{\mathrm{min}}$ and $\gamma_{\mathrm{min}}^{\mathrm{new}}$ are equal by the construction of $\cF_0$, recall Eq. \eqref{Eq: finsler function}. So in this case:
      \[
    \begin{array}{ll}
    d_{\mathrm{c}}([\bp_0],[\bp_1]) &=
    d_{\cF_0^+}(\bp_1,\overline{\bp}_0) =d_{\cF_0^+}(\bp_0,\overline{\bp}_1)\\
     % =d_{\cF_0}(\bp_0,\overline{\bp}_1)
     &=d_{\mathrm{proj}}([\bp_0],[\bp_1]).
    \end{array}
    \]
    where the second equality follows by Lemma \ref{Lemma: antipodal symmetry}. We note that $\gamma_{\mathrm{min}}$
      connects $\bp_0$ to $\bp^{\textrm{prev}}_1$
     % \footnote{\mbox{where $\bp^{\textrm{prev}}_1 \in \mathcal{R}_\xi$ is equal to $\bp_1$ of case \textbf{1)}}} 
      iff $\gamma_{\mathrm{min}}^{\mathrm{new}}$ connects $\overline{\bp}^{\textrm{prev}}_1=\bp_1$ to $\overline{\bp}_0$.
    \item[\textbf{3)}] For $\bp_1 \in -\mathcal{R}_\xi$, we rely on the symmetry:
    \[ \gamma^{\mathrm{new}}_{\mathrm{min}}(t)=-\gamma_{\mathrm{min}}(t).\]
    The length of both geodesics $\gamma_{\mathrm{min}}$ and $\gamma_{\mathrm{min}}^{\mathrm{new}}$ is equal by the construction of $\cF_0$. So in this case:
 \[
    \begin{array}{ll}
    d_{\mathrm{c}}([\bp_0],[\bp_1]) &=d_{\cF_0^+}(\overline{\bp}_0,\bp_1) %=d_{\cF_0}(\overline{\bp}_1,\overline{\bp}_0)\\
     %&
     =d_{\mathrm{proj}}([\bp_0],[\bp_1]),
    \end{array}
    \]
    where we note that $-\bp_0=\overline{\bp}_0$ and $-\bp_1^{\textrm{prev}}=\bp_1$. \\[2pt]
    \item[\textbf{4)}] For $\bp_1 \in -\overline{\mathcal{R}}_\xi$, we rely on the symmetry:
    \[ \gamma_{\mathrm{min}}^{\mathrm{new}}(t)=-\overline{\gamma}_{\mathrm{min}}(1-t).
    \] 
   The lengths of $\gamma_{\mathrm{min}}$ and $\gamma_{\mathrm{min}}^{\mathrm{new}}$ coincide, and:
 \[
    \begin{array}{ll}
    d_{\mathrm{c}}([\bp_0],[\bp_1]) &=d_{\cF_0^+}(\bp_1,\bp_0) = d_{\cF_0^+}(\overline{\bp}_0,\overline{\bp}_1) %=d_{\cF_0}(\overline{\bp}_1,\overline{\bp}_0)\\
     %&
     \\ &=d_{\mathrm{proj}}([\bp_0],[\bp_1]),
    \end{array}
    \]
    where we note that $-\overline{\bp}_1^{\textrm{prev}}=\bp_1$ and $-\overline{\bp}_0=\bp_0$.
\end{enumerate}
% For $[\bq]=[\bp_1],[\bp_2]$ or $[\bp_2^{-1}\bp_1]\in \mathcal{Q}_\xi$, we have that $d_{\mathrm{c}}([\bp_0],[\bq])=d_{\mathrm{proj}}([\bp_0],[\bq])$ and thus the triangle inequality is satisfied as $d_{\mathrm{proj}}$ satisfies the triangle inequality. 
For $[\bp_1]\in \tilde{\mathcal{Q}}_\xi$ and for $[\bp_2]\in \R^2 \times P^1$, we have
\begin{align*}
    d_{\mathrm{c}}([\bp_0],[\bp_1])&=d_{\mathrm{proj}}([\bp_0],[\bp_1]) \\
    &\leq d_{\mathrm{proj}}([\bp_0],[\bp_2]) + d_{\mathrm{proj}}([\bp_2],[\bp_1])\\
    &\leq d_{\mathrm{c}}([\bp_0],[\bp_2]) + d_{\mathrm{c}}([\bp_2],[\bp_1]).
\end{align*}
The first inequality follows from the triangle inequality of $d_{\mathrm{proj}}$. The second inequality follows from the constraint $u^1\geq0$ of the +-control and thus the triangle inequality is satisfied for $d_{\mathrm{c}}$ on $\tilde{\mathcal{Q}}_\xi$.
We have now shown that \eqref{thm: main in Q}, \eqref{dho} hold. 

Now let us verify, the second part of the theorem where $[\bp_1] \notin \tilde{\mathcal{Q}}_{\xi}$.
Then the minimizing sub-Riemannian geodesic of the model $d_{\mathrm{proj}}$ has at least one cusp. 
Thereby imposing the constraint $u^1\geq 0$ is a restrictive constraint and therefore the distance increases, $
d_{\mathrm{c}}([\bp_0],[\bp_1]) >
d_{\mathrm{proj}}([\bp_0],[\bp_1])$.

Finally, we show the triangle inequality is violated for $[\bp_1] \notin \tilde{\mathcal{Q}}_{\xi}$. 
\citet[Cor.~4.5]{BoscainESAIM} has shown that the minimizing geodesic of $d_{\cF_0}$ has atmost two cusps. We assume without loss of generality that $d_{\cF_0}(\bp_0,\bp_1)\leq d_{\cF_0}(\bp_0,\overline{\bp}_1)$, thereby $d_{\mathrm{proj}}([\bp_0],[\bp_1])=d_{\cF_0}(\bp_0,\bp_1)$.  Suppose the minimizing sub-Riemannian geodesic has one cusp and let $\bp_{\mathrm{cusp}}$ be the cusp of the minimizing geodesic from $\bp_0$ to $\bp_1$, then by the definition of a cusp (Def.~\ref{def:cusp}) one has:
\begin{align*}
    d_{\cF_0}(\bp_0,\bp_1) &=d_{\cF_0^\pm}(\bp_0,\bp_{\mathrm{cusp}})+d_{\cF_0^\mp}(\bp_{\mathrm{cusp}},\bp_1)
\end{align*}
where $\cF_0^-$ is the negative control version of $\cF_0^+$ in \eqref{Eq: finsler distance forward}, imposing $u^1\leq 0$ rather than $u^1\geq 0$. 
%Each of the two parts of the geodesic is now cusp-free and thereby minimizing \cite[]{BoscainESAIM}.

From the Definition of $d_{\mathrm{c}}$, see Eq.~\eqref{dmon}, we have that 
\begin{equation*}
    d_{\mathrm{c}}([\bp_0],[\bp_{\mathrm{cusp}}]) \leq d_{\cF_0^\pm}(\bp_0,\bp_{\mathrm{cusp}})
\end{equation*}
and 
\begin{equation*}
    d_{\mathrm{c}}([\bp_{\mathrm{cusp}}],[\bp_1]) \leq d_{\cF_0^\mp}(\bp_{\mathrm{cusp}},\bp_1).
\end{equation*} 

However, imposing the constraint $u^1 \geq 0$, implies $d_{\mathrm{c}}([\bp_0],[\bp_1]) > d_{\mathrm{proj}}([\bp_0],[\bp_1])$. Now, we conclude
\begin{align*}
    d_{\mathrm{c}}([\bp_0],[\bp_{\mathrm{cusp}}])+ d_{\mathrm{c}}([\bp_{\mathrm{cusp}}],[\bp_1]) &\leq \\ d_{\cF_0^\pm}(\bp_0,\bp_{\mathrm{cusp}})+d_{\cF_0^\mp}(\bp_{\mathrm{cusp}},\bp_1) &=d_{\cF_0}(\bp_0,\bp_1)\\
=d_{\mathrm{proj}}([\bp_0],[ \bp_1]) 
    &< d_{\mathrm{c}}([\bp_0],[\bp_1])
\end{align*}
This proves \eqref{thm: main not in Q, 1 cusp}. The case of two cusps is similar.

%For the second case, that the minimizing sub-Riemannian geodesic from $\bp_0$ to $\bp_1$ has two cusps, one has that:
%\begin{align*}
%    d_{\cF_0}(\bp_0,\bp_1) & =d_{\cF_0^\pm}(\bp_0,\bp_{\mathrm{cusp}_1})+d_{\cF_0^\mp}(\bp_{\mathrm{cusp}_1},\bp_{\mathrm{cusp}_2}) \\
 %   %d_{\cF_0^\pm}%%%(\bp_{\mathrm{cusp}_2}%,\bp_1).
%\end{align*}
%With the same reasoning %as in the one-cusp %situation, one can %concludes that the %triangle inequality is %violated.

\end{proof}

The natural question that now arises is: 
\\[4pt]
\emph{What are the differences between the model $d_{\mathrm{proj}}$ and the cusp-free model $d_{\mathrm{c}}$ when we start from the origin $[\bp_0]$ and the end point $[\bp_1]$ is not in %the set 
$\tilde{\mathcal{Q}}_{\xi}$?}
\\[4pt]
In the next section, we shall visualize the differences. The absence of cusps in $d_{\mathrm{c}}$ increases the distance compared to $d_{\mathrm{proj}}$ (see Fig.~\ref{fig:monprojoutcone}), which confirms the final part of Theorem~\ref{th:main}. 

In the development of the spheres (Figs. \ref{fig:Maxwellmon}, \ref{fig:Maxwellproj}), we will see that a vertical line segment through the origin is part of the spheres with $R\leq \pi$ of $d_{\mathrm{c}}$, while not being part of the spheres of $d_{\mathrm{proj}}$. 
Essentially, the main difference becomes visible if we connect
$[\bp_0]=\{\bp_0,\overline{\bp}_0\}$ to $[\bp_1^{\varepsilon}]=\{\bp^{\varepsilon}_1,\overline{\bp}^{\varepsilon}_1\}$ where
\begin{equation} \label{p1e}
 \bp_0=((0,0),(1,0)), \ \bp^{\varepsilon}_1=((0,\varepsilon),(1,0)), 
\end{equation}
and $0<\varepsilon \xi <\pi$. That is, if the boundary conditions approach each other from the spatial lateral direction. Then, the projective line bundle geodesic has two cusps, whereas the cusp-free geodesic of the model $d_{\mathrm{c}}$ first applies an in-place rotation at the boundary and then connects to $[\bp_1^\varepsilon]$ with a $U$-curve (which in the limiting case $\varepsilon \downarrow 0$ becomes a straight line). This can be observed in Fig.~\ref{fig:lat}.

It also confirms the following theoretical result that  applies to the general cost case with $\mathcal{C}\geq \delta$ \eqref{Eq: finsler function}. %Similar to 

\begin{proposition}[Global and local controllability]
\label{corr:tobewritten}
Both models $d_{\mathrm{c}}, d_{\mathrm{proj}}$ are globally controllable, 
but $d_{\mathrm{c}}$ is not locally controllable as it has the property
\begin{equation} \label{eq:bb}
\limsup_{[\bp]\to[\bp_0]} d_{\mathrm{c}}([\bp],[\bp_0]) \geq  \pi \delta.
\end{equation}
\end{proposition}

\begin{proof}
Essentially, the estimate $\pi\delta$ is due to the fact that if we approach the origin $[\bp_0]$ from the side, that is, via $[\bp_1^{\varepsilon}]$ \eqref{p1e}, the minimizing geodesic of the new model $d_{\mathrm{c}}$ (depicted in green in Fig.~\ref{fig:lat}) does \emph{not} exhibit the (undesirable) double cusps in its spatial projection and approaches an in-place rotation over $\pi$ for $\varepsilon \downarrow 0$. For details on the proof, see Appendix~\ref{App:PropGlobalLocalControllability}.
\end{proof}

\begin{figure*}
\centering
\hfill
    \begin{subfigure}[T]{0.4\textwidth}
        \centering        
        \includegraphics[height=0.85\textwidth, angle=90, trim={500 100 200 450}, clip]{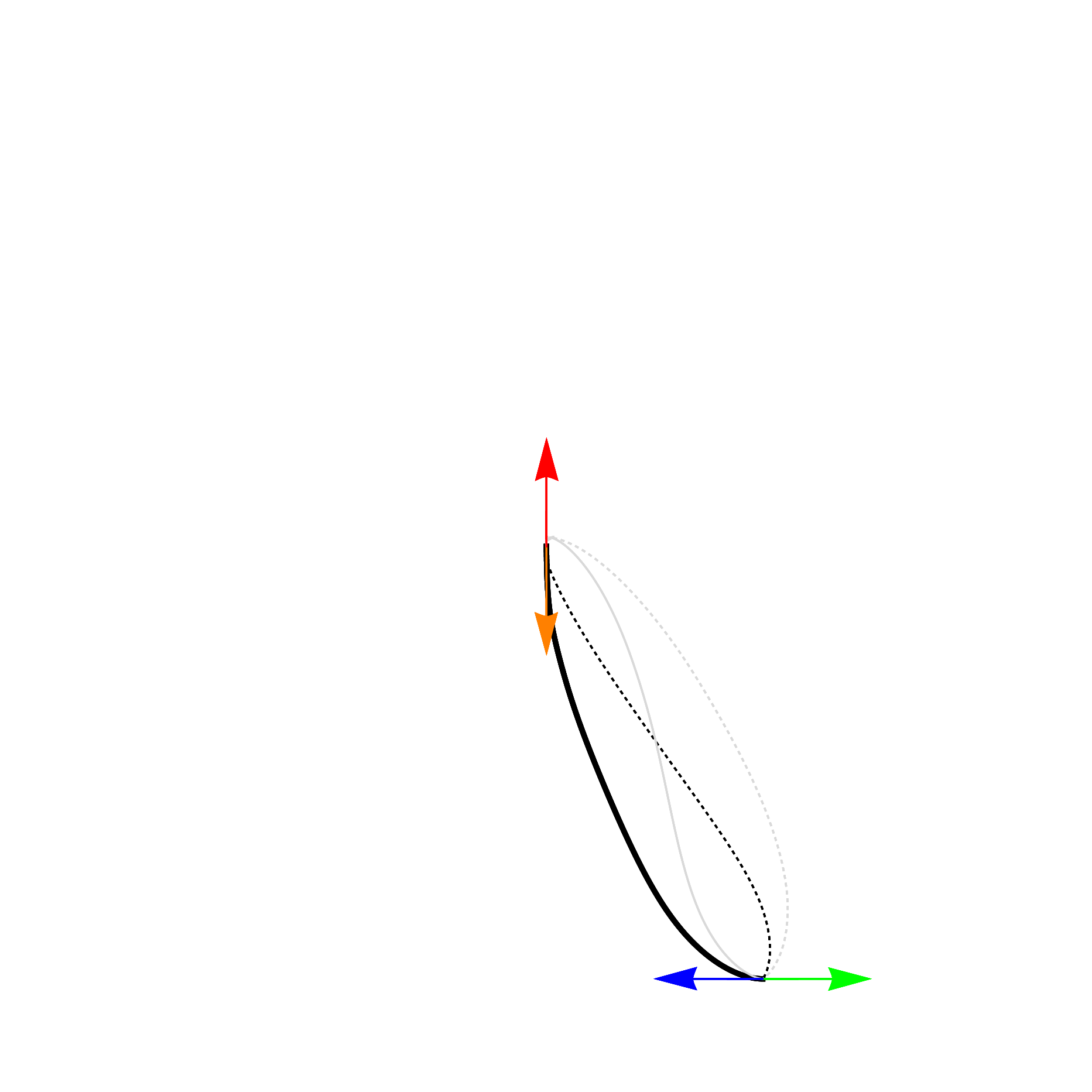}
        \caption{The four minimizing geodesics of the new model $d_{\mathrm{c}}$. See the right hand side of Eq.~\eqref{dmon}.}
        \label{fig:geodesics_dc_inQ}
    \end{subfigure}   
    \hfill
    \begin{subfigure}[T]{0.4\textwidth}
        \centering        
        \includegraphics[height=0.85\textwidth,angle=90, trim={500 100 200 450}, clip]{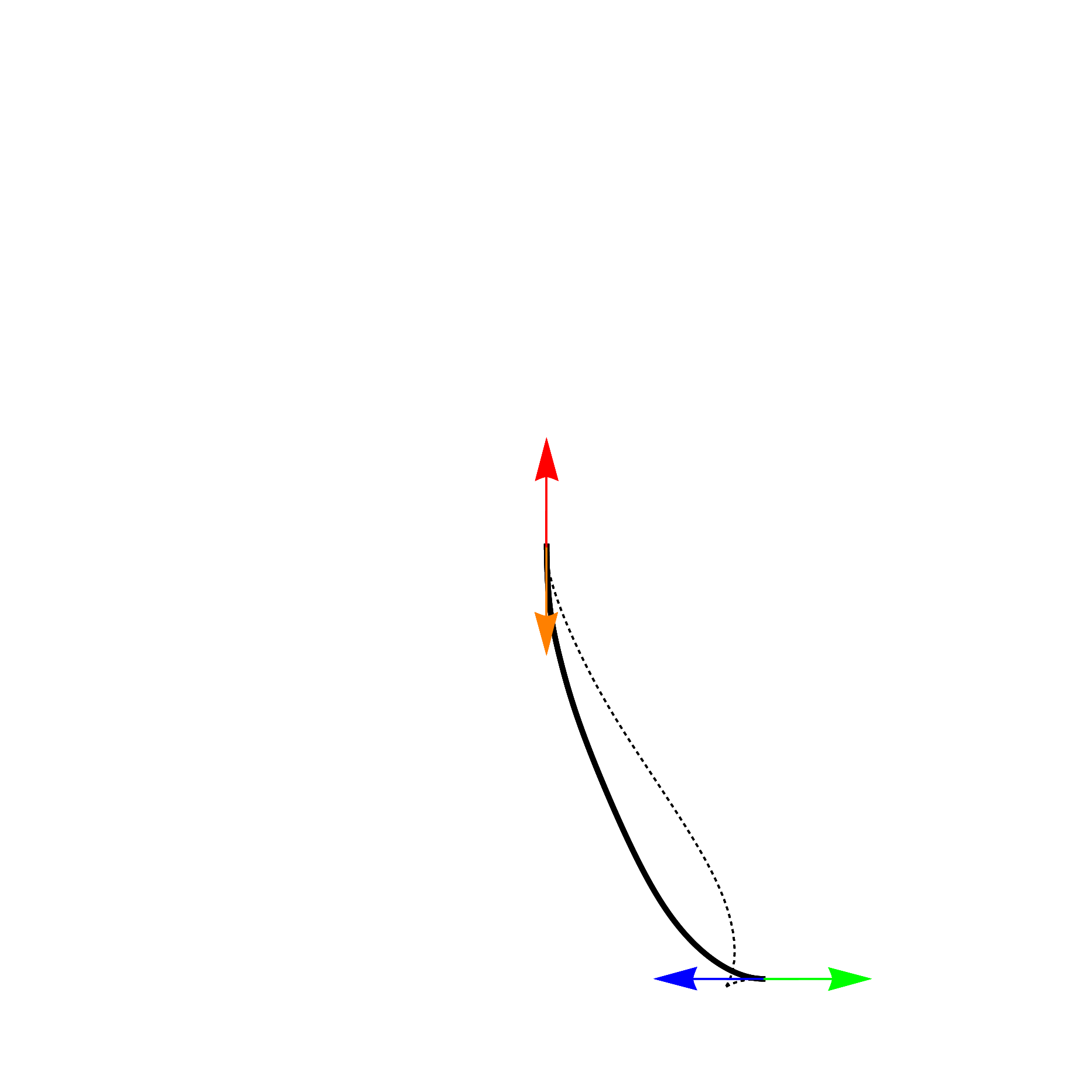}
        \caption{The two minimizing geodesics of the model $d_{\mathrm{proj}}$. See the right hand side of Eq.~ \eqref{dproj}.}
        \label{fig:inQ}
    \end{subfigure}
\hfill~
        
    \caption{Illustration of Theorem \ref{th:main} for case $[\bp_1] \in \mathcal{Q}_\xi$. The bold geodesics of the models $d_{\mathrm{c}}$ and $d_{\mathrm{proj}}$ are the same, and they are the minimizers of both models, i.e., $d_{\mathrm{c}}=d_{\mathrm{proj}}$.}
    \label{fig:monprojinQ}
\end{figure*}

\begin{figure*}
\centering
\hfill
\begin{subfigure}[T]{0.4\textwidth}
    \centering
    \includegraphics[height=0.85\textwidth,angle=90, trim={450 80 400 450}, clip]{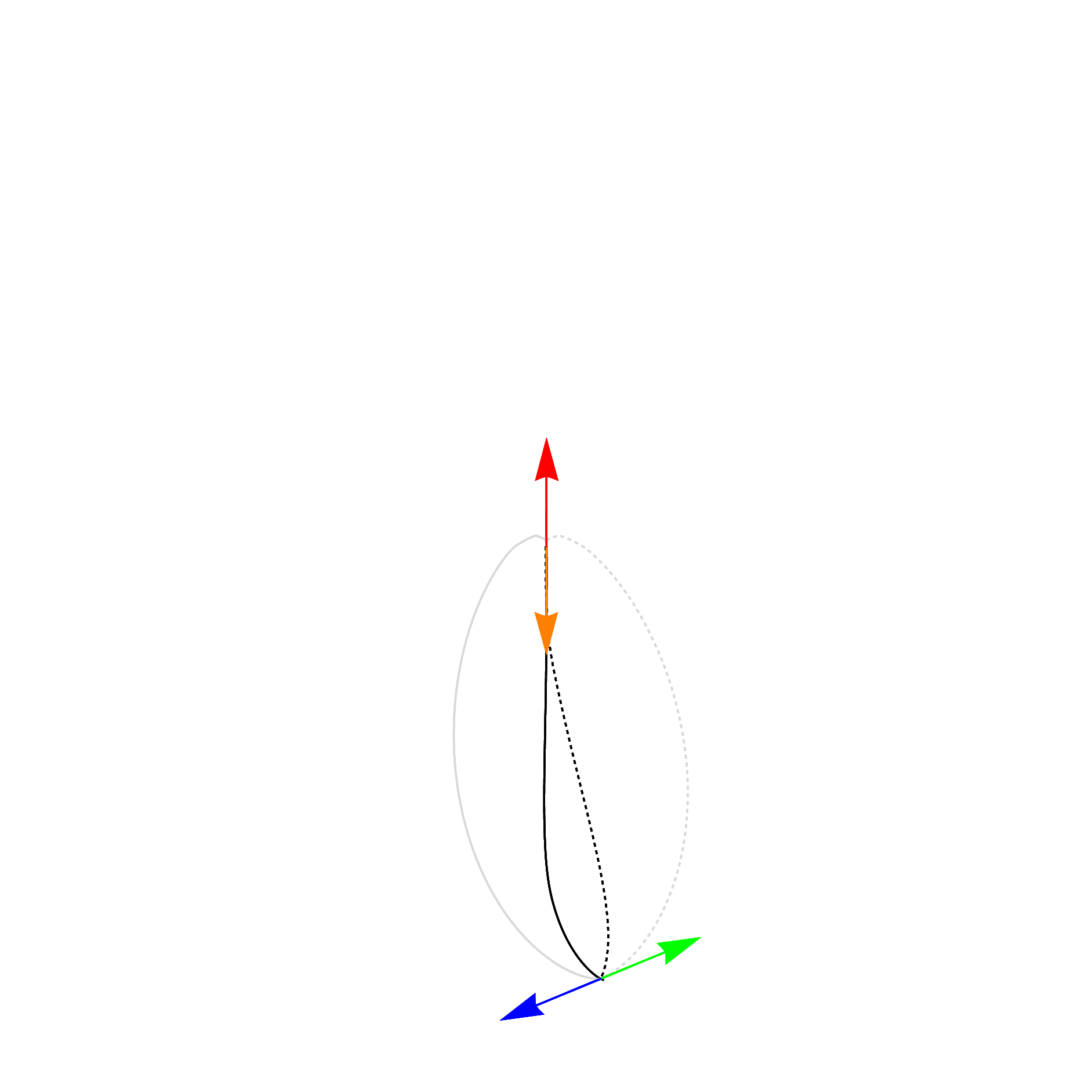}
    \caption{The four minimizing geodesics of the new model $d_{\mathrm{c}}$. See the right hand side of Eq.~\eqref{dmon}.}
    \label{fig:geodesics_dc_not_inQ}
\end{subfigure}    
\hfill
\begin{subfigure}[T]{0.4\textwidth}
    \centering
    \includegraphics[height=0.85\textwidth,angle=90, trim={450 80 400 450}, clip]{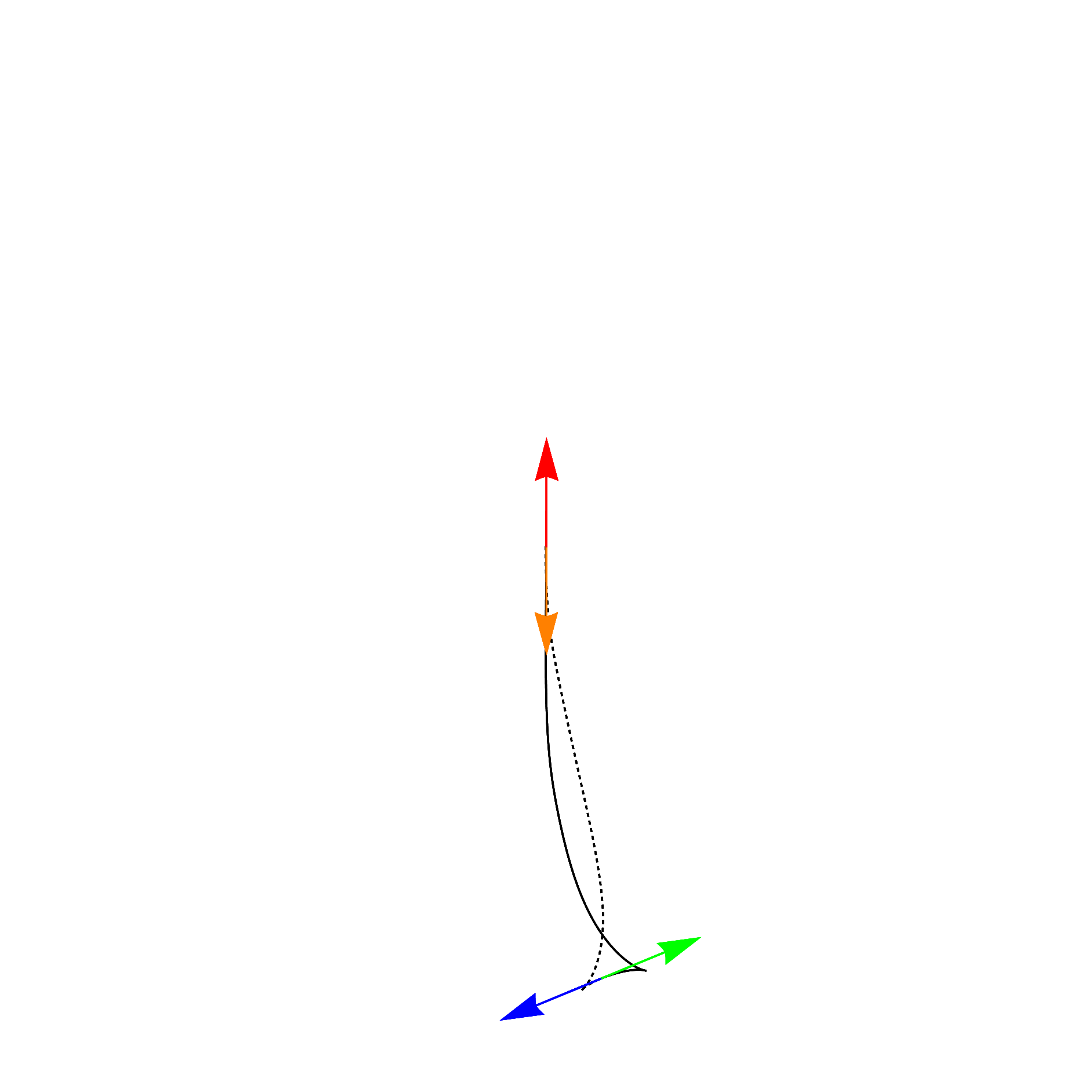}
    \caption{The two minimizing geodesics of the model $d_{\mathrm{proj}}$. See the right hand side of Eq.~ \eqref{dproj}.}
    \label{fig:not_inQ}
\end{subfigure}
\hfill~

\caption{Illustration of Theorem \ref{th:main} for case $[\bp_1] \notin \mathcal{Q}_\xi$. The geodesics over which are minimized in the models $d_{\mathrm{c}}$ and $d_{\mathrm{proj}}$ are all different. The geodesics of $d_{\mathrm{proj}}$ have cusps, whereas the geodesics of $d_{\mathrm{c}}$ do not have cusps. By constraining the model to be cusp-free, we have $d_{\mathrm{c}}>d_{\mathrm{proj}}$.}
\label{fig:monprojoutcone}
\end{figure*}

\subsection{Visualizing the Differences of the Models \texorpdfstring{\unboldmath $d_{\mathrm{c}}$ and $d_{\mathrm{proj}}$}{dc and dproj}}\label{section: visual dmon}
In our visual comparisons between $d_{\mathrm{c}}$ and $d_{\mathrm{proj}}$, recall Eq.~\eqref{dmon} and Eq.~\eqref{dproj}, we recognize 3 main observations that we list below. \\ \\
% \textbf{1. Different Minimization:} \\ %(even if $[\bp_1] \in \tilde{\mathcal{Q}}_{\xi}$)
% %\textbf{:} 
% In the model $d_{\mathrm{c}}$ we minimize over four geodesics, since the model $d_{\cF_0^+}$ is not inversion invariant, while in the projective case ($d_{\mathrm{proj}}$) we minimize over only two geodesics, as visualized in the Figures \ref{fig:monprojinQ} and \ref{fig:monprojoutcone}. 
% \\ \\
\textbf{1. Same Geodesic Cases:} \\
For $[\bp_1]\in\tilde{\mathcal{Q}}_\xi$ one has
$d_{\mathrm{c}}([\bp_0],[\bp_1])=d_{\mathrm{proj}}([\bp_0], [\bp_1])$, and the minimizing geodesics of the models coincide, by Thm.~\ref{th:main}. The non-minimizing geodesics differ. All geodesics are shown in Fig.~\ref{fig:monprojinQ} with coinciding geodesics in bold.  \\ \\
\textbf{2. Different Geodesic Cases:} \\
For $[\bp_1]\notin\mathcal{Q}_\xi$ one has that the underlying models of $d_{\mathrm{c}}$ and $d_{\mathrm{proj}}$ are different, ($d_{\cF_0^+}$ and $d_{\cF_0}$ respectively), so that all geodesics differ, cf. Fig.~\ref{fig:monprojoutcone}.
\\ \\
\textbf{3. Different Maxwell Sets: } \\
A Maxwell point of model $d_{\mathrm{c}}$ respectively $d_{\mathrm{proj}}$ is by definition a point $[\bp_1] \in \R^2 \times P^1$ such that there exist two or more distinct minimizing geodesics of the model $d_{\mathrm{c}}$ respectively $d_{\mathrm{proj}}$, with equal length connecting $[\bp_0]$ and $[\bp_1]$.
% Geodesics lose their optimality at Maxwell points.

% \begin{definition}[Maxwell set of $d_{\mathrm{c}}$ and $d_{\mathrm{proj}}$]\label{def: Maxwell set}
% The Maxwell sets of $d_{\mathrm{c}}$ and $d_{\mathrm{proj}}$ are given by
%     \begin{align*}
%         \mathcal{M}_{\mathrm{c}} := \Big\{ &[\bp] \in \R^2 \times P^1 \mid \exists \gamma^1 \neq \gamma^2 \in \Gamma_0^+([\bp_0], [\bp]),\\
%         & L(\gamma^1) = L(\gamma^2) = d_{\mathrm{c}}([\bp_0],[\bp]) \Big\}, \textrm{ and}   \\
%    \mathcal{M}_{\mathrm{p}}:=\Big\{ &[\bp] \in \R^2 \times P^1 \mid \exists \gamma^1 \neq \gamma^2 \in \Gamma_0([\bp_0], [\bp]),\\
%         & L(\gamma^1) = L(\gamma^2) = d_{\mathrm{proj}}([\bp_0],[\bp]) \Big\},
%     \end{align*}    
% respectively, with $\Gamma_0([\bp_0], [\bp_1])$ and $\Gamma_0^+([\bp_0], [\bp_1])$ as defined in Def.~\ref{def:Gamma} and length functional 
% \[L(\gamma) := \int_0^1 \cF_0(\gamma(t),\dot{\gamma}(t)) \mathrm{d}t.\]
% \end{definition}
In $\R^2 \times P^1$ we must keep track of the geodesic front departing from $\bp_0$ and from $\overline{\bp}_0$, as the origin is $[\bp_0]=\{\bp_0,\overline{\bp}_0\}$.
The fronts departing from $\bp_0$ are depicted in green in Fig.~\ref{fig:Maxwellmon} and in Fig.~\ref{fig:Maxwellproj}. The fronts departing from $\overline{\bp}_0$ are depicted in red ($\theta \geq 0$) and blue ($\theta < 0$), both in Fig.~\ref{fig:Maxwellmon} and in Fig.~\ref{fig:Maxwellproj}. The moment these optimal geodesic fronts meet each other, Maxwell-points are produced. They become visible as folds on the spheres. 

There are two main differences between the de Maxwell-set of $d_{\mathrm{c}}$ and $d_{\mathrm{proj}}$, namely, the size of the Maxwell-set is reduced when considering $d_{\mathrm{c}}$ instead of $d_{\mathrm{proj}}$, and the collision of the red and blue front takes place exactly at $R=\pi$. We will now explain these differences in more detail: 

\begin{itemize}
    \item \textbf{(Reduced Maxwell set)} \ %Note that 
    In the sphere of $d_{\mathrm{proj}}$, depicted in Fig.~\ref{fig:Maxwellproj}, there is a fold in the green front. This fold is part of the Maxwell set. 
    In the new model $d_{\mathrm{c}}$, we do not have that instability problem. This is because the vertical lines through the origin appearing in the folds of the green surfaces of the $d_{\mathrm{c}}$ spheres, depicted in Fig.~\ref{fig:Maxwellmon}, are not Maxwell-points. These points $(0,0,\theta_1)$ with $|\theta_1|<\pi/2$, are reached by a unique geodesic, namely the shortest in-place rotation. 
    \item \textbf{(Collision of fronts)} \ In the sphere of $d_{\mathrm{c}}$ the green, red, and blue fronts hit each other at the origin with radius $\pi$ (cf. Fig.~\ref{fig:Maxwellmon}), in contrast to $d_{\mathrm{proj}}$ where this occurs outside the origin at radius $R \approx\frac{10\pi}{9}$ (cf. Fig.~\ref{fig:Maxwellproj}).
\end{itemize}

\begin{figure*}
    \centering    
    \includegraphics[width=0.75\textwidth]{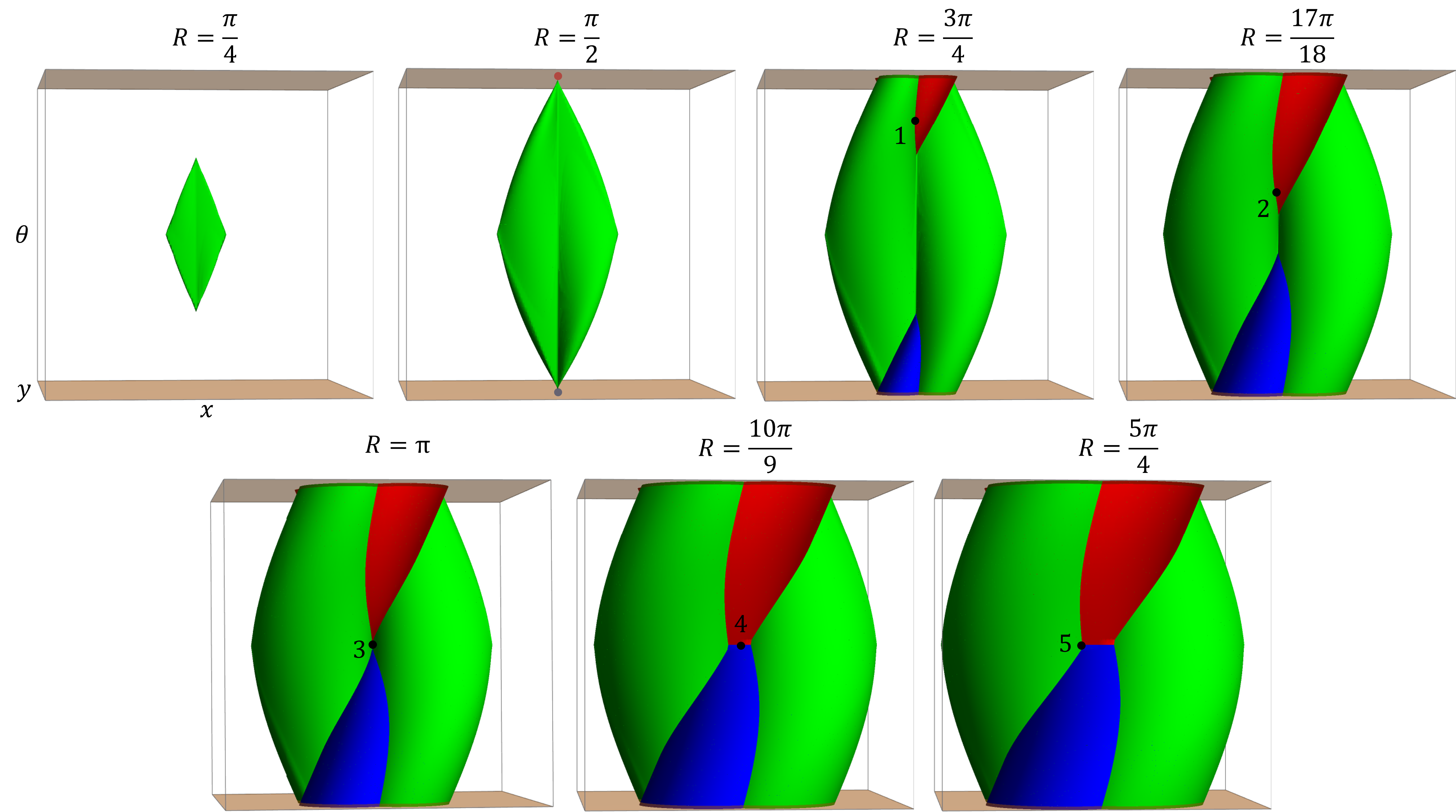}
    \textbf{\includegraphics[width=0.6\textwidth, trim={0 50 0 130}, clip]{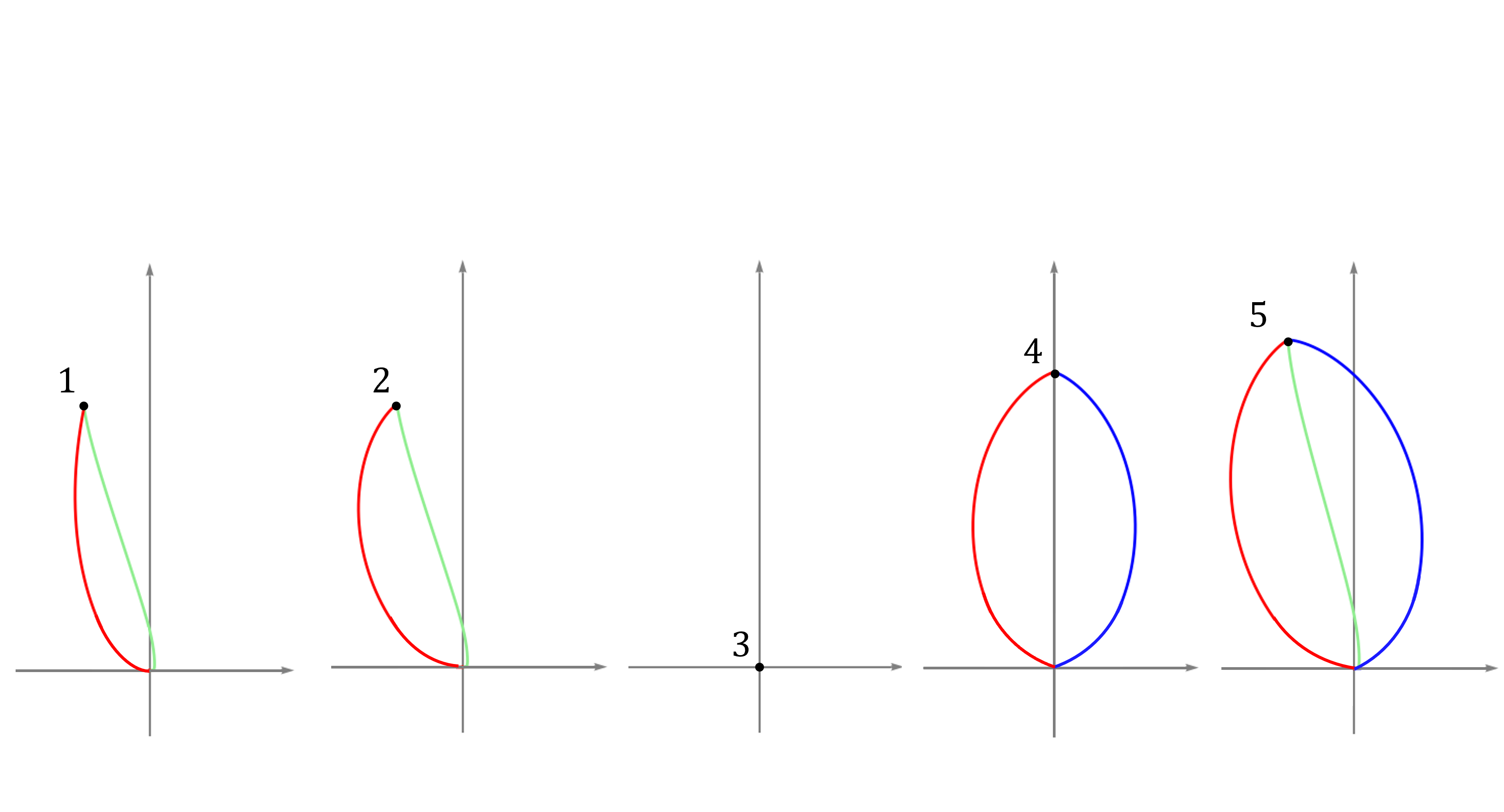}}
    \caption{
    Spheres of increasing radius $R$ around $[\bp_0]$ of the $d_{\mathrm{c}}$ model:
    the optimal fronts departing from $\bp_0=(0,0,0)$ (green), from $\overline{\bp}_0=(0,0,\pi)$ (red) and from $\overline{\bp}_0=(0,0,-\pi)$ (blue) together form the $d_{\mathrm{c}}$ sphere of radius $R$ plotted from $\theta \in [-\pi/2,\pi/2]$.
    The points where two or more colors meet are Maxwell points.
    In comparison to the $d_{\mathrm{proj}}$ model in Fig.~\ref{fig:Maxwellproj}, we see that for radii $R < \pi$ there are considerable differences, for large radii $R \geq \pi$ we see %more  
    similarity. Below we plot the set of multiple same-length geodesics ending up in the Maxwell point indicated by 1-5 (only 3 is not a Maxwell point, but a limit of Maxwell points).
    }   
    
    % In green we see the optimal fronts departing from $\bp_0=(0,0,0)$, in red the optimal fronts departing from $\overline{\bp}_0=(0,0,\pi)$ and in blue the optimal fronts departing from $\overline{\bp}_0=(0,0,-\pi)$ which together form the $d_{\mathrm{c}}$ sphere of radius $R$.

    % Intersection of the Maxwell-set of the $d_{\mathrm{c}}$ model with the $d_{\mathrm{c}}$ sphere 
    % % $\{[\bp]\in \R^2 \times P_1 \;|\; d_{\mathrm{c}}([\bp],[\bp_0])=R\}$ 
    % with increasing radius $R$, and with $\theta \in [-\pi/2,\pi/2]$. 
    % In green we see the optimal fronts departing from $\bp_0=(0,0,0)$, in red the optimal fronts departing from $\overline{\bp}_0=(0,0,\pi)$ and in blue the optimal fronts departing from $\overline{\bp}_0=(0,0,-\pi)$ which together form the $d_{\mathrm{c}}$ sphere of radius $R$.
    % In comparison to the $d_{\mathrm{proj}}$ model in Fig.~\ref{fig:Maxwellproj} we see that for small radii $R < \pi$ there are considerable differences, for large radii $R \geq \pi$ we see more  similarity. 
    \label{fig:Maxwellmon} 
    
\end{figure*}

\begin{figure*}
\centering
\includegraphics[width=0.75\hsize]{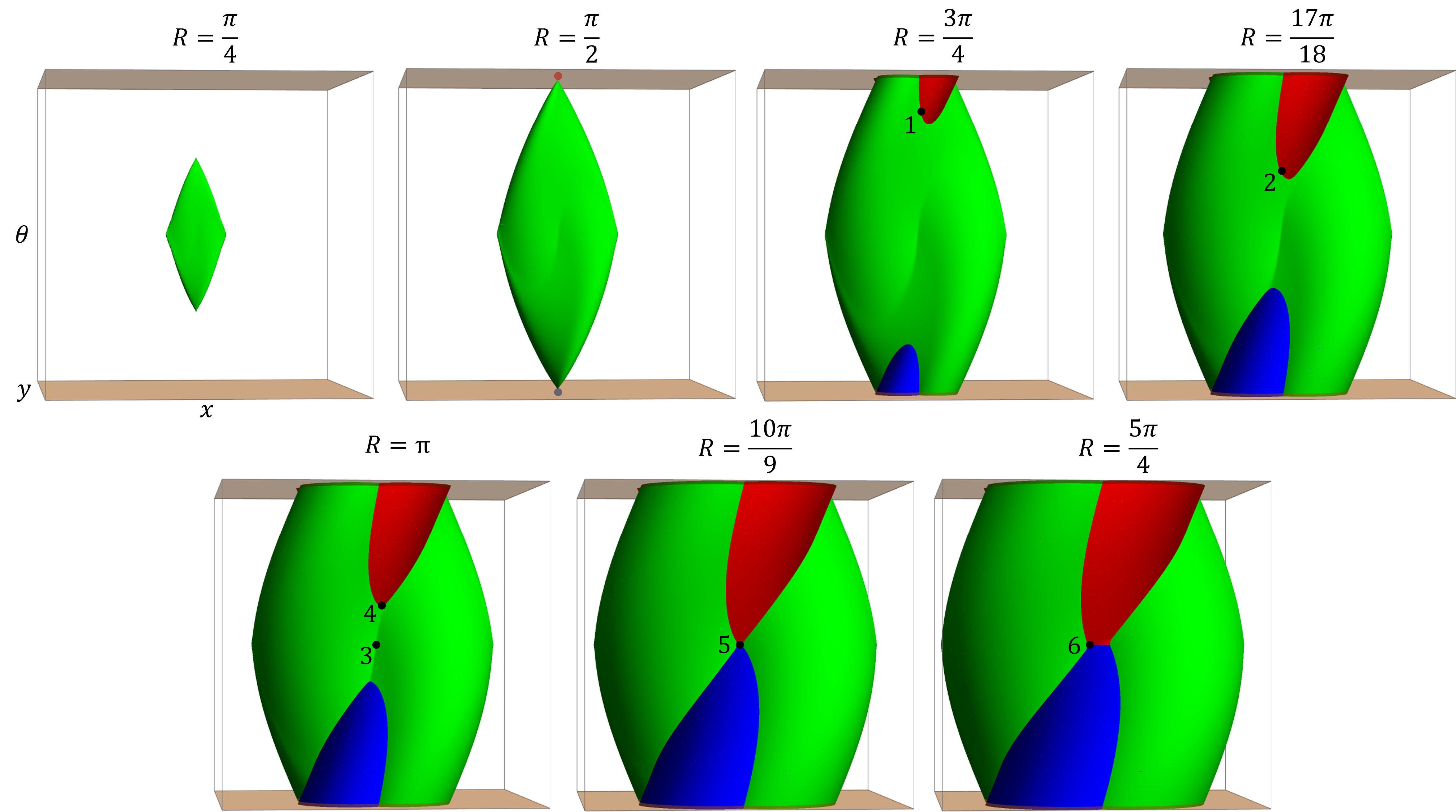}
\includegraphics[width=0.8\hsize]{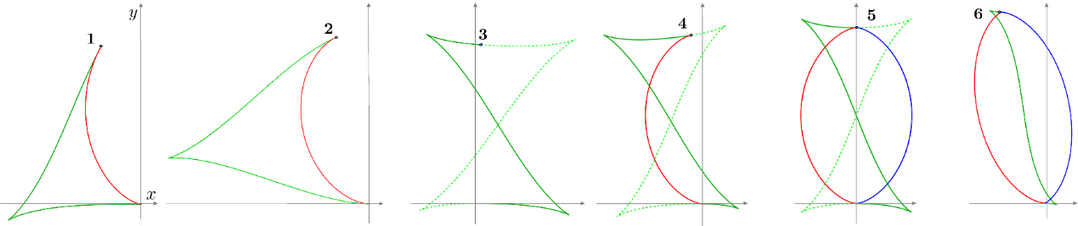}
\caption{
Spheres of increasing radius $R$ around $[\bp_0]$ of the $d_{\mathrm{proj}}$ model:
the optimal fronts departing from $\bp_0=(0,0,0)$ (green), from $\overline{\bp}_0=(0,0,\pi)$ (red) and from $\overline{\bp}_0=(0,0,-\pi)$ (blue) together form the $d_{\mathrm{proj}}$ sphere of radius $R$ plotted from $\theta \in [-\pi/2,\pi/2]$.
The points where two or more colors meet as well as the green fold are Maxwell points.
% In comparison to the $d_{\mathrm{c}}$ model in Fig.~\ref{fig:Maxwellmon} we see that for small radii $R < \pi$ there are considerable differences, for large radii $R \geq \pi$ we see more  similarity.
% Figure published previously in \cite{BekkersGSI} 
% depicting the progression of the intersection of the Maxwell set with increasing radii $R>0$ in $d_{\mathrm{proj}}$ (visible as folds). 
% We recall this figure (with permission) only for direct qualitative comparison purposes to Fig.~\ref{fig:Maxwellmon} where the progression of the Maxwell set is plotted of $d_{\mathrm{c}}$. 
% For the relation of model $d_{\mathrm{c}}$ to model $d_{\mathrm{proj}}$ and their different properties see Theorem~\ref{th:main}\label{fig:Maxwellproj}. 
% We see considerable differences for $R<\frac{\pi}{2}$ and rather similar for $R>\frac{\pi}{2}$. Parameter $\nu$ denotes the number of distinct geodesics departing from the origin $[\bp_0]$ meeting with the same length in the indicated points. 
% The critical radii $\bar{R}\approx \frac{17}{18}\pi$ and  $\tilde{R}\approx \frac{10}{9} \pi$ where resp. the front departing from $(0,0,\pm\pi)$ hits the local component of the Maxwell set, and where 4  geodesics with the same length meet are computed in \cite{BekkersGSI}.
Below we plot the set of multiple same-length geodesics ending up in the maxwell point indicated by the numbers 1-6.
Reproduced from Figs.~2 and 3 in \cite{BekkersGSI}.
}
\label{fig:Maxwellproj}
\end{figure*}

\section{Connected-Component- Informed Cost Function}\label{Section: connected-component-informed cost function}
%The connected component algorithm of \cite{berg2024connectedcomponentsliegroups} is a combination of a line-segmentation method and a grouping algorithm

\subsection{Cost Function \texorpdfstring{\unboldmath $\mathcal{C}$}{C} for the Data-Driven \texorpdfstring{\unboldmath $d_{\mathrm{c}}$}{dc} model}
To incorporate image information into the geodesic tracking model $d_{\mathrm{c}}$, we introduce a cost function $\mathcal{C}(\cdot)$, see Eq.~\eqref{Eq:CostFunction}. This means that whenever $\mathcal{C} \neq 1$, the model becomes data-driven. In this work, we employ the model $d_{\mathrm{c}}$ to track the edges in the image. The cost function $\mathcal{C}$ should be low at the edges and high everywhere else. We therefore define the data term $\mathcal{C}: \R^2 \times S^1 \rightarrow [\delta, 1]$ as follows:
\begin{align}\label{Eq:CostFunction}
    \mathcal{C}(\bp) := \frac{1}{1+\lambda|\mathcal{V}(\bp)|^p}
\end{align}
with $\lambda, p > 0$ and $0<\delta := 1 / (1 + \lambda)<1$. The edge detector $\mathcal{V}: \R^2\times S^1 \rightarrow [0, 1]$ measures the presence of edges, and is given by the crossing-preserving line filter \cite{BekkersGSI, BekkersSIAM, berg2024connectedcomponentsliegroups} applied to the modulus of the imaginary part of the orientation score \eqref{def: orientation score}.  

\subsection{The Connected-Component- Informed Cost Function \texorpdfstring{\unboldmath $\mathcal{C}_{\mathrm{new}}$}{Cnew}}
Next, we introduce a grouping strategy for the cost function $\mathcal{C}$, informed by the connected component algorithm of \cite{berg2024connectedcomponentsliegroups}, to reduce the risk of a geodesic switching between the edges of distinct components. 
This motivates the introduction of a new cost function, $\mathcal{C}_{\mathrm{new}}$, as described below.

The connected component algorithm groups elements into $n$ equivalence classes $\{K_i\}_{i = 1}^n$, as defined in \cite[Def.~5]{berg2024connectedcomponentsliegroups}, based on a left-invariant Riemannian distance with $i \leq n$ and $n$ being the number of connected components. The domain of the cost function will be grouped into new, expanded, equivalence classes $\tilde{K}_i$ based on the shortest distance to a connected component.

Specifically, $\bp \in \R^2 \times S^1$ belongs to the expanded equivalence class $\tilde{K}_i$, if it is closest to the connected component $K_i$, indexed by $i \in \{1, \ldots, n\}$:
\begin{equation}\label{Eq: equivalence relation}
\bp \in \tilde{K}_i \Leftrightarrow
    \left\{ \begin{array}{ll}
        & \forall_{ j  \in \{1, \ldots, n\}}\\
        &d(\bp, K_i) \leq d(\bp, K_j)
    \end{array}\right\},
\end{equation}
where $d(\bp, K_j) = \inf_{\bq\in K_j} d_{\mathcal{G}}(\bp, \bq)$ for the Riemannian distance $d_{\mathcal{G}}$, as given in \eqref{Eq: Riemannian distance}. The metric parameters $g_{11},g_{22}$ and $g_{33}$ \eqref{Eq: sub-Riemannian metric tensor} define the shape of the kernel. 
%and therefore exert a strong influence on the outcome of the grouping process.

For each equivalence class $\tilde{K}_i$, we can then define the corresponding grouped cost function:
\begin{equation}\label{Eq:CostFunctionC_K}
\mathcal{C}_{\tilde{K}_i}(\bp) =
\left\{
\begin{array}{ll}
\mathcal{C}(\bp) &\textrm{ if } \bp \in \tilde{K}_i,  \\
1 & \textrm{else}.
\end{array}
\right.
\end{equation}

Then the final connected-component-informed cost function is defined by
\begin{equation} \label{CNEW}
\mathcal{C}_{\mathrm{new}}(\bp, \bp_1)
=
\left\{
\begin{array}{ll}
\mathcal{C}_{\tilde{K}_1}(\bp) &\textrm{if }\bp_1 \in \tilde{K}_1, \\
\vdots & \vdots \\
\mathcal{C}_{\tilde{K}_n}(\bp) &\textrm{if }\bp_1 \in \tilde{K}_n.
\end{array}
\right.
\end{equation}
Thereby, the connected component of the end point $\bp_1$ determines which (connected-component-informed) cost function to choose in the steepest descent backtracking \eqref{eq:backtracking}. 

% The to be grouped set $J$ is a binarized version of the cost function $\mathcal{C}$;
% \begin{equation*}
%     J:=\{\bp\in \R^2 \times S^1 \;| \; \mathcal{C}(\bp) \leq T \}.
% \end{equation*}
% Consistent with the approach of the connected component algorithm, the grouping is performed by dilation operations on $\R^2 \times S^1$. Suppose that the connected component algorithm produces $K \geq 1$ components. Every connected component is given by an equivalence class $[\bp_{k}]$, as defined in \cite[Def.~5]{berg2024connectedcomponentsliegroups}, with $k\leq K$. Its corresponding binary image is given by;
% \begin{align*}
%     U_k(\bp) &= \mathbbm{1}_{[\bp_{k}]}(\bp) =  \begin{cases}
%         1  &\text{if }\bp\in [\bp_{k}]\\
%         0 &\text{else}
%     \end{cases}
% \end{align*}

\subsubsection{Algorithm}

The (grouping) procedure 
to compute each $\mathcal{C}_{\tilde{K}_i}$ for $\mathcal{C}_{\mathrm{new}}$ \eqref{CNEW}
is given in Algorithm~\ref{Algorithm: grouping of C}, and relies on morphological dilations with Riemannian-distance-$d_{\mathcal{G}}$-based kernels, as explained in Appendix~\ref{App:C}.

This algorithm relies on a morphological dilation (denoted by $\square$) Eq. \eqref{eq:dil} using a kernel $k_\tau$ of size $\tau$, as defined in Eq.~\eqref{Eq: kernel dilation}. The metric parameters $g_{11},g_{22}$ and $g_{33}$ \eqref{Eq: sub-Riemannian metric tensor} define the shape of the kernel $k_\tau$. As the scaling of $\tau>0$ can be absorbed by the scaling of $\{g_{ii}\}$, we set $\tau=1$ and $k_\tau=k_1$ in the algorithm.

\begin{algorithm}
\small
\caption{The Connected-Component-Informed Cost Function $\mathcal{C}_{\tilde{K}_i}$}\label{Algorithm: grouping of C}
\begin{algorithmic}[1]
\Require $n \geq 1$
\Procedure{$\mathbf{Group}$}{$\mathcal{C},\{K_i\}_i^n, k_1$}
\State $I \gets \; \{\bp \in \R^2 \times S^1 \mid \mathcal{C}(\bp) \leq 0.1 \}$\Comment{initialize}
\For{$i=1$ to $n$} 
    \State $I_{\tilde{K}_i} \gets \; \emptyset$ 
    \State $U_{\tilde{K}_i} \gets \indbb{K_i}$\;
\EndFor    
\While{$I \neq \emptyset$}\Comment{the grouping of set $I$}
    \For{$i=1$ to $n$}
        \State $W_{\tilde{K}_i} \gets \; -(k_1\square-U_{\tilde{K}_i})$   
        \State $I_{\tilde{K}_i} \gets I_{\tilde{K}_i} \cup \left(\mathrm{supp}(W_{\tilde{K}_i}) \cap I\right) $
        \State $U_{\tilde{K}_i} \gets W_{\tilde{K}_i}$
    \EndFor
    \State $I \gets I \setminus \bigcup_{i=1}^n I_{\tilde{K}_i}$
\EndWhile
\For{$i=1$ to $n$}  \Comment{defining $\mathcal{C}_{\tilde{K}_i}$} 
    \State $\mathcal{C}_{\tilde{K}_i} \gets -(k_1\square-\indbb{ \{I_{\tilde{K}_i}\}})\cdot \mathcal{C}$
    \State \textbf{return} $\mathcal{C}_{\tilde{K}_i}$
\EndFor
\EndProcedure
\Statex
\end{algorithmic}
  \vspace{-0.3cm}%
\end{algorithm}

\section{Automatic Switch from Geodesics to Spatial Snakes}\label{section: switching criterion}
Snake modeling is an iterative procedure in which an initial contour, derived from the connected component algorithm of \cite{berg2024connectedcomponentsliegroups}, is progressively adapted to align with the edges of the image.
In our snake model, two different models are used to refine the initial placement of the contour, as illustrated in the flow diagram of Fig.~\ref{fig: flowchart}: (1) the accurate geodesic tracking model $d_{\mathrm{c}}$ defined on the projective line bundle $\R^2 \times P^1$ (see Step 4a of Fig.~\ref{fig: flowchart}) and (2) a fast spatial snake model (Step 4b).

The geodesic tracking  model tracks the edges in the higher-dimensional space $\R^2 \times P^1$. However, processing the lifted data in $\R^2 \times P^1$ is inherently more computationally demanding than restricting operations to $\R^2$. To address this issue, we aim to reduce the computations in the higher-dimensional space to those strictly necessary, thereby improving efficiency while minimizing accuracy loss.

\subsection{Spatial Snakes} %: Edge Refinement}
As the name suggests, the spatial snake model refines the initial contour only spatially and \emph{not} angularly. Here, we propose to use an existing edge detection model in $\R^2$, applied to the real part of a fixed orientation layer $\mathrm{Re}( U_f(\cdot, \theta) ): \R^2 \to \mathbb{\R}$ of the orientation score $U_f: \SE(2) \to \mathbb{C}$ given by \eqref{def: orientation score}.

We will be using the following edge detection model: the norm of the Gaussian gradient $G_{\mathrm{norm}}(\cdot): \mathbb{L}_2(\R^2) \rightarrow \mathbb{L}_2(\R^2)$ given by
$$G_{\mathrm{norm}}(\cdot):=\| \nabla(G_s \ast \cdot)\|,$$
where $G_s$ is the Gaussian function with scale parameter $s \in \R_{>0}$. The selection of the scale $s>0$ is automatic, where we first follow Lindeberg's scale selection principle \cite{lindeberg1998edge}, and then apply the subsequent standard edge-focusing method of \cite{BartTeHaarRomenyFront}.

Let $\gamma = (\bx(\cdot), \bn(\cdot)) \in \Gamma(\bp_0,\bp_0)$ be an initial contour for the spatial snake model. The location $\bx(t)$ of the edge at time $t$ is refined in the direction perpendicular to $\bn(t)=\left(\cos(\theta(t)),\sin(\theta(t)\right)$, given by  $\hat\bn(t)=(-\sin(\theta(t)),\cos(\theta(t))$. We update by: %defined as:
\begin{equation*}
    \bx_{\mathrm{new}}(t):= \bx(t) +  \lambda^{\mathrm{loc}}(t) \hat\bn(t),
\end{equation*}
where the scalar displacement $\lambda^{\mathrm{loc}}(t)$ is given by
\begin{align*}
    \lambda^{\mathrm{loc}}(t) :=\underset{\lambda \in \R}{\operatorname{argmax}}^{*} \; \left(G_{\mathrm{norm}}(F)\right) \left(\bx(t)+\lambda \hat\bn(t) \right).
\end{align*}
with fixed orientation layer $F(\cdot)=\textrm{Re}(U_f(\cdot, \theta(t))): \R^2 \to \R$. Here $\textrm{argmax}^{*}$ takes the $\lambda = \lambda^{\mathrm{loc}}(t)$ corresponding to the local maximum of the spatial gradient norm of $F$ along the scan-line $\textrm{span}\{\hat \bn(t)\}$ that is closest to $\bx(t)$.  

\begin{figure*}
\centering
    \begin{subfigure}[T]
    {0.35\linewidth}
        \centering
        \includegraphics[width=0.8\linewidth, trim={0 11 0 10}, clip]{images/ChipWaferASML.png}
        \caption{}
    \end{subfigure}
    \begin{subfigure}[T]{0.35\linewidth}
        \centering
        \includegraphics[width=0.8\linewidth, angle=-90, trim={100 55 100 55}, clip]{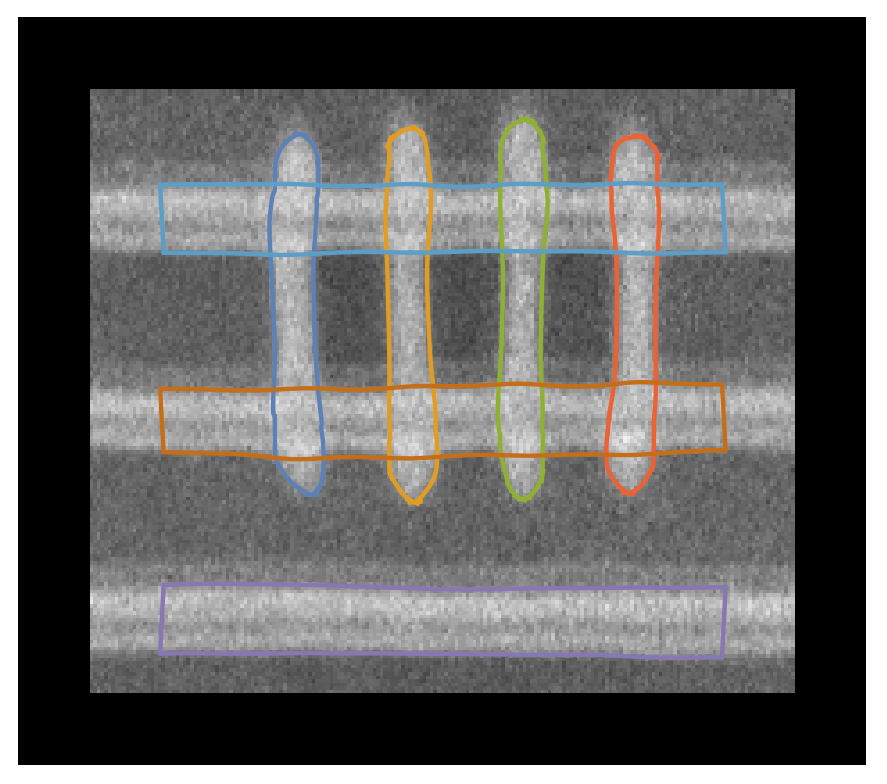}
        \caption{}
    \end{subfigure}
    \caption{Result of the proposed algorithm applied to an SEM image of FinFET processing step. The contour of each component is indicated with a distinct color.}
    \label{Fig:ResultsFinFET}
\end{figure*}

\subsection{Switching Criterion \label{ch:sc}}
By construction of the orientation score transform \eqref{def: orientation score}, edges in $\R^2$ are lifted to horizontal curves in $\R^2 \times S^1$. 
For the definition of a horizontal curve, see Def.~\ref{def:Gamma}. 
As a result, the initial contour $\gamma(t) = (\bx(t), \bn(t)) \in \Gamma(\bp_0, \bp_0)$ at time $t$ will only be located in the same orientation layer as the lifted edge, when the contour is locally horizontal. 
If this is indeed the case, it is meaningful to refine the contour in the direction $\hat{\bn}(t)$ on the corresponding $\theta(t)$-layer. 
In all other cases, an angular displacement is required, in which case we rely on the geodesic tracking model.

This observation motivates the following switching criterion: the model switches to the spatial snake model whenever the initial curve is nearly horizontal. Formally, this condition holds if the `deviation from horizontality' \cite[Part II]{DuitsAMS} given by $\phi(t) \in C([0,1],[0,\pi])$ of the initial curve $\gamma(t)\in \Gamma(\bp_0,\bp_0)$ satisfies
\begin{equation}\label{Eq:switching criterion}
\phi(t):=\arccos{\frac{\ul{\dot{x}}(t)\cdot \bn(t)}{\|\ul{\dot{x}}(t)\|}} \leq \alpha,
\end{equation}
for a threshold $\alpha \in (0,\pi]$.

\subsubsection{Contour Completion}
The initial contour, obtained from the connected component algorithm of \cite{berg2024connectedcomponentsliegroups}, generally forms a closed continuously differentiable curve in $\R^2 \times S^1$ (for the appropriate settings). Consequently, the switching criterion $\phi(t) - \alpha$ in Eq.~\eqref{Eq:switching criterion} alternates between a nonpositive and a positive sign an even number of times. 

Suppose that the sign switches at times $t_0, t_1 \in [0,1]$, and suppose that $\phi(t) \leq \alpha$ for the interval $t \in [t_0,t_1]$. In this situation, the placement of the initial contour will be refined using spatial snakes for $t \in [t_0,t_1]$. The rest of the contour can be completed with the geodesic tracking model using $\gamma_{\mathrm{new}}(t_0)=(\bx_{\mathrm{new}}(t_0),\bn(t_0))$ as the sink $\bp_0$ and $\gamma_{\mathrm{new}}(t_1)=(\bx_{\mathrm{new}}(t_1),\bn(t_1))$ as the source $\bp_1$.

\section{Experiments}\label{section: experiment}
In this section, we evaluate the performance of our snake model on the projective line bundle qualitatively and quantitatively using both experimental and synthetic SEM images. We begin by describing the experimental setup, including the parameter settings of the algorithm. Subsequently, we present the results obtained by applying the proposed algorithm to the experimental and synthetic data, compared to the previous tracking model on the projective line bundle \cite{BekkersGSI}.

\subsection{Experimental Setup}\label{section: experiment setup}
The steps of the algorithm are visualized in Fig.~\ref{fig: flowchart}. The first two steps are the connected component algorithm of \cite{berg2024connectedcomponentsliegroups}. 
The connected component algorithm groups elements into $n$ equivalence classes $\{K_i\}_{i = 1}^n$, based on a left-invariant Riemannian distance. 
%with %metric parameters 
%$\{g_{11}, %g_{22}, g_{33}\} %= \{0.3, 0.8, %5\}$. \eqref{metric}

The connected component algorithm starts by lifting the data to $\R^2 \times S^1$ using the orientation score transform. The orientation score transform includes standard cake wavelet settings: we used $N_o=12$ orientations with overlap factor $o=4$ so that each wavelet covers a $\frac{2\pi}{12}$ cone. % in the Fourier domain. 
%We have set: $N_o=12$ and $o=4$.
%Additionally, there is a crossing-preserving line filter, which has a number of parameters. 
%The metric parameters of the line filter are given by $\{g_{11}, g_{22}, g_{33}\} = \{1, 1, 1\}$. 
The crossing-preserving line filter detects lines at spatial and angular scales $\sigma_s = \frac{3}{4}\cdot w$ and $\sigma_a=2\cdot \frac{2\pi}{48}$, with $w$ the structure's width based on an isotropic metric. We have set $\lambda=50$ and $p=3$ in \eqref{Eq:CostFunction}.

The initial contour for the snake model is derived from the output of the connected component algorithm, see Appendix \ref{app:initial_contour}.

The flowchart includes two edge refinement methods: a fast spatial snake model and an accurate geodesic tracking model (Step 4a and 4b of Fig.~\ref{fig: flowchart}, respectively). 
The algorithm automatically switches from the spatial snake model to the geodesic tracking model when the initial contour deviates from being horizontal, that is, when the deviation exceeds the angle $\alpha= 3 \cdot \frac{2 \pi}{48}$. 
The switch is visualized in Step 3 of Fig.~\ref{fig: flowchart}. The parts that are updated by the geodesic tracking model are colored red. The remaining parts (blue parts) are refined by the spatial snake model, as visualized in Step 4b. 

The cost function $\mathcal{C}$ \eqref{Eq:CostFunction} of the geodesic tracking model is defined by the modulus of the imaginary part of the orientation score with a crossing-preserving line filter. For the orientation score transform, we used the same settings as for the connected component algorithm. For the line filter, we used the settings: $\sigma_s=2$, $\sigma_a=4\cdot \frac{2\pi}{48}$, $\lambda=100$, $p=1$.
%This crossing-preserving line filter includes several parameters. %The metric parameters of the line filter are given by $\{g_{11}, g_{22}, g_{33}\} = \{1, 1, 1\}$. 

The grouping of the cost function $\mathcal{C}_{\mathrm{new}}$ is visualized in Step 4a of Fig.~\ref{fig: flowchart}, where every group has a different color.
The cost function $\mathcal{C}$ is grouped by Algorithm~\ref{Algorithm: grouping of C} with %metric parameters 
$\{g_{11},g_{22},g_{33}\}=\{0.2,1,7\}$. 
For the geodesic tracking model \eqref{Eq: finsler function}, we must set $\xi \approx \sqrt{g_{11}/g_{33}}$ and $\mathcal{C}=\mathcal{C}_{\mathrm{new}}(\cdot, \bp_1)$.

\subsection{Experiment 1 %(Experimental Data)
}
For the first experiment, we applied the algorithm to an experimental SEM image of a FinFET processing step, as shown in Fig.~\ref{fig: SEM example}. The width $w$ of the structures is approximately 16 pixels. Since the data are experimental, the results are presented qualitatively as figures.

\subsubsection{Results}
The results of Experiment 1 are shown in Fig.~\ref{Fig:ResultsFinFET}. By lifting the image to $\R^2 \times S^1$, we can successfully detect the edge-boundaries of overlapping structures.
Fig.~\ref{Fig:ResultsFinFET} further demonstrates that only a spatial refinement suffices, where the initial contour deviates slightly from being horizontal (the blue parts in Step 3 of Fig.~\ref{fig: flowchart}). 
The geodesic tracking model manages to refine the remaining parts (the red parts in Step 3 of Fig.~\ref{fig: flowchart}).

% \subsection{Experiment 2}
% For Experiment 2, we applied the algorithm to a synthetic SEM image of a processing step of a DRAM and the resulting contours are presented in a figure.

% \subsubsection{Experiment Setup}
% For the connected component algorithm, we used the following settings. 
% For the orientation score transform, 48 orientations were used with an overlap factor of 1. We reduced the overlap factor for a better separation between the orientation layers, as the angle between horizontal and diagonal components is smaller than in Experiment 1.
% For the line filter, the parameters were set to $\sigma_a = 1 \cdot \frac{2\pi}{48}$, $\sigma_s = \frac{3w}{4}$, $\lambda = 50$, and $p = 2$, where the width $w$ of the components is approximately 4 pixels.

% For the line filter of the cost function $\mathcal{C}$ as given in \eqref{Eq:CostFunction}, we used the following settings. The parameters were set to: $\sigma_a=2*\frac{2\pi}{48}, \sigma_s=2, \lambda=100, p=1$.

% \subsubsection{Results}
% Fig.~\ref{Fig:ResultsDRAM}

% \begin{figure}
%     \begin{subfigure}[T]{0.45\linewidth}
%         \centering
%         % \includegraphics[width=\linewidth, trim={0 11 0 10}, clip]{images/}
%     \end{subfigure}
%     \begin{subfigure}[T]{0.45\linewidth}
%         \centering
%         % \includegraphics[width=\linewidth]{images/}
%     \end{subfigure}
%     \caption{Result of the proposed algorithm applied to an synthetic SEM image of processing step of a DRAM. The contour of each component is indicated with a distinct color.}
%     \label{Fig:ResultsDRAM}
% \end{figure}
\begin{figure}[t]
        \centering
        \includegraphics[width=0.7\linewidth]{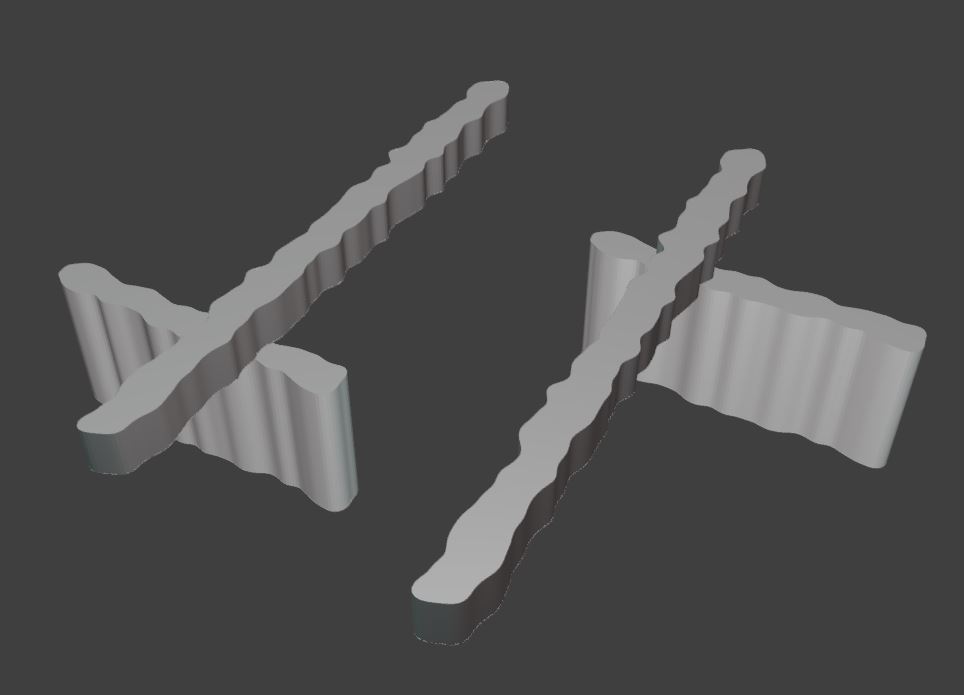}
        \caption{Example of geometry used for the generation of the synthetic data set with Nebula.}
        \label{fig:synthetic_SEM_geometry}
\end{figure}

\subsection{Experiment 2} 
%(Synthetic Data)}
For Experiment 2, we generated a data set consisting of 12 images using the Monte Carlo Simulator Nebula \cite{van2020nebula}. 
Nebula requires a 3D geometry as input. 
The geometry used in our experiments contains two shorter lines at the bottom and two longer lines at the top.
The width of the electronic structures in nanometers (nm) in the data set varies by $w = \{8, 12, 16\}$. 
The height of the top layer strongly influences the visibility of the bottom layer: a thicker top layer results in reduced visibility of the structures below. 
In the data set, the height of the top layer in nanometers varies by $h = \{5, 15, 25, 40 \}$. To make the data more realistic and the contouring task less trivial, roughness is added to the edges of the electronic structures. 
An example of the input geometry for Nebula is shown in Fig.~\ref{fig:synthetic_SEM_geometry}: 
the goal is to find the contours of this geometry. 

For the simulation of the SEM images, we used a 1:1 scale, which means that one nanometer corresponds to one pixel.
The corresponding SEM image and its ground-truth contouring, derived from the input geometry for Nebula (Fig.~\ref{fig:synthetic_SEM_geometry}), are shown in Fig.~\ref{fig: Example synthetic SEM}.

In Experiment 2, we applied our algorithm to the data set and compared its performance with that of the previous geodesic tracking model of \cite{BekkersGSI}. To be precise, we replaced only the geodesic tracking model in Step 4a of Fig.~\ref{fig: flowchart} with the method of \cite{BekkersGSI}, which includes the model $d_{\mathrm{proj}}$ and the (ungrouped) cost function $\mathcal{C}$. 
The rest of the pipeline, shown in Fig.~\ref{fig: flowchart}, remained the same.
The results of this experiment are presented both quantitatively and qualitatively, in the form of plots and figures.

\begin{figure}[t]
\centering
\hfill
    \begin{subfigure}[T]{0.45\linewidth}
        \centering
        \includegraphics[width=\linewidth]{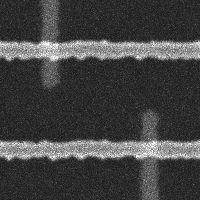}
        \caption{Example of synthetic SEM image with overlapping structures.}
        \label{fig:synthetic_SEM}
    \end{subfigure}
    \hfill
    \begin{subfigure}[T]{0.45\linewidth}
        \centering
        \includegraphics[width=\linewidth]{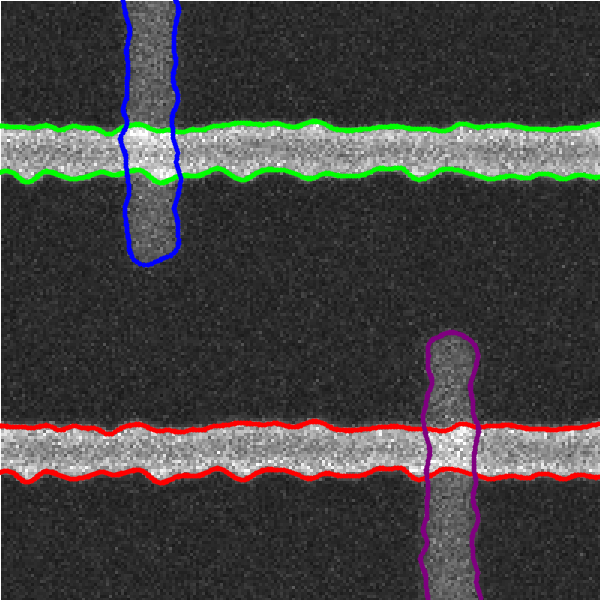}
        \caption{Visualization of the ground-truth contouring of the four different structures, indicated by different colors.}
        \label{fig:synthetic_sem_gt_segmentation}
    \end{subfigure}
    \hfill~
    \caption{Example of synthetic SEM image of the data set generated with Nebula and its ground-truth contouring.}
    \label{fig: Example synthetic SEM}
\end{figure}

\begin{comment}
{\color{red} \subsubsection{Data preparation}
In Experiment 2, we first preprocessed the data by taking the logarithm of the image. Subsequently, an additional normalization step was introduced. The images were lifted to $\R^2 \times S^1$ with the orientation score transform. Then followed by an $L_1$-normalization on each orientation layer. The image was projected by summing over the orientation layers. This sequence was executed twice before performing the final lift. These preprocessing steps effectively compensate for the low intensity of the components against the dark background, enhancing the visibility of the relevant structures.}
\end{comment}

\subsubsection{Evaluation Metrics}
In this section, we will describe the two metrics used for the analysis of the second experiment. 

Let $\gamma_{\mathrm{out}}$ be the spatial projection of the output of the algorithm; then $\gamma_{\mathrm{out}}$ is a closed curve in $\R^2$. Let $\gamma_{\mathrm{gt}}$ be the closed curve in $\R^2$ representing the ground-truth contouring. 
In order to measure the accuracy of the method, we use the following two metrics:
\begin{enumerate}
    \item the Mean Average Surface Distance between $\gamma_{\mathrm{out}}$ and $\gamma_{\mathrm{gt}}$: 
    \begin{align}\label{Eq:MASD}
    \begin{split}
        d_{\textrm{MASD}}&(\gamma_{\mathrm{out}},\gamma_{\mathrm{gt}}):=\\
        \frac{1}{2} \Big( &\frac{1}{L(\gamma_{\mathrm{out}})}\int \limits_0^{L(\gamma_{\mathrm{out}})}d(\gamma_{\mathrm{out}}(s),\Gamma_{\mathrm{gt}})\; \mathrm{d}s +\\
        &\frac{1}{L(\gamma_{\mathrm{gt}})}\int \limits_0^{L(\gamma_{\mathrm{gt}})} d(\gamma_{\mathrm{gt}}(s),\Gamma_{\mathrm{out}})\mathrm{d}s\Big),
    \end{split}
    \end{align}
     with $L(\gamma)$ the length of the \emph{spatial} curve $\gamma \in \{\gamma_{out},\gamma_{gt}\}$, and $s$ spatial arc length.   
    \item the Hausdorff distance between $\gamma_{\mathrm{out}}$ and $\gamma_{\mathrm{gt}}$:
    \begin{align}\label{Eq:Hausdorff}
        \begin{split}
            &d_{\textrm{HD}}(\gamma_{\mathrm{out}},\gamma_{\mathrm{gt}}):=\\
            &\max \{ \sup_{x\in \Gamma_{\mathrm{out}}}d(x,\Gamma_{\mathrm{gt}}),\sup_{y\in \Gamma_{\mathrm{gt}}}d(y,\Gamma_{\mathrm{out}}) \}.
        \end{split}
    \end{align}
 
\end{enumerate}
Here, the curves $\gamma_{\mathrm{gt}}$ and $\gamma_{\mathrm{out}}$ have been parameterized using the arc length. The sets $\Gamma_{\mathrm{gt}}=\{\gamma_{\mathrm{gt}}(s)\mid s \in [0,L(\gamma_{\mathrm{out}})]\}$ and 
$\Gamma_{\mathrm{out}}=\{\gamma_{\mathrm{out}}(s) \mid s \in [0,L(\gamma_{\mathrm{out}})]\}$ are associated with the curves $\gamma_{\mathrm{gt}}, \gamma_{\mathrm{out}}$.
Finally, $d(x,Y)=\inf_{y \in Y} d(x,y)$, with the Euclidean distance $d(x,y)=\|x-y\|$ in $\R^2$.    

\begin{figure*}
\centering
    \begin{subfigure}[T]{0.3\linewidth}
        \centering
        \includegraphics[trim={60 60 60 60}, clip, angle=-90, width=\linewidth]{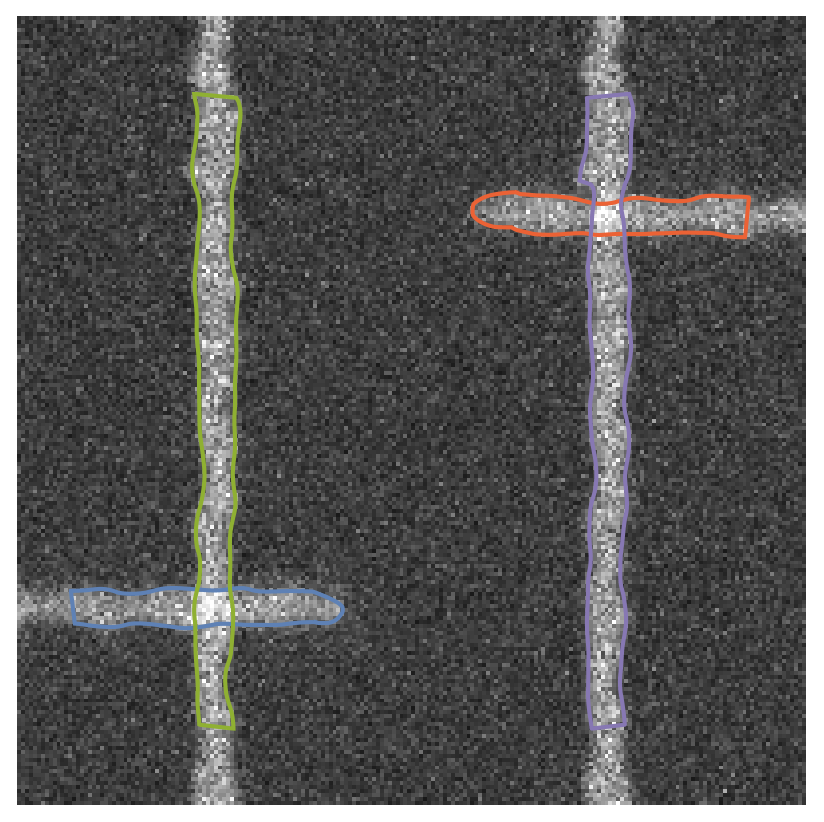}
        \caption{}
        \label{fig: subfig qualitative_results cd 8}
    \end{subfigure} 
\hfill
    \begin{subfigure}[T]{0.3\linewidth}
        \centering
        \includegraphics[trim={75 75 75 75}, clip, angle=-90, width=\linewidth]{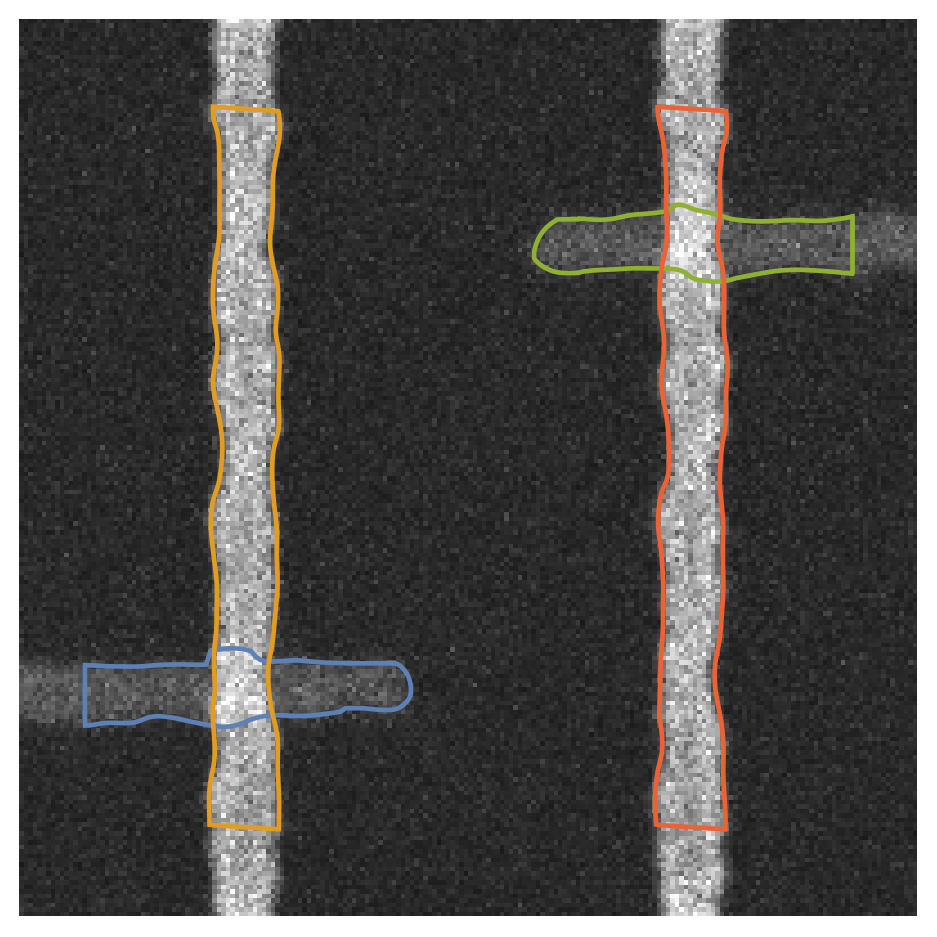}
        \caption{}
        \label{fig: subfig qualitative_results cd 12}
    \end{subfigure}  
\hfill
    \begin{subfigure}[T]{0.3\linewidth}
        \centering
        \includegraphics[trim={155 155 155 155}, clip, angle=-90, width=\linewidth]{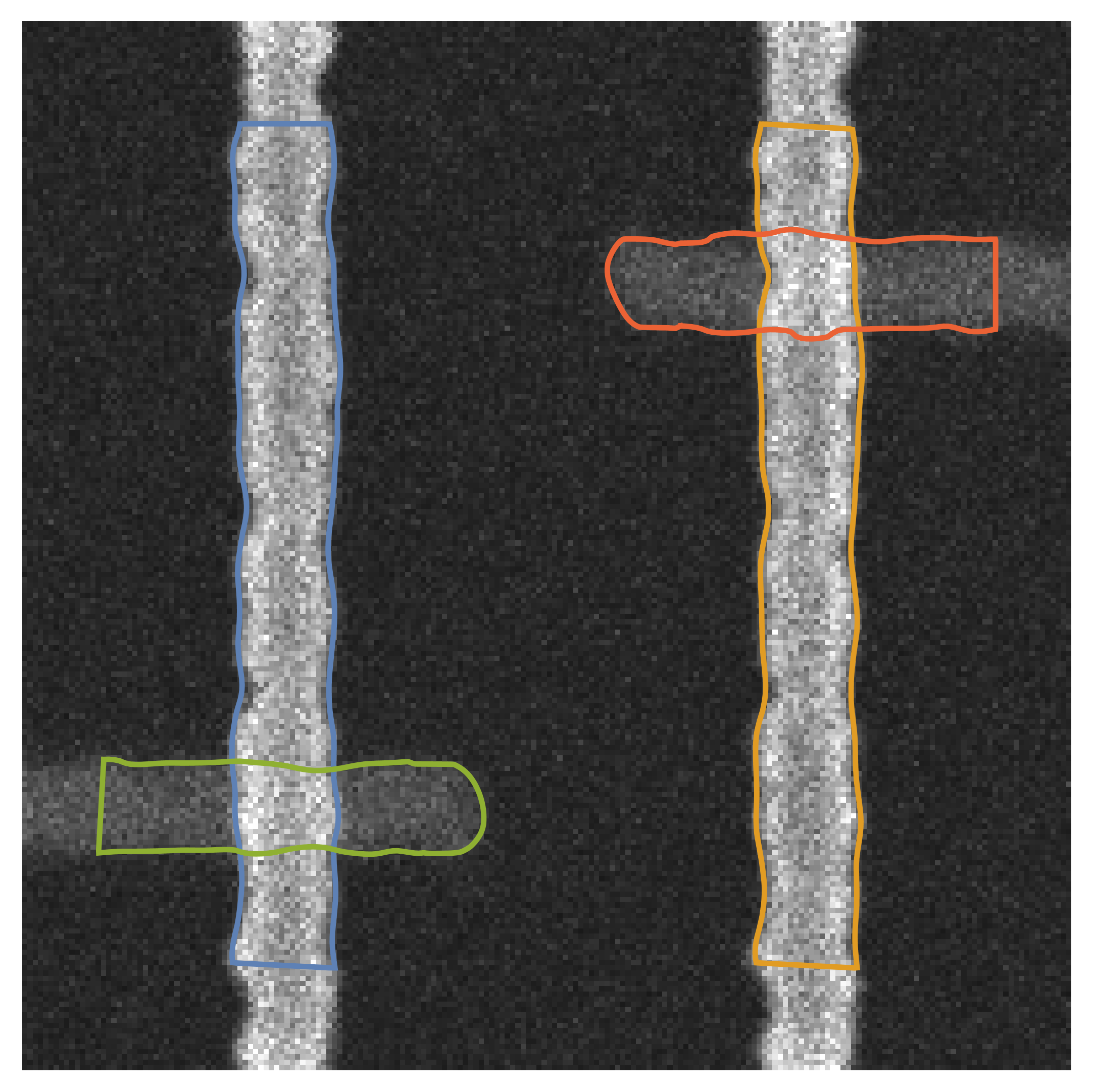}
        \caption{}
        \label{fig: subfig qualitative_results cd 16}
    \end{subfigure}
    \caption{Output of our method (Fig.~\ref{fig: flowchart}) with $d_{\mathrm{c}}$, Eq.~\eqref{dmon}, and $\mathcal{C}_{\mathrm{new}}$, Eq.~\eqref{CNEW},
    with variations in width and height. The width and height of the electronic structures increases from left to right.}
    \label{fig:qualitative_results}    
\end{figure*}

\begin{figure*}
\centering
    \begin{subfigure}[T]{0.3\linewidth}
        \centering
        \includegraphics[trim={62 62 62 62}, clip, angle=-90, width=\linewidth]{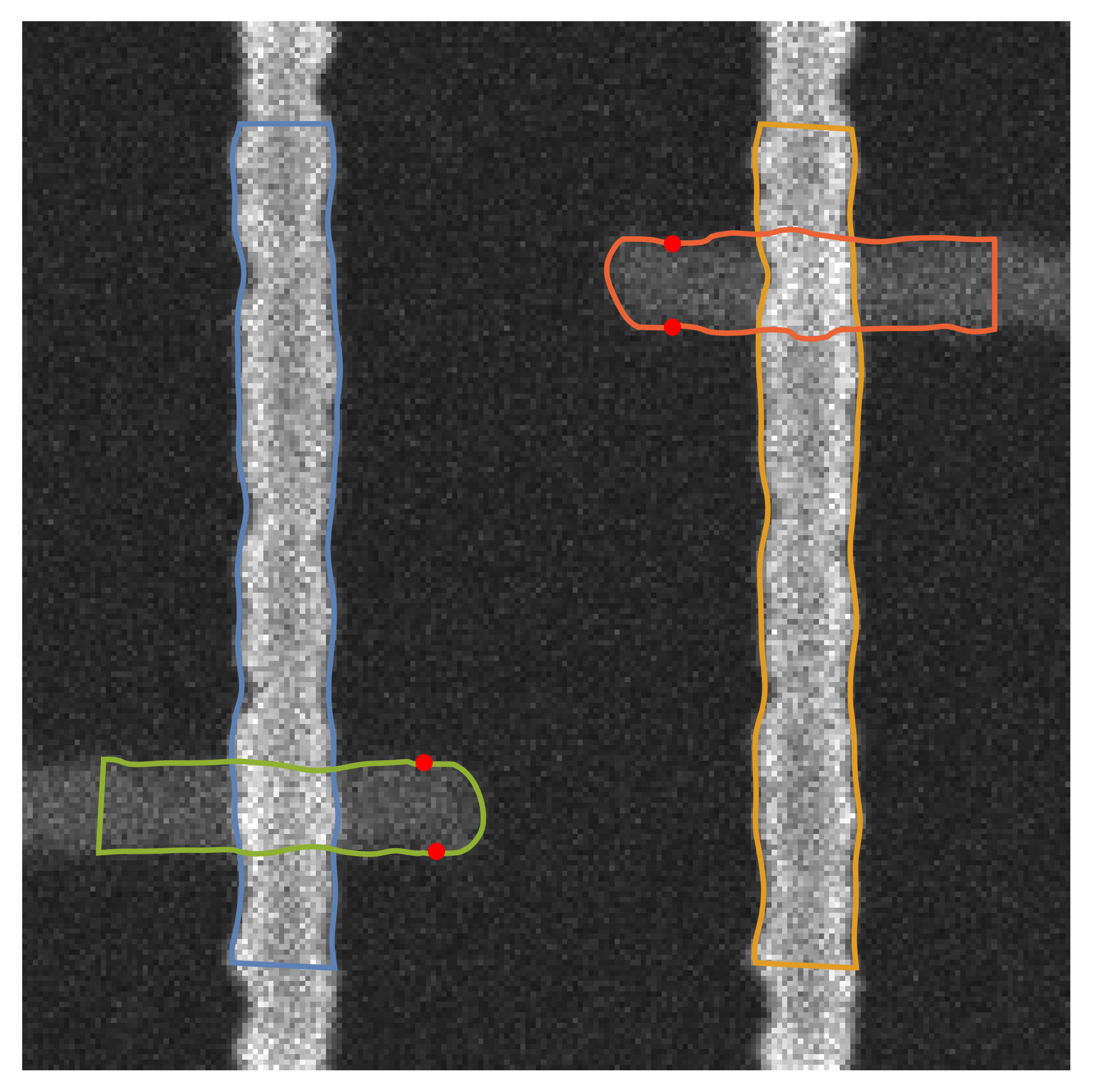}
        \caption{}
        \label{Fig: sub fig, results grouped cost}
    \end{subfigure}    
    %\hfill
    \begin{subfigure}[T]{0.3\linewidth}
        \centering
        \includegraphics[trim={62 62 62 62}, clip, angle=-90, width=\linewidth]{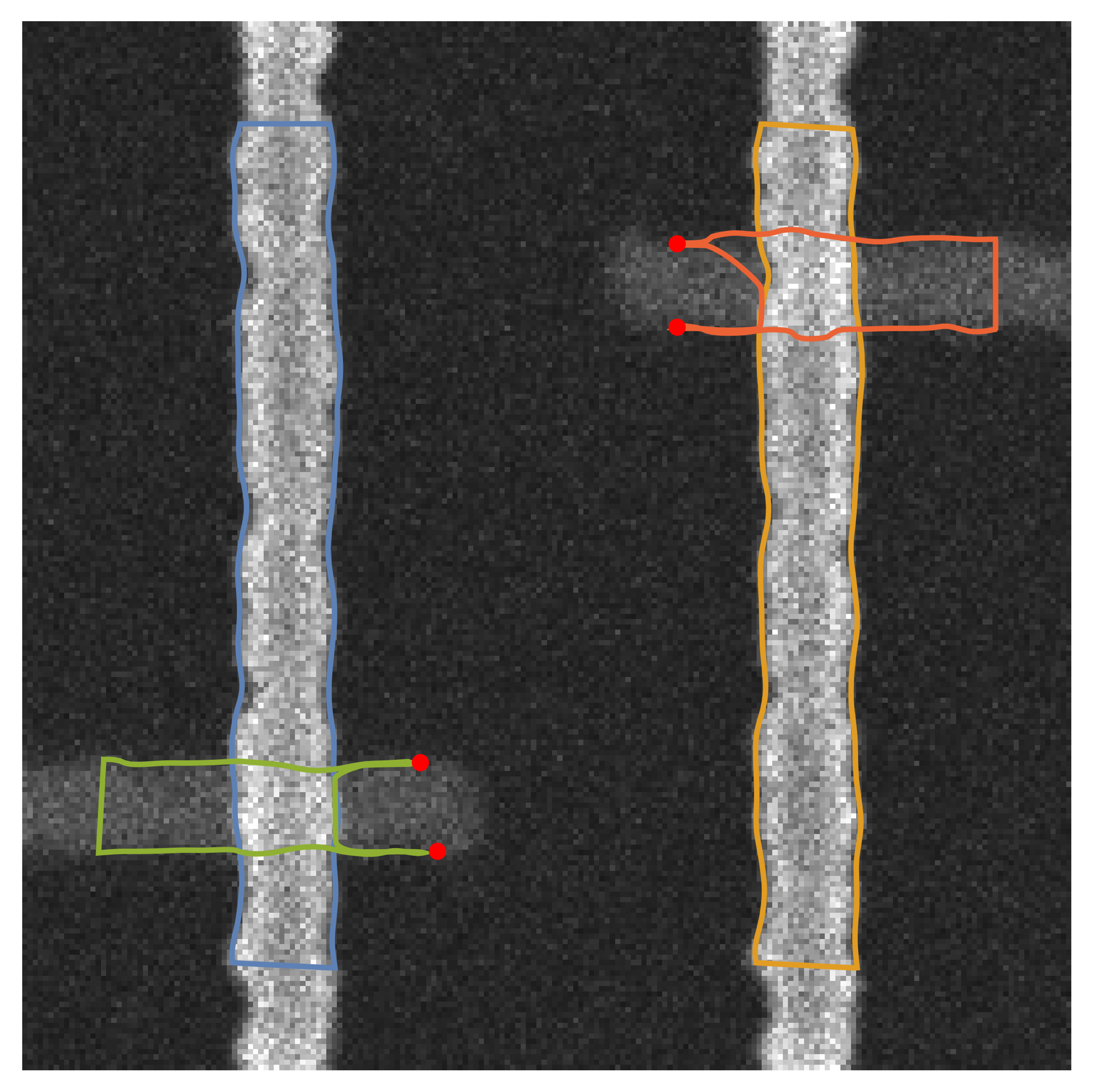}
        \caption{}
        \label{Fig: sub fig, results ungrouped cost}
    \end{subfigure}    
\caption{Comparison: (a) our method (Fig.~\ref{fig: flowchart}) with $d_{\mathrm{c}}$ and the connected-component-informed cost function $\mathcal{C}_{\mathrm{new}}$, Eq.~\eqref{CNEW}, and (b) our method with $d_{\mathrm{c}}$ and the previous cost function $\mathcal{C}$, Eq.~\eqref{Eq:CostFunction}. The sinks and sources of the geodesic tracking model are indicated with red dots. Even though the geodesic tracking model is cusp-free, cusps may still occur within the spatial snake model and during the switch between the spatial snake and geodesic tracking model.}
\label{Fig: results ungrouped cost function}
\end{figure*}

\begin{figure*}
\centering
    \begin{subfigure}[T]{0.3\linewidth}
        \centering
        \includegraphics[trim={62 62 62 62}, clip, angle=-90, width=\linewidth]{images/Experiments/CD16Height40sinktarget.png}
        \caption{}
    \end{subfigure} 
    %\hfill
    \begin{subfigure}[T]{0.3\linewidth}
        \centering
        \includegraphics[trim={62 62 62 62}, clip, angle=-90, width=\linewidth]{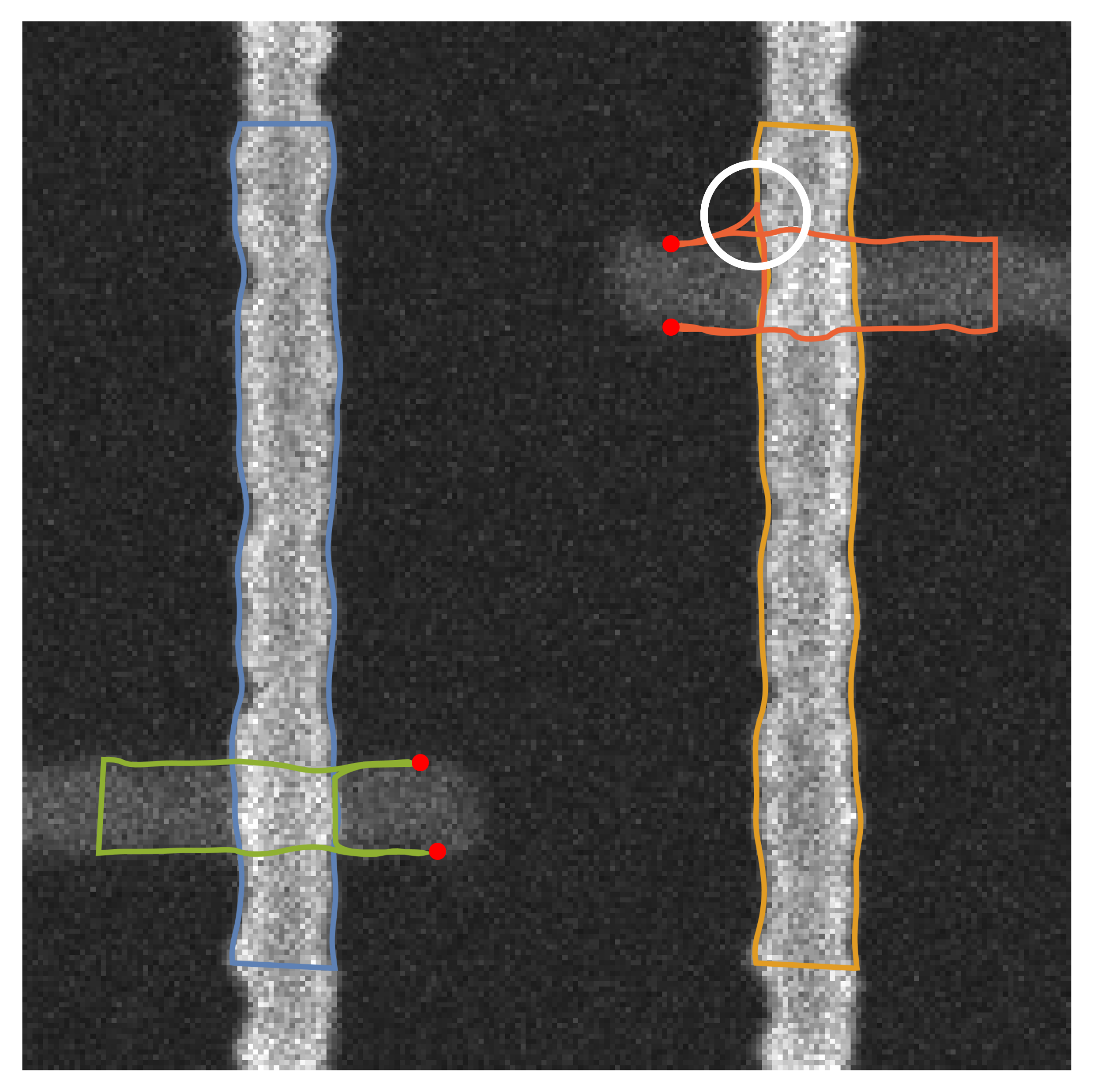}
        \caption{}
        \label{Fig: subfig results ERIK}
    \end{subfigure}
    \caption{Comparison: (a) our method (Fig.~\ref{fig: flowchart}) with $d_{\mathrm{c}}$ and $\mathcal{C}_{\mathrm{new}}$, (b) previous tracking models $d_{\mathrm{proj}}$ from \cite{BekkersGSI}. The sinks and sources of the geodesic tracking model are indicated with red dots.
    The red track in the previous model exhibits a cusp in this common practical setting.}
    \label{Fig: results ERIK}
\end{figure*}

% \begin{figure}
%         \centering
%         \includegraphics[width=\linewidth]{images/Results3SEMimages.png}
%         \caption{Result of our single method applied to three different SEM images where we only adapted cost function parameters $p\in[1,3],\sigma_\alpha \in[1 \Delta\theta,4\Delta\theta]$ where $\Delta\theta=\frac{2\pi}{48}$, and $\sigma_s = \frac{3w}{4}$ where $w$ is the width of the components.}
% \end{figure}

\subsubsection{Qualitative \& Quantitative Results}
Figure~\ref{fig:qualitative_results} shows the contouring results for three of the twelve synthetic SEM images. 
These results demonstrate that operating in $\R^2 \times P^1$ enables accurate edge detection, even when structures overlap.

The grouping of the cost function of the geodesic tracking model mitigates the risk of the geodesic jumping between edges of different components. 
In Fig.~\ref{Fig: results ungrouped cost function}, we compare the result of our method using the (grouped) cost function $\mathcal{C}_{\mathrm{new}}$ with the (ungrouped) cost function $\mathcal{C}$. 
Fig.~\ref{Fig: sub fig, results ungrouped cost} shows that when the ungrouped cost function $\mathcal{C}$ is used, geodesics tend to switch between the edges of the horizontal and vertical components. 
In contrast, when using $\mathcal{C}_{\mathrm{new}}$, the geodesic consistently follows the edge of the correct component, as visualized in Fig.~\ref{Fig: sub fig, results grouped cost}.

The primary advantage of our new geodesic model is that it is cusp-free, in contrast to the model of \cite{BekkersGSI}, which can produce cusps. 
Figure~\ref{Fig: results ERIK} illustrates the results obtained when our geodesic tracking method is replaced by the model of \cite{BekkersGSI}, while all other steps in the flowchart of Fig.~\ref{fig: flowchart} remained the same. 
As seen in the bottom-right vertical electronic structure of Fig.~\ref{Fig: subfig results ERIK}, the method of \cite{BekkersGSI} produces a geodesic that contains a cusp. 
Furthermore, since this method uses the ungrouped cost function $\mathcal{C}$, the resulting geodesics jump between the edges of different components.

Figures~\ref{fig: subfig qualitative_results cd 8}, ~\ref{fig: subfig qualitative_results cd 12} and~\ref{fig: subfig qualitative_results cd 16} show an undershoot of the curvature along the straight segments of the electronic structures. 
This effect is most noticeable for the brighter, horizontal structures. 
In these cases, the contours do not fully align with the curvature of the edges. 
This behavior is caused by the spatial snake model, which refines the initial contour only within a fixed orientation layer. 
The restriction to a fixed orientation layer reduces the model's adaptability to local curvature. 
For edges exhibiting significant roughness, the geodesic tracking model is preferable, as it can better adapt to variations in curvature.

The plot on the left of Fig.~\ref{Fig: MASD distance}, shows for each electronic structure the Mean Average Surface distance ~\eqref{Eq:MASD} between the contour obtained by our algorithm and the ground truth. For most structures, the error is less than one pixel.
If we compare our method with the previous tracking approach of \cite{BekkersGSI}, we observe in the right plot of Fig.~\ref{Fig: MASD distance} that the Mean Average Surface distance increases with the height of the top structures for those located at the bottom. 
This result can be explained, as the previous method uses the ungrouped cost function $\mathcal{C}$. 
The algorithm tends to favor edges with higher contrast, which are the edges of the electronic structures in the foreground. 
Consequently, the geodesic tracking model $d_{\mathrm{c}}$ tends to change during tracking from an edge in the background to an edge in the foreground. 
As the height increases, the intensity of the electronic structures in the background decreases, which then leads to an increase of the error.

The plot on the left of Fig.~\ref{Fig: Hausdorff distance} shows, for each electronic structure, the Hausdorff distance ~\eqref{Eq:Hausdorff} between the contour obtained by our algorithm and the corresponding ground-truth contour. 
For most electronic structures, the error is less than four pixels. 
However, there are some outliers. 
Fig.~\ref{Fig: effect noise} illustrates one of the outliers. 
As shown in Fig.~\ref{Fig: subfig input, effect noise}, the edges in the overlapping regions are faint. 
As a result, errors occurred during the edge refinement steps of the spatial snake model, see Fig.~\ref{Fig: subfig output, effect noise}.

If we compare our method with the previous tracking method~\cite{BekkersGSI}, we observe in the right plot of Fig.~\ref{Fig: Hausdorff distance} that the Hausdorff distance increases again with the height of the top structures for those located at the bottom. This behavior is similar to what we observed for the Mean Average Surface distance.

Our method is also stable under different widths of the electronic structures in the SEM images, in the sense that the Mean Average Surface and Hausdorff distance stay the same when varying the width, as can be observed in Figs.~\ref{Fig: MASD distance Width},~\ref{Fig: Hausdorff distance Width}.

\begin{figure*}
        \centering
        \includegraphics[width=\linewidth]{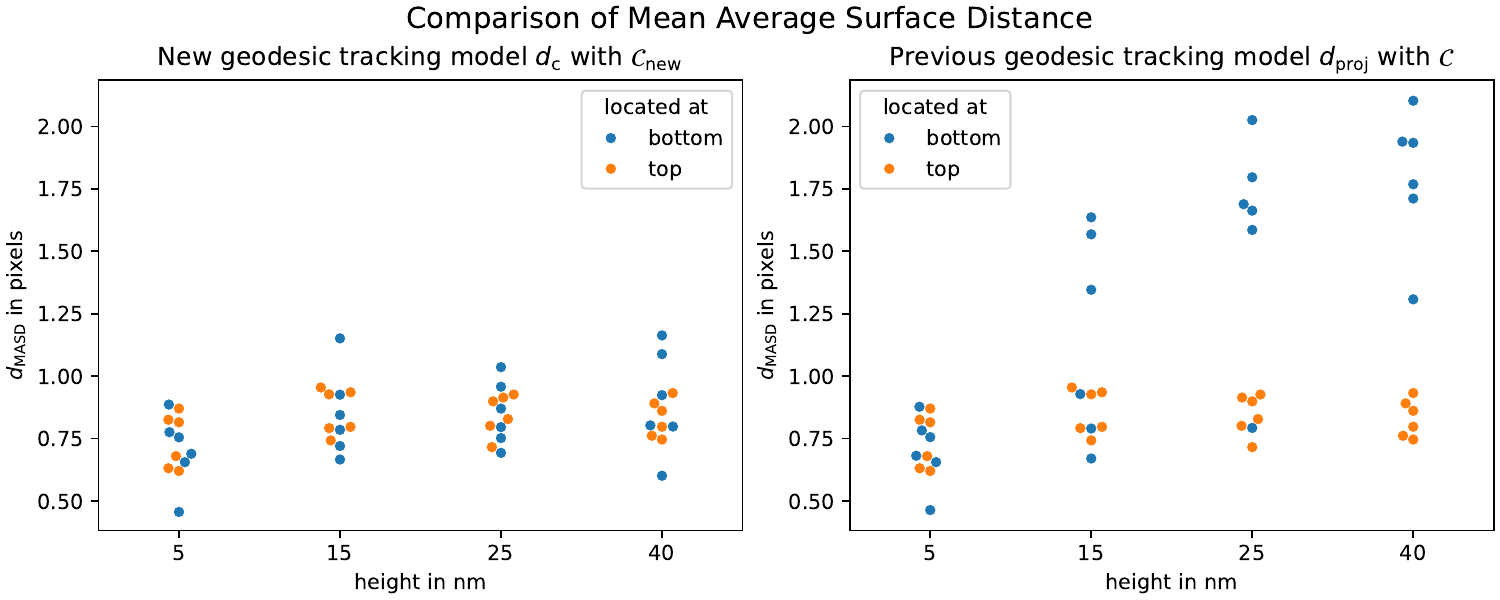}
        \caption{Comparison of the Mean Average Surface distance \eqref{Eq:MASD}: on the left our method (Fig.~\ref{fig: flowchart}) with $d_{\mathrm{c}}$ and 
        $\mathcal{C}_{\mathrm{new}}$, on the right previous tracking model $d_{\mathrm{proj}}$ with $\mathcal{C}$, cf.~\eqref{dproj} and
        \cite{BekkersGSI}. Every point corresponds to an individual electronic structure in a SEM images, sorted by the height of such a structure.
        }
        \label{Fig: MASD distance}
\end{figure*}

\begin{figure*}
        \centering
        \includegraphics[width=\linewidth]{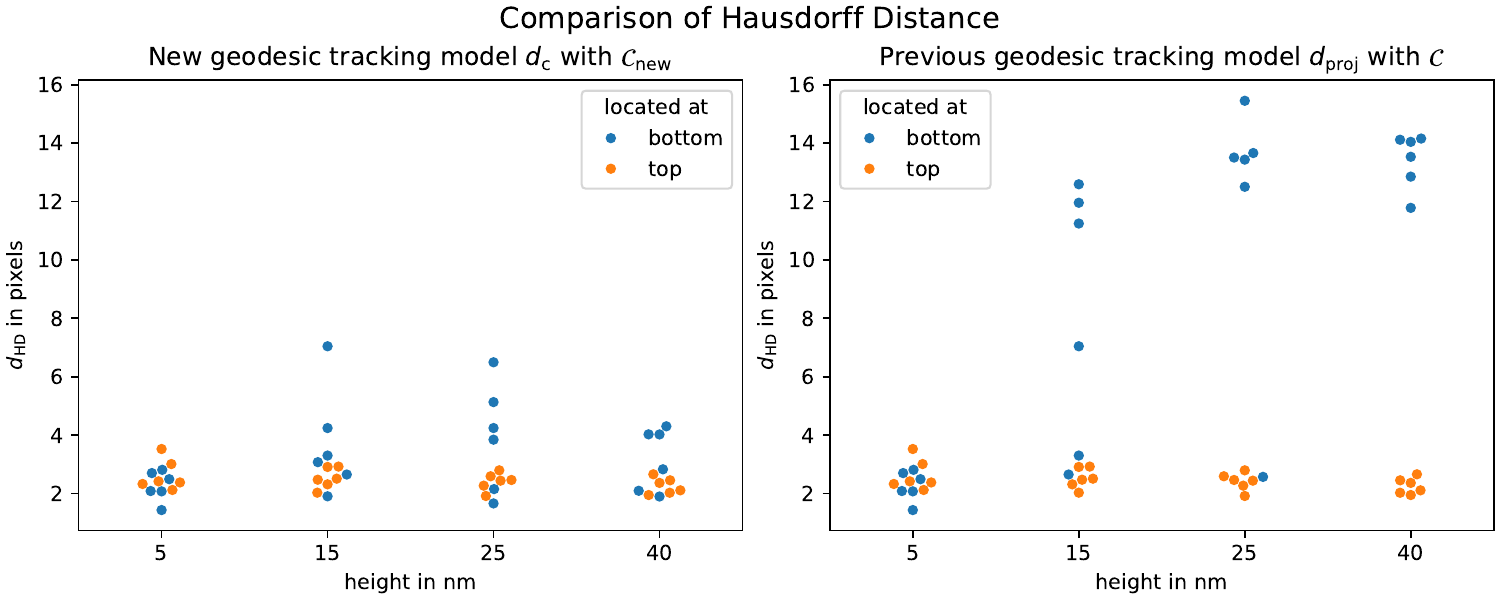}
        \caption{Comparison of the Hausdorff distance \eqref{Eq:MASD}: on the left our method (Fig.~\ref{fig: flowchart}) with $d_{\mathrm{c}}$ and $\mathcal{C}_{\mathrm{new}}$, on the right previous tracking model $d_{\mathrm{proj}}$ with $\mathcal{C}$, cf.~\eqref{dproj} and
        \cite{BekkersGSI}. Every point corresponds to an individual electronic structure in a SEM images, sorted by the height of such a structure.
%\cite{bekkers2018nilpotent}.
        }
        \label{Fig: Hausdorff distance}
\end{figure*}

% \begin{figure*}
%         \centering
%         \includegraphics[width=\linewidth]{images/Plot_Tracking_Models/datavisualcomparisonMASDWidth.png}
%         \caption{Comparison of the Mean Average Surface distance \eqref{Eq:MASD}: on the left our method (Fig.~\ref{fig: flowchart}) with $d_{\mathrm{c}}$ and 
%         %$\{\mathcal{C}_{\tilde{K}_i}\}_i^n$
%         $\mathcal{C}_{\mathrm{new}}$, on the right previous tracking models $d_{\mathrm{proj}}$, cf.~\eqref{dproj} and 
%         \cite{BekkersGSI}. Every point corresponds to an individual component of a geometry, sorted by the width of the components.
%         }
%         \label{Fig: MASD distance Width}
% \end{figure*}

% \begin{figure*}
%         \centering
%         \includegraphics[width=\linewidth]{images/Plot_Tracking_Models/datavisualcomparisonHWidth.png}
%         \caption{Comparison of the Hausdorff distance \eqref{Eq:MASD}: on the left our method (Fig.~\ref{fig: flowchart}) with $d_{\mathrm{c}}$ and $\mathcal{C}_{\mathrm{new}}$, on the right previous tracking models $d_{\mathrm{proj}}$, cf.~\eqref{dproj} and
%         \cite{BekkersGSI}. Every point corresponds to an individual component of a geometry, sorted by the width of the components.
% %\cite{bekkers2018nilpotent}.
%         }
%         \label{Fig: Hausdorff distance Width}
% \end{figure*}

\begin{figure}
\centering
    \begin{subfigure}[T]{0.45\linewidth}
        \centering
        \includegraphics[trim={155 155 155 155}, clip, width=\linewidth]{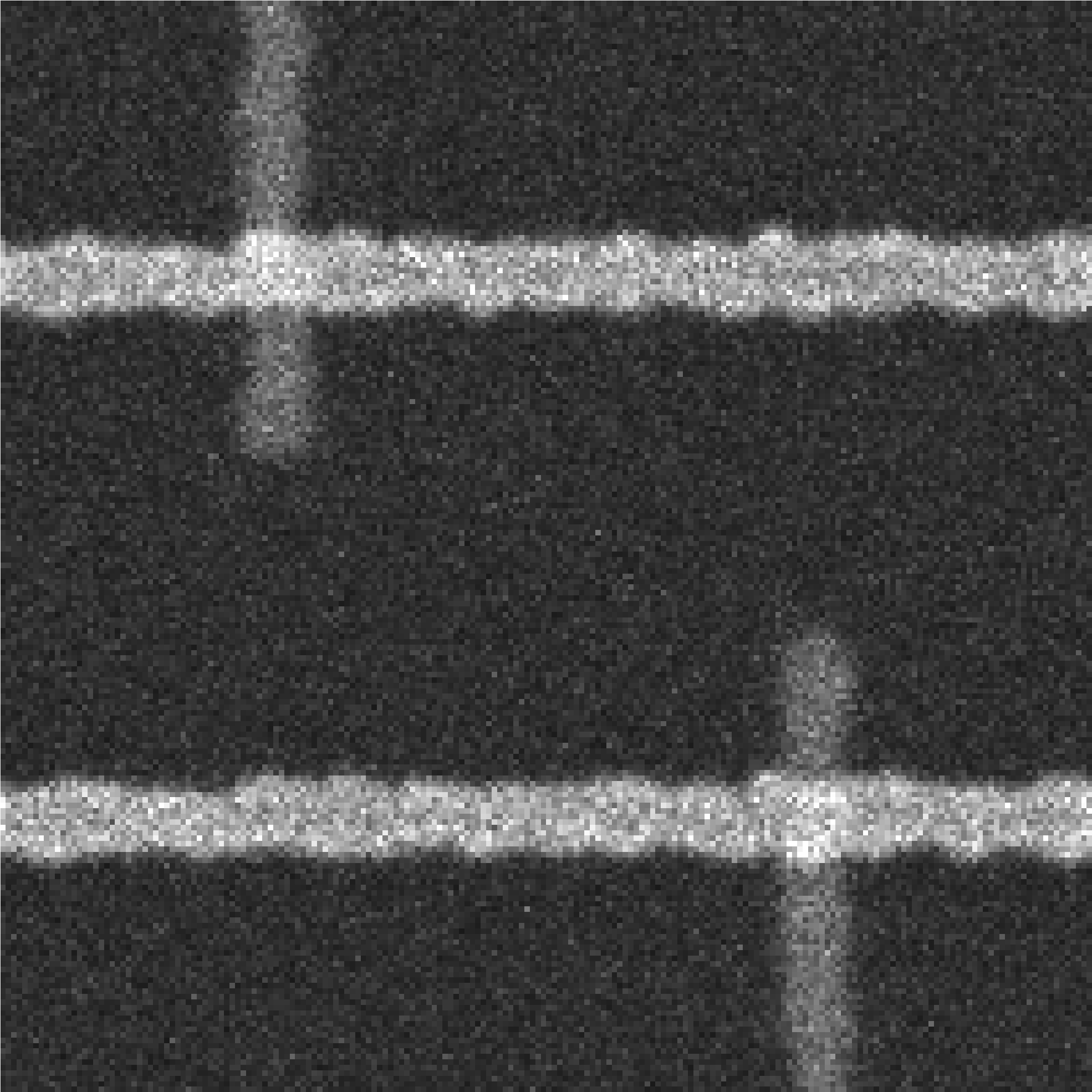}
        \caption{}
        \label{Fig: subfig input, effect noise}
    \end{subfigure} 
    \hfill
    \begin{subfigure}[T]{0.45\linewidth}
        \centering
        \includegraphics[trim={155 155 155 155}, clip, angle=-90, width=\linewidth]{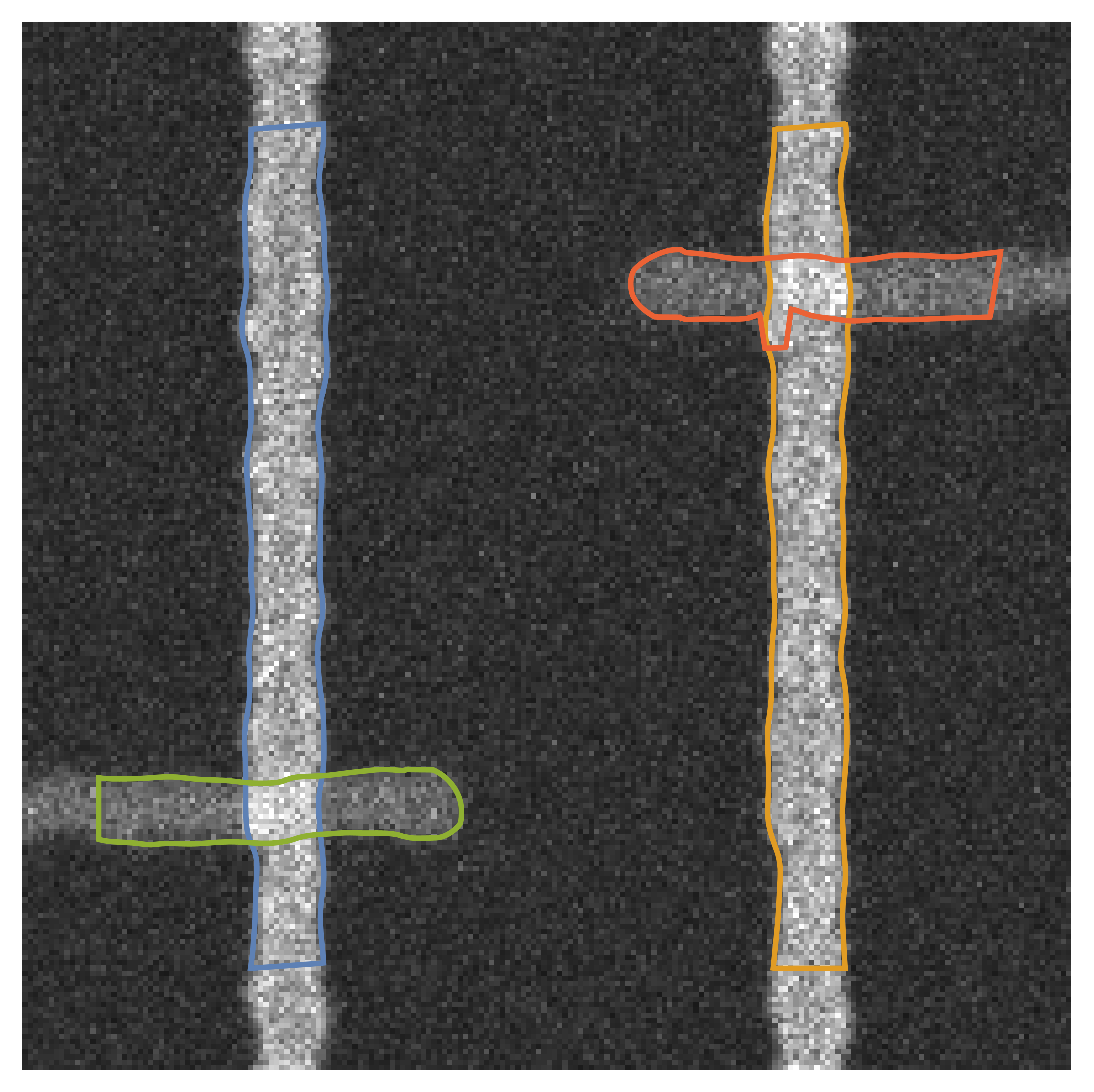}
        \caption{}
        \label{Fig: subfig output, effect noise}
    \end{subfigure}
    \caption{Example of an outlier in the Hausdorff distance for our method. Note in (a) that the edges of the overlapping structures are poorly visible. (b) shows that an error occurred during the edge refinement step.}
    \label{Fig: effect noise}
\end{figure}

\section{Conclusion}\label{section: conclusion}
In this article, we introduced a new metric %pseudo-distance 
$d_{\mathrm{c}}$ underlying a geodesic tracking model on the projective line bundle $\R^2 \times P^1$. The main advantage of this new distance $d_{\mathrm{c}}$ is that it is \emph{cusp-free}, in contrast to the previously introduced distance $d_{\mathrm{proj}}$ \cite{BekkersGSI} on $\R^2 \times P^1$. 

We demonstrated in Proposition~\ref{Corollary: sign-semidefinite distance} that the pseudo-distance $d_{\mathrm{c}}$ is well-posed and symmetric. In Theorem~\ref{th:main} we analyzed where $d_{\mathrm{c}}$ and $d_{\mathrm{proj}}$ coincide and where they differ.
The new model 
$d_{\mathrm{c}}$ satisfies the triangle inequality on a large subset $\tilde{\mathcal{Q}}_\xi$~\eqref{Eq: Tilde Qxi} that we characterized.

In Proposition~\ref{corr:tobewritten}, we further showed that $d_{\mathrm{c}}$ is globally controllable but not locally controllable. Consequently, the distance map $d_{\mathrm{c}}$ has a discontinuity at the origin. This is a desirable property as visualized in Fig.~\ref{fig:lat}: it avoids double cusps when moving in the lateral spatial direction. This difference is also reflected in a different Maxwell set \emph{for small radii} (compare Fig.~\ref{fig:Maxwellmon} to Fig.~\ref{fig:Maxwellproj}).

In Section~\ref{section: experiment}, we proposed a segmentation algorithm for detecting contours of electronic structures in SEM images. This algorithm incorporates the new geodesic tracking model and is summarized in the flowchart, as shown in Fig.~\ref{fig: flowchart}. In addition to the new tracking model, the algorithm includes two new relevant practical solutions:  
\begin{enumerate} 
\item
we introduced a switching criterion to minimize computations in the higher-dimensional space;  
\item
we demonstrated the advantage of integrating the connected component algorithm from \cite{berg2024connectedcomponentsliegroups} into our geodesic tracking model. By splitting the cost function $\mathcal{C}$ into $n$ cost functions $\{\mathcal{C}_{\tilde{K}_i}\}_i^n$  one for each component, cf. $\mathcal{C}_{\mathrm{new}}$ \eqref{CNEW}, we reduce the risk of jumping to the wrong connected component during tracking.
\end{enumerate}

We evaluated our algorithm quantitatively and qualitatively on both experimental and synthetic SEM images:
\begin{itemize}
\item
Qualitative results (Fig.~\ref{fig:introexpgen} and Fig.~\ref{fig:qualitative_results}) demonstrate that the algorithm successfully detects edges of overlapping structures on experimental and synthetic SEM images. 
\item
Quantitative results from Experiment~2 show that the Mean Average Surface Distance is less than one pixel (i.e. 1 nanometer) for most electronic structures (Fig.~\ref{Fig: MASD distance}), while the Hausdorff distance remains below four pixels (Fig.~\ref{Fig: Hausdorff distance}). 
\end{itemize}
%However, a few outliers were observed, primarily in cases where the edges of overlapping structures were faint. This led to inaccuracies during the edge refinement process.
\subsubsection*{Recommendations for Future Work}
The geodesic tracking model can handle curvature well, whereas the spatial snake model tends to undershoot the curvature. %However, 
The geodesic tracking model requires us to perform computations in $\R^2 \times P^1$, while for the spatial snake model the computations are done in $\R^2$, which is computationally less demanding. This reveals a trade-off between speed up and accuracy.

Regarding accuracy, we do observe a few outliers (in Fig.~\ref{Fig: Hausdorff distance}). They typically occur in cases where edges of overlapping structures are faint (Fig.~\ref{Fig: effect noise}). There fast spatial snake refinement is more risky. We propose to include edge-strength in the switching criterion. %(Subsect). 
%So, depending on the accuracy goal and or the roughness of your edges one could decide when to use which model.

%If we compare the performance of %both models, the geodesic tracking %model demonstrates that it can %handle curvature, whereas the %spatial snake model tends to %undershoot the curvature of the %edges. However, this improvement %comes at a computational cost: the %geodesic model operates in the %higher-dimensional space %$\R^2 \times P^1$, while the %spatial snake model performs its %computations in $\R^2$, %making it less computationally %demanding. Therefore, the choice %between the two approaches depends %on the desired balance between %computational efficiency and %accuracy. In particular, this %decision should consider the %roughness of the edges, as higher %roughness reduces the accuracy of %the spatial snake model.

We could eliminate the final parameters $(p,\lambda,\sigma_s, \sigma_\alpha, \xi=\sqrt{g_{11} / g_{33}})$, as discussed in Section~\ref{section: experiment setup} (also proposed in \cite{BekkersGSI, BekkersSIAM, berg2024connectedcomponentsliegroups}), by integrating 
an equivariant deep learning method (like \cite{smets2023pde}) to compute the cost function in Step~4a of Fig.~\ref{fig: flowchart}.

\subsection*{Acknowledgements}

The research in this article is primarily funded by the \emph{DELTAS} project (project nr. 24PPS053, ref nr. TKI-HTSM/240170,  2024-2029) with Eindhoven University of Technology TU/e and industrial partner ASML, within the PPS Innovation scheme (strategic program 23SP014 - Semiconductor Manufacturing Equipment), where we gratefully acknowledge the High Holland Tech organization for financial support.

The research is  partly supported by the project, \emph{Geometric Learning for Image Analysis}, R. Duits,  VICI Exact Sciences (nr. VI.C.202-031, 2021-2027), where we gratefully acknowledge 
The Dutch Foundation of Research (NWO) for financial support. 

We gratefully acknowledge Finn Sherry for the code
of the geodesic tracking models $d_{\cF_0}$ \eqref{Eq: finsler distance} and $d_{\cF_0^+}$ \eqref{Eq: finsler distance forward} \cite{berg2025crossing}.
We thank Tim Houben for the generation of the data set of synthetic SEM images. 
We thank Gijs Bellaard for his valuable suggestions and feedback on earlier versions of this manuscript.

%and would eliminate the final parameters in Table \ref{Table:CC} and Table \ref{Table:C_K}.

\begin{appendices}

\section{Model \texorpdfstring{\unboldmath $d_{\mathrm{c}}$}{dc} 
%on  \texorpdfstring{\unboldmath $\R^2 \times P^1$}{R2 x %P1}}
}
\label{app:MCP}

In this section, we provide the proofs of the key results (Lemma~\ref{Lemma: antipodal symmetry}, Propositions~\ref{Corollary: sign-semidefinite distance} \& \ref{corr:tobewritten}) underlying the interpretation of the sign-semidefinite distance $d_{\mathrm{c}}$ on the projective line bundle $\R^2 \times P^1$.

Note that the curves in the formulation of $d_{\mathrm{c}}$ in Props.~\ref{Corollary: sign-semidefinite distance} \& \ref{corr:tobewritten} are within $\R^2 \times S^1$, and not within $\R^2 \times P^1$ as their sign-semidefinite control is ill-posed on $P^1=S^1/_{\sim}$. 
Nevertheless, the resulting map $d_{\mathrm{c}}$ is well-posed on $\R^2 \times P^1$. %and by Theorem~\ref{th:main} it is a distance when constrained to $\mathcal{Q} \times \mathcal{Q}$.

\subsection{Proof of Lemma\texorpdfstring{~\ref{Lemma: antipodal symmetry}}{: antipodal symmetry}}\label{app:proofantisymmetry}
\begin{proof}
To prove Lemma~\ref{Lemma: antipodal symmetry}, we need to show that
\begin{align} \label{keyid}
    d_{\cF}(\bp_0,\bp_1) &= d_{\cF}(\overline{\bp}_1,\overline{\bp}_0),
\end{align}
for all $\ul{p}_0,\ul{p}_1 \in \R^2 \times S^1$, 
 where the Finsler function $\cF \in \{\cF_0,\cF_0^+\}$ is defined in \eqref{Eq: finsler function} and \eqref{F0plus}.
Next we show this by proving that there is a length-preserving bijection between the sets $\Gamma_{\mathcal{F}}(\overline{\bp}_1,\overline{\bp}_0)$ and $\Gamma_{\mathcal{F}}(\bp_0,\bp_1)$. 

Let $\gamma \in \Gamma_{\mathcal{F}}(\bp_0,\bp_1)$, where $\Gamma_{\cF}=\Gamma_0$ if $\cF=\cF_0$ and $\Gamma_{\cF}=\Gamma_0^+$ if $\cF=\cF_0^{+}$, as given in Def.~\ref{def:Gamma}. 

From curve $\gamma$ we create the curve $\hat{\gamma} \in \Gamma_{\mathcal{F}}(\overline{\bp}_1,\overline{\bp}_0)$ given by $\hat{\gamma}(t):=\overline{\gamma}(1-t)$.
We have $\dot{\hat{\gamma}}(t) = - \dot{\overline\gamma}(1-t)$.
The length of $\hat{\gamma}$ is given by:
\begin{equation}\label{eq:lengthhatgamma}
\begin{array}{rcl}
    L(\hat\gamma) &=&
      \int \limits_0^1\cF(\hat\gamma(t),\dot{\hat\gamma}(t))  \,{\rm d}t\\
    &=& \int \limits_0^1\cF(\overline{\gamma}(1-t),-\dot{\overline{\gamma}}(1-t)) \, {\rm d}t\\
    &=& \int \limits_0^1\cF(\overline{\gamma}(\tau),-\dot{\overline{\gamma}}(\tau)) \, {\rm d}\tau,
\end{array}
\end{equation}
where at the third equality we performed the substitution $\tau = 1 - t \in [0,1]$. 

To show that the length of $\gamma$ is the the same as the length of $\hat\gamma$, we need to show that
\begin{align}\label{Eq: prop finsler 1} 
     \cF\left(\overline{\gamma}(\tau), -\dot{\overline{\gamma}}(\tau)\right) = \cF\left(\gamma(\tau), \dot{\gamma}(\tau)\right).   
\end{align}
Recall from Def.~\ref{def:controlu1} that $u^1$ and $u^3$ are the controls of $\gamma$. 
For the curve $\tau \mapsto \overline{\gamma}(\tau) = (\bx(\tau),-\bn(\tau))$, the controls become $\overline{u}^1$ and $\overline{u}^3$ are given by:
\begin{equation} \label{newcontrols}
\begin{array}{ll} 
    \overline{u}^1(\tau)&:= \dot{\bx}(\tau) \cdot (-\bn(\tau)) = -u^1(\tau), \\
    \overline{u}^3(\tau)&:= \dot{\theta}(\tau) = u^3(\tau).
\end{array}
\end{equation}
By construction of $\cF$, we have that
\begin{equation*} 
\begin{array}{ll}
    \cF\left(\overline{\gamma}(\tau), -\dot{\overline{\gamma}}(\tau)\right)     
    &= \mathcal{C}(\overline{\gamma}(\tau))\sqrt{\xi^2|u^1(\tau)|^2 + |-u^{3}(\tau)|^2}\\
    &= \mathcal{C}(\gamma(\tau))\sqrt{\xi^2|u^1(\tau)|^2 + |u^{3}(\tau)|^2}\\
    &= \cF\left(\gamma(\tau), \dot{\gamma}(\tau)\right).
\end{array}
\end{equation*}
The second equality holds since the cost function is by assumption well-defined on $\R^2 \times P^1$, i.e. $\mathcal{C}(\bp)=\mathcal{C}(\overline{\bp})$ for all $\bp \in \R^2 \times S^1$.

If we substitute \eqref{Eq: prop finsler 1} in Eq. \eqref{eq:lengthhatgamma}, we have that
\begin{equation*}
    L(\hat\gamma) = L(\gamma).
\end{equation*}

Vice versa, we can apply the same trick to go from %curves 
$\hat\gamma \in \Gamma_{\mathcal{F}}(\overline{\bp}_1,\overline{\bp}_0)$ to %curves 
$\gamma \in \Gamma_{\mathcal{F}}(\bp_0,\bp_1)$, which implies that there is a length-preserving bijection between the sets $\Gamma_{\mathcal{F}}(\overline{\bp}_1,\overline{\bp}_0)$ and $\Gamma_{\mathcal{F}}(\bp_0,\bp_1)$ and (\ref{keyid}) follows.
\end{proof}

\subsection{Proof of 
Proposition\texorpdfstring{~\ref{Corollary: sign-semidefinite distance}}{: sign-semidefinite distance}}
\begin{proof} 
To prove Prop.~\ref{Corollary: sign-semidefinite distance}, we must establish two results. First, we need to show that $d_{\mathrm{c}}$ as defined in Eq.~\eqref{dmon} 
can be expressed as 
\begin{equation} \label{TBS}
     d_{\mathrm{c}}([\bp_0],[\bp_1])=\!\! \inf \limits_{\gamma \in \Gamma^c([\bp_0],[\bp_1])} \int \limits_0^1\cF_0(\gamma(t),\dot{\gamma}(t)) {\rm d}t.
\end{equation}
Second, we need to show that $d_{\mathrm{c}}$ is symmetric.

We begin by addressing the first point. Consider the set $\Gamma_0^-(\bp_0,\bp_1)$, which is given by Def.~\ref{def:Gamma} but now with the constraint $u^1\leq 0$. 
The Finsler function $\cF_0^-$ is defined by \eqref{F0plus} but now with the constraint $u^1\leq 0$. 
The set $\Gamma^c([\bp_0],[\bp_1])$ of cusp-free piecewise $C^1$-curves for which the sign of their control $u^1$ (Def.~\ref{def:controlu1}) is either nonnegative or nonpositive sign everywhere, is equal to:
\begin{equation}\label{eq:Gamma_c}
    \Gamma^c([\bp_0],[\bp_1]) = \Gamma^+_0([\bp_0],[\bp_1]) \cup \Gamma^-_0([\bp_0],[\bp_1]).
\end{equation}
We can rewrite the right-hand side of \eqref{TBS} as
\begin{align}\label{eq:statementprop1}
\begin{split}
&\inf \limits_{\gamma \in \Gamma^c([\bp_0],[\bp_1])} \int \limits_0^1\cF_0(\gamma(t),\dot{\gamma}(t))  \,{\rm d}t\\
&\begin{array}{rcll}    
    &\overset{\eqref{eq:Gamma_c}}{=}& \min \big\{ \\
    &&\inf \limits_{\gamma \in \Gamma^+_0([\bp_0],[\bp_1])} \int \limits_0^1\cF_0(\gamma(t),\dot{\gamma}(t))  \,{\rm d}t,\\
    &&\inf \limits_{\gamma \in \Gamma^-_0([\bp_0],[\bp_1])} \int \limits_0^1\cF_0(\gamma(t),\dot{\gamma}(t))  \,{\rm d}t \big \} \vspace{6pt}\\[6pt] 
    &=&\min \big\{
    d_{\cF_0^+}(\bp_0,\bp_1),\; d_{\cF_0^-}(\bp_0,\bp_1), \\ 
    && \qquad \ d_{\cF_0^+}(\bp_0,\overline{\bp}_1),\;
    d_{\cF_0^-}(\bp_0,\overline{\bp}_1),\\
    && \qquad \  d_{\cF_0^+}(\overline{\bp}_0,\bp_1), \; d_{\cF_0^-}(\overline{\bp}_0,\bp_1), \\  
    && \qquad \  d_{\cF_0^+}(\overline{\bp}_0,\overline{\bp}_1), \;d_{\cF_0^-}(\overline{\bp}_0,\overline{\bp}_1)\big\}.
\end{array}
\end{split}
\end{align}
Next we will show that $d_{\cF_0^+}(\bp_0,\bp_1)=d_{\cF_0^-}(\overline{\bp}_0,\overline{\bp}_1)$.
We show this by establishing a length-preserving bijection between $\Gamma_0^+(\bp_0,\bp_1)$ and $\Gamma_0^-(\overline{\bp}_0,\overline{\bp}_1)$. 
The bijection is defined by the mapping $\gamma(t) \mapsto \overline{\gamma}(t)$, with $\gamma \in \Gamma_0^+(\bp_0,\bp_1)$ and $\overline{\gamma} \in \Gamma_0^-(\overline{\bp}_1,\overline{\bp}_1)$. 
We need to show that
\begin{align*}
     \cF_0^-\left(\overline{\gamma}(t), \dot{\overline{\gamma}}(t)\right) = \cF_0^+\left(\gamma(t), \dot{\gamma}(t)\right), 
\end{align*}
in order to prove that the bijection is length-preserving.

Recall from Def.~\ref{def:controlu1} that $u^1$ and $u^3$ are the controls of $\gamma\in \Gamma_0^+(\bp_0,\bp_1)$. 
For the curve $\overline{\gamma}(\cdot) = (\bx(\cdot),-\bn(\cdot)) \in \Gamma_0^-(\overline{\bp}_0,\overline{\bp}_1)$, its controls $\overline{u}^1$ and $\overline{u}^3$ are given by (\ref{newcontrols}).
By construction of $\cF_0^-$, we have that
\begin{equation*} 
\begin{array}{ll}
    \cF_0^-\left(\overline{\gamma}(t), \dot{\overline{\gamma}}(t)\right) 
    &= \mathcal{C}(\overline{\gamma}(t))\sqrt{\xi^2|-u^1(t)|^2 + |u^{3}(t)|^2}\\
     &= \mathcal{C}(\gamma(t))\sqrt{\xi^2|u^1(t)|^2 + |u^{3}(t)|^2}\\
     &= \cF_0^+\left(\gamma(t), \dot{\gamma}(t)\right).
\end{array}
\end{equation*}
The second equality holds since $\mathcal{C}(\bp)=\mathcal{C}(\overline{\bp})$ for all $\bp \in \R^2 \times S^1$. This implies that 
\begin{equation}\label{Eq: prop finsler 2} 
   d_{\cF_0^-}(\overline{\bp}_0,\overline{\bp}_1)=  d_{\cF_0^+}(\bp_0,\bp_1) ,
\end{equation} 
for all $\ul{p}_0,\ul{p}_1 \in \R^2 \times S^1$. 
Hence, combining Eq.~(\ref{eq:statementprop1}) and Eq.~(\ref{Eq: prop finsler 2}) and (\ref{keyid})  (i.e. Lemma~\ref{Lemma: antipodal symmetry}) we find 
\begin{equation} \label{aap}
\begin{aligned}
& \inf\limits_{\gamma \in \Gamma^c([\bp_0],[\bp_1])} \int \limits_0^1\cF_0(\gamma(t),\dot{\gamma}(t))  \,{\rm d}t
\\
&\quad=
    \min \Big\{
    d_{\cF^{+}_0}(\bp_0,\bp_1), d_{\cF^{+}_0}(\overline{\bp}_0,\bp_1), \\
&\hspace{5em} d_{\cF^{+}_0}(\bp_0,\overline{\bp}_1), d_{\cF^{+}_0}(\overline{\bp}_0, \overline{\bp}_1) 
\Big\} 
\\
&\quad=d_{\mathrm{c}}([\bp_0],[\bp_1]).
\end{aligned}
\end{equation}
Now we have shown the first part of Prop.~\ref{Corollary: sign-semidefinite distance}. 

The second part, i.e., $d_{\mathrm{c}}([\ul{p}_0],[\ul{p}_1])=d_{\mathrm{c}}([\ul{p}_1],[\ul{p}_0])$ now follows directly from (\ref{aap}) and %Lemma~\ref{Lemma: antipodal symmetry} (i.e. (
(\ref{keyid}).
\end{proof}

% Let $\gamma$ be the minimizing geodesic of the model $\cF_0^+$. We have for the curve $\hat{\gamma}(t)=\overline{\gamma}(t)$ (while setting $\overline{\gamma}(t):=\overline{\gamma(t)}$) that:
% \begin{align*}
%     d_{\cF_0^+}(\overline{\bp}_0,\overline{\bp}_1) &= 
%     \inf \limits_{ 
%      \hat\gamma \in \Gamma_0(\overline{\bp}_0,\overline{\bp}_1)
%     }  \int \limits_0^1\cF_0^+(\hat\gamma(t),\dot{\hat{\gamma}}(t))  {\rm d}t \\
%     &= 
%     \inf \limits_{ \gamma \in \Gamma_0(\bp_0,\bp_1)
%     }  \int \limits_0^1\cF_0^+(\overline{\gamma}(t),\dot{\overline{\gamma}}(t)) {\rm d}t \\
%     &= 
%     \inf \limits_{  
%     \gamma \in \Gamma_0(\bp_0,\bp_1)
%     }  \int \limits_0^1\cF_0^-(\gamma(t),\dot{\gamma}(t))  {\rm d}t \\    
%     &= d_{\cF_0^-}(\bp_0,\bp_1). 
% \end{align*} 

\subsection{Proof of  Proposition~\ref{corr:tobewritten}}\label{App:PropGlobalLocalControllability}
\begin{proof}
Both the sub-Riemannian distance $d_{\cF_0}$ and the plus-control variant $d_{\cF_0^{+}}$ are globally controllable. This follows by the Chow-Rashevskii theorem and has been proven before in \cite[Thm. 1]{duitsmeestersmirebeauportegies}. 
The global controllability of $d_{\mathrm{c}}$ and $d_{\mathrm{proj}}$ amounts to finite distance functions $d_{\mathrm{c}}$, $d_{\mathrm{proj}}$, which then directly follows from \eqref{dmon} and \eqref{dproj} respectively.

Regarding the violation of local controllability for model $d_{\mathrm{c}}$, we recall that we assumed $\mathcal{C}\geq \delta$ in the definition \eqref{Eq: finsler distance forward} of $d_{\cF_0^+}$, so by \eqref{dmon} it remains to show that
\[
\limsup_{[\bp]\to[\bp_0]}d_{\mathrm{c}}([\bp],[\bp_0]) \geq  \pi.
\]
for the case $\mathcal{C}=1$. Now consider the special case where we approach $\bp_0$ with $\bp=\bp_1^{\varepsilon}$ from the side with the same orientation.
I.e., the setting \eqref{p1e}, with $0<\xi \varepsilon< \pi$. In the model $d_{\cF_0}$ the minimizing geodesic connecting $\bp_0$ and $\bp^{\varepsilon}_1=((0,\varepsilon),(1,0))$
has two cusps (in the `inflexional trajectory') as can be seen in Fig.~\ref{fig:lat} and understood from the phase portrait analysis (mathematical pendulum \cite{BoscainESAIM}), \cite[Fig.4]{Moiseev}. There
$d_{\cF_0}(\bp_1^{\varepsilon},\bp_0) \to 0$ and
$d_{\cF_0}(\overline{\bp}_1^{\varepsilon},\bp_0) \to \pi$ when $\varepsilon \downarrow 0$ by continuity of $d_{\cF_0}$.

Now in the geodesics of $d_{\mathrm{c}}$ such cusps are not present. 
By \eqref{dmon} we
either connect the pairs $\{\bp_0,\bp_1^{\varepsilon}\}, 
\{\overline{\bp}_0,\overline{\bp}_1^{\varepsilon}\}$ or the pairs
$\{\overline{\bp}_0,\bp_1^{\varepsilon}\}, 
\{\bp_0,\overline{\bp}_1^{\varepsilon}\}$ with the two minimizing geodesic in $d_{\cF_0^+}$.
The second 
case
 produces full (cusp-free) U-curve sub-Riemannian \cite[Fig.3]{Moiseev} that collapse to a $\pi$ switch in orientation as $\varepsilon \downarrow 0$ yielding distance $\pi$.
The first case cannot be connected with a cusp-free geodesic and the connecting minimizing geodesic in $d_{\cF_0^+}$ will have equal in-place rotations at the boundaries \cite[Thm.2, Fig.8-A1]{duitsmeestersmirebeauportegies} also collapsing to distance $\pi$ if $\varepsilon\downarrow 0$. Thereby \eqref{eq:bb} follows.
\end{proof}

\section{Distance Map Computation \label{app:B}}

\subsubsection*{Eikonal PDE for $\cF = \cF_0^+$}
For a Finsler function $\cF:T(\R^2 \times S^1) \to \R^+$ and $\bp_0 \in \R^2 \times S^1$, the eikonal PDE is given by
\begin{equation}\label{eq:eikonal_pde_appendix}
\begin{cases}           
\cF^*(\bp, {\rm d} W(\bp)) = 1, \text{for } \bp \neq \bp_0, \\
W(\bp_0) = 0,
\end{cases}
\end{equation}
where the dual Finsler function $\cF^*: T^{*}(\R^2 \times S^1)\to \R$ is defined by 
\begin{equation}\label{eq:dual_finsler_function}
\cF^*(\bp,\hat{\bp}):=
\sup_{\dot{\bp} \neq 0}
\frac{\langle \hat{\bp},\dot{\bp}\rangle}{\cF(\bp,\dot{\bp})},
\end{equation}
where we use the notation $\langle \hat{\bp},\dot{\bp}\rangle:=\hat{\bp}(\dot{\bp})$.

The viscosity solution $W: \R^2 \times S^1 \to \R$ of \eqref{eq:eikonal_pde} equals the induced Finslerian $d_\cF(\bp_0, \cdot)$ \cite{BekkersSIAM}. 

\begin{remark}[Eikonal PDE \eqref{eikF0plus}: the case $\cF=\cF_0^+$]\label{rem:plus_finsler} 
    We associate a Finsler function $\cF_0^+$ with the asymmetric metric $d_{\cF_0^+}$ \eqref{Eq: finsler distance forward}, namely
    $\cF_0^{+}$ given by \eqref{F0plus}.
    By direct computation \cite{duitsmeestersmirebeauportegies}, we can then find for $\cF = \cF_0^+$ the following Hamiltonian: 
    \begin{equation}\label{H0plus}
        \vert (\cF_0^+)^*(\bp, \hat{\bp}) \vert^2 = \mathcal{C}^{-2}(\bp)(\xi^{-2} (p_1)_+^2 + p_3^2)=1
    \end{equation}
    with $\hat{\bp} = p_1 \omega^1\vert_\bp + p_3 \omega^3\vert_\bp$ and $(a)_+ = \max\{a,0\}$.
    This explains the eikonal system \eqref{eikF0plus}, as that now directly follows by
    subsitution of
    \[
        \hat{\bp}={\rm d}W(\bp) = \sum_{i=1}^{3}\mathcal{A}_{i} W(\bp)\left.\omega^{i}\right|_{\bp}.
    \]
    into \eqref{H0plus}.
\end{remark}

% \begin{remark}[Hamiltonian view of eikonal PDE]
% We can see $\vert \cF^*(\bp, \cdot)\vert^2: T_\bp^*(\R^2 \times S^1) \to \R$ as a Hamiltonian.
% \end{remark}

\subsubsection*{Converging Sub-Riemannian Algorithm}

There are many ways to solve the eikonal PDE \eqref{eq:eikonal_pde}. 
One popular, efficient, and theoretically well-underpinned method is anisotropic fast marching. J.-M.~Mirebeau developed an advanced, efficient anisotropic fast marching methods \cite{mirebeau_anisotropic_2014} for sub-Riemannian metric tensor fields on $\R^2 \times S^1$, which relies on an anisotropic Riemannian relaxation \cite[Thm.~2]{duitsmeestersmirebeauportegies}. 
However, this anistropic fast marching method cannot deal with truly sub-Riemannian metric tensor fields (as it has %reasonable 
limitations on the Riemannian anisotropy in its stencils). 
This is one of the reasons (despite the numeric results in \cite{SanguinettiFM}) why we choose to use an iterative method, which exactly hardcodes the (sub-Riemannian) constraint. 
%to horizontal subspaces (recall Def.~\ref{def:Hor}).

The algorithm does not rely on an anisotropic Riemannian relaxation \cite[Thm.~2]{duitsmeestersmirebeauportegies}, as in advanced anisotropic fast-marching methods \cite{mirebeau_anisotropic_2014, mirebeau2019hamiltonian}, but it \emph{exactly hardcodes the (sub-Riemannian) constraint to horizontal subspaces (recall Remark~\ref{rem:sub-riemannian manifold}).} 
%(recall Def.~\ref{def:Hor}).

Anisotropic fast-marching is beneficial in terms of algorithmic complexity, but as argued in \cite{berg2025crossing}, the iterative method is simple, transparent, and flexible. Moreover, \emph{it is easily parallellizable}, and  can easily take advantage of GPU hardware, avoiding the %dramatic 
loss of computation time reported in \cite{SanguinettiFM}.

For each $t \in [t_n, t_{n+1}]$ with $t = n \varepsilon$ at step $n \in \mathbb{N} \cup \{0\}$ and for $\varepsilon > 0$, the iterative method is given by:
\begin{equation*}   
    \begin{cases}
        \frac{\partial W^\varepsilon_{n+1}}{\partial t}(\bp_1,t) = 1 -  \cF^*\left(\bp, {\rm d}W^{\varepsilon}_{n+1}(\bp_1,t)\right) , \\
        W^\varepsilon_{n+1}(\bp_1,t_n) = W^\varepsilon_{n}(\bp_1,t_n), \; \text{for} \; \bp_1 \neq \bp_0,\\
        W^\varepsilon_{n+1}(\bp_0,t_n) =0
    \end{cases}
\end{equation*}
for $n=\{1,2,\dots\}$ and for $n=0$ initialize
\begin{align*}    
    \begin{cases}
        \frac{\partial W^\varepsilon_1}{\partial t}(\bp_1,t) = 1- \cF^*\left(\bp, {\rm d}W^{\varepsilon}_1(\bp_1,t)\right) ,\\
        W^\varepsilon_1(\bp_1,0) = \delta_{\bp_0}^{\R^2 \times S^1}(\bp_1)
    \end{cases}    
\end{align*}
with   
$    
\delta_{\bp_0}^{\R^2 \times S^1}(\bp_1) =
    \begin{cases}
        0 &\text{if}\;\bp_0=\bp_1 \\
        +\infty&\text{else}.
    \end{cases}
$. \\ 
Finally one obtains (akin to \cite[Eq.4.4]{BekkersSIAM}) the Finslerian distance map $W: \R^2\times S^1 \to \R^{+}$ via
\[
d_{\cF_0^+}(\bp,\bp_0)=
W(\bp):= \lim \limits_{\varepsilon \downarrow 0} \lim_{n \to \infty} W^{\varepsilon}_{n}(\bp,n \varepsilon).
\]
In practice one sets $0<\varepsilon \ll 1$ small, uses standard finite differences, and if a fixed point is obtained up to a tolerance $10^{-5}$ on the relative error, one stops.

\section{Morphological Dilation \label{App:C}}

In this appendix, the concept of morphological dilation is explained, on which the grouping Algorithm~\ref{Algorithm: grouping of C} of $\mathcal{C}$ into $\mathcal{C}_{\mathrm{new}}$ is based. 
$\mathcal{C}_{\mathrm{new}}$ is used in the geodesic tracking model $d_{\mathrm{c}}$, as shown in Step 4a of Fig.~\ref{fig: flowchart}. 

\begin{definition}[Morphological convolution]
    Let $f_1,f_2: \R \times S^1 \rightarrow \R$ be two lower semi-continuous functions, and let the group be $\SE(2)$. Then their morphological convolution is given by:
    \begin{equation*}
        (f_1 \square f_2)(\bp) = \inf_{g \in \SE(2)} \{ f_1(g^{-1}\cdot\bp)+f_2(g\cdot\ul{p_0)}\},
    \end{equation*}
    using the group action \eqref{groupaction}.
\end{definition}
The morphological dilation $D(\bp,\tau)$ of an image $U \in C(\R^2 \times S^1)$ with a kernel $k_\tau$ is given by:
\begin{equation} \label{eq:dil}
    D(\bp,\tau) = -(k_\tau \square -U)(\bp).
\end{equation}
The kernel $k_\tau \geq 0$ that we use is given by:
\begin{align}\label{Eq: kernel dilation}
    k_\tau(\bp) &=   \begin{cases}
        0  &\text{if }d_{\mathcal{G}}(\bp_0, \bp) \leq \tau\\
        \infty &\text{else}
    \end{cases}
\end{align}
where $d_{\mathcal{G}}$ denotes the Riemannian distance given by \eqref{Eq: Riemannian distance}. Recall from \eqref{Eq: sub-Riemannian metric tensor} that the metric parameters of the left-invariant metric $\mathcal{G}$ are notated by $g_{11}, g_{22}$ and $g_{33}$. 
% Intuitively, these parameters represent the cost of tangential, lateral, and angular movement, respectively. Hence, they determine the shape of the ball of radius $\tau$, i.e. $d_{\mathcal{G}}(\bp_0,\cdot)\leq \tau$. 
%and therefore the result of the grouping algorithm. 

As computing $d_{\mathcal{G}}$ is demanding, we will use a distance approximation $\rho$ as proposed in \cite[Eq.~29]{bellaard2023analysis}. We can fix the radius of the ball to 1, as it can easily be derived from \cite[Eq.~30]{bellaard2023analysis} that multiplying the radius with $\tau$ boils down to rescaling metric parameters of $\{g_{ii}\}_{i=1}^3$ with a factor $\frac{1}{\tau}$.

\begin{comment}
the viscosity solution of the Hamilton-Jacobi-Bellmann (HJB) equations. For $\alpha>1$ the HJB-equations are given by:
\begin{equation*}
    \begin{cases}
        \frac{\partial W}{\partial t} (\bp,t) = \frac{1}{\alpha} H^\alpha (\mathrm{d}W(\bp,t))\\
        W(\bp,0) = U 
    \end{cases}
\end{equation*}
with the Hamiltonian:
\begin{equation*}
    H^\alpha(\mathrm{d}W(\bp))= \frac{1}{\alpha}|\cF^*(\bp, \mathrm{d}W(\bp))|^\alpha.
\end{equation*}
 \citet{bellaard2023analysis} show that the viscosity solution of the Hamilton-Jacobi-Bellmann equation is given by:
\begin{equation*}
    W(\bp,t) = -(k_t^\alpha \square -U)(\bp)
\end{equation*}
where the morphological kernel $k^\alpha_t: \R^2 \times S^1 \rightarrow \R_{\geq 0}$ is defined by for $\bp \in \R^2 \times S^1$:
\begin{equation*}
    k_t^\alpha(\bp)= \frac{t}{\beta}\left(\frac{d_{\mathcal{G}}(\bp_0, \bp)}{t}\right)^\beta
\end{equation*}
with $\frac{1}{\alpha}+\frac{1}{\beta}=1$, with reference element $\bp_0=\left((0,0), (1,0)\right)$, and with Riemannian distance $d_{\mathcal{G}}(\bp_0,\cdot)$ as given in Eq. \ref{Eq: Riemannian distance}. 
\end{comment}

\section{Initial Contour in \texorpdfstring{\unboldmath $\R^2 \times P^1$}{R2 x P1}}\label{app:initial_contour}

Recall the overall scheme in Fig.~\ref{fig: flowchart} and consider Step 3 in which the initial contours are visualized. 
These initial contours are the first (automatically obtained) estimate of the placement of the edges (of the electronic structure in the SEM image), and will be refined in subsequent steps. 

In the scheme of Fig.~\ref{fig: flowchart}, initial contours are derived from the output of the connected components $\{K_i\}_{i = 1}^n \subset \R^2 \times P^1$ using \cite{berg2024connectedcomponentsliegroups}, visualized in Step 2. Next we explain this derivation.

First, we obtain an initial contour $\tilde\gamma_{K_i}: [0,1] \rightarrow \R^2$ by performing a morphological edge detection in $\R^2$ (dilation minus erosion) \cite{MorphologicalEdgeDetection}[Eq. 14] on the spatial projection of the connected components $\{K_i\}_{i = 1}^n$. 
Then, the initial contours $\tilde{\gamma}_{K_i}$ are lifted towards contours $\gamma_{K_i}$ in $\R^2\times P^1$ as follows.

Let $\indbb{\tilde{\gamma}_{K_i}}$ be the indicator function of the initial contour $\tilde\gamma_{K_i}:[0,1] \rightarrow \R^2$ of the component $K_i$, given by $\indbb{\tilde\gamma_{K_i}}
%(\bx)
: \R^2 \rightarrow [0,1]$ with 
\begin{equation*}
    \indbb{\tilde\gamma_{K_i}}(\bx) = \begin{cases}
        1& \text{if } \exists  t \in[0,1], \text{ s.t. } \tilde\gamma(t)=\bx \\
        0& \text{otherwise}.
    \end{cases}
\end{equation*}
Then, the `trivial' lift is given by 
\begin{equation*}
    (W \indbb{\tilde\gamma_{K_i}})(\bx,\theta) := \indbb{\tilde\gamma_{K_i}} (\bx)
\end{equation*}
for all $\theta \in [0,\pi)$. 
The indicator function of the connected component $K_i$ is given by 
\begin{equation*}
    \indbb{K_i}(\bx,\theta)= \begin{cases}
        1& \text{if } (\bx,\theta)\in K_i\\
        0& \text{otherwise}
    \end{cases}
\end{equation*}
The contour $\gamma_{K_i}:[0,1]\rightarrow\R^2 \times P^1$ is obtained from the support of the direct product $(W \indbb{\tilde\gamma_{K_i}})\cdot (\indbb{K_i})$. 

\begin{remark}[Deviation of horizontality of $\gamma_{K_i}$]
   The initial contours $\gamma_{K_i}:[0,1]\rightarrow\R^2 \times P^1$ may not be horizontal, Def.~\ref{def:Gamma}. 
    The parts where the contour deviates from being horizontal require a different edge-refinement method than the horizontal parts, indicated by red and blue in Step 3 in Fig.~\ref{fig: flowchart}.
\end{remark}

In practice, all functions mentioned above are discretized by sampling $\bx \in \R^2$ and $\theta \in [0,\pi)$ on an equidistant grid. We apply a thinning algorithm  \cite{LEE1994462} to ensure that the width of the (discrete) initial contour in $\R^2 \times P^1$ is equal to one voxel.

\section{Electronic Structure Width}

\begin{figure*}
        \centering
        \includegraphics[width=\linewidth]{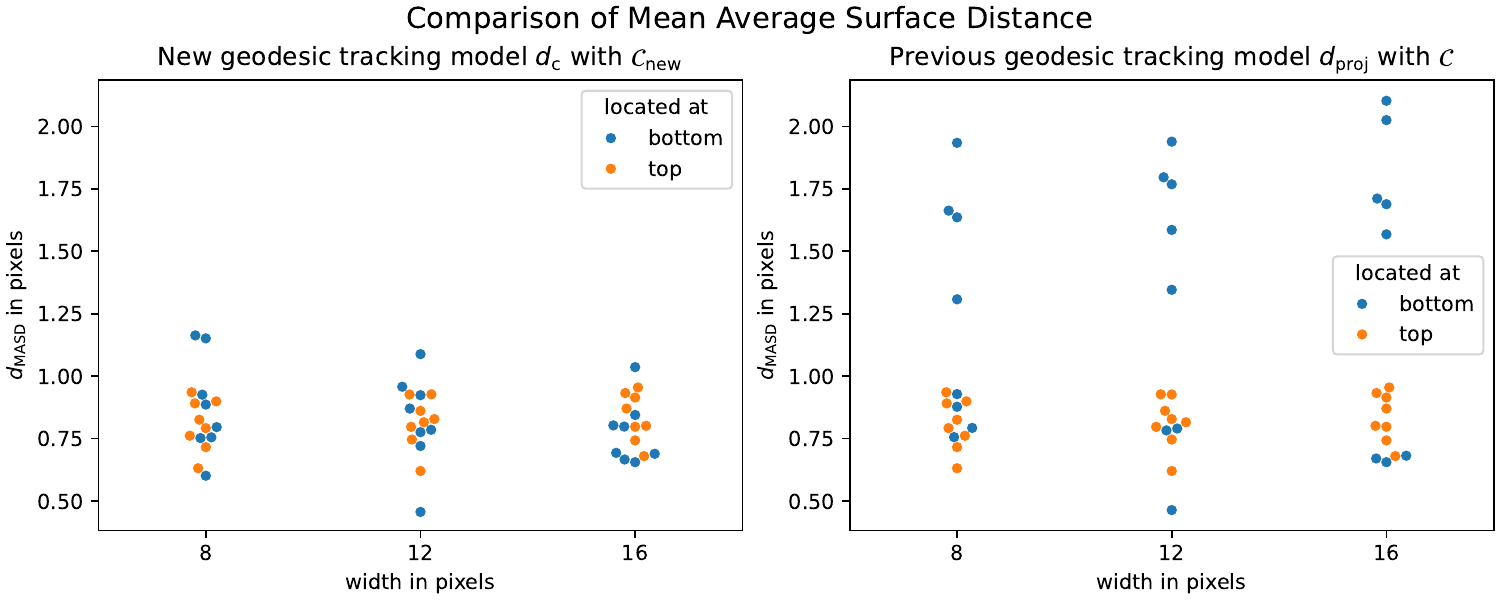}
        \caption{Comparison of the Mean Average Surface distance \eqref{Eq:MASD}: on the left our method (Fig.~\ref{fig: flowchart}) with $d_{\mathrm{c}}$ and 
        %$\{\mathcal{C}_{\tilde{K}_i}\}_i^n$
        $\mathcal{C}_{\mathrm{new}}$, on the right previous tracking model $d_{\mathrm{proj}}$ with $\mathcal{C}$, cf.~\eqref{dproj} and
        \cite{BekkersGSI}. Every point corresponds to an individual electronic structure in a SEM images, sorted by the width of such a structure.
        }
        \label{Fig: MASD distance Width}
\end{figure*}

\begin{figure*}
        \centering
        \includegraphics[width=\linewidth]{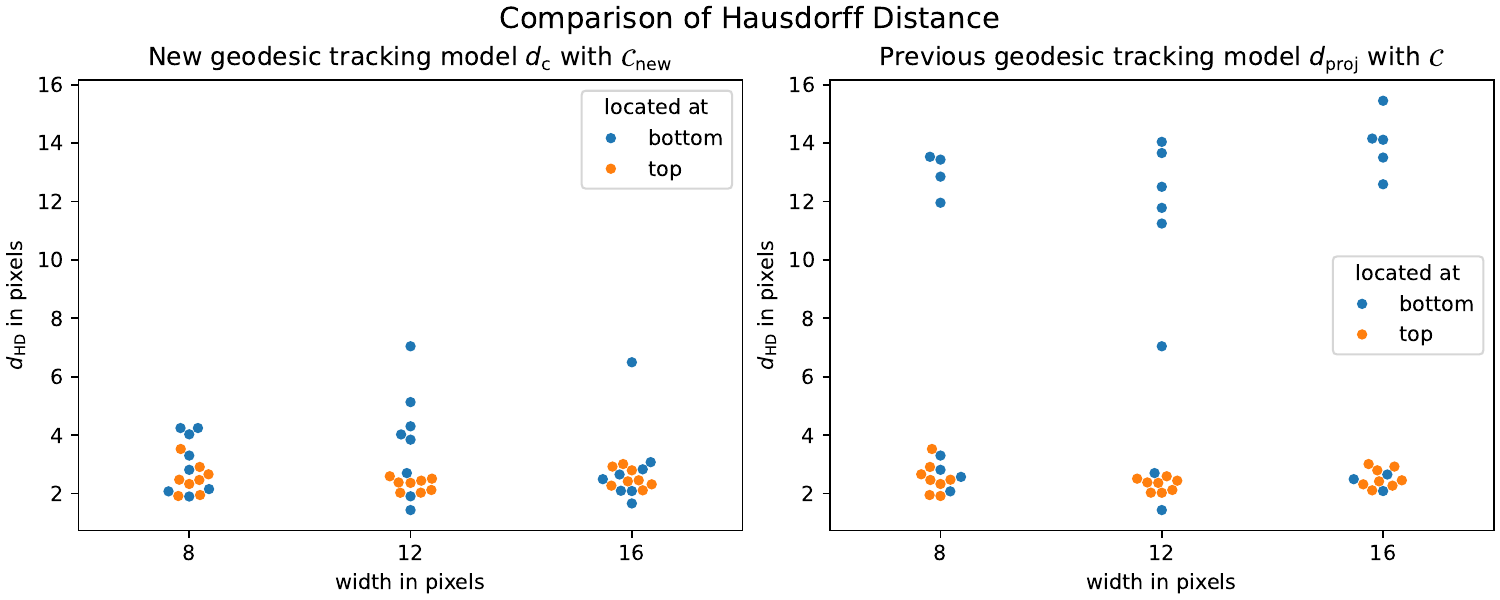}
        \caption{Comparison of the Hausdorff distance \eqref{Eq:MASD}: on the left our method (Fig.~\ref{fig: flowchart}) with $d_{\mathrm{c}}$ and $\mathcal{C}_{\mathrm{new}}$, on the right previous tracking model $d_{\mathrm{proj}}$ with $\mathcal{C}$, cf.~\eqref{dproj} and
        \cite{BekkersGSI}. Every point corresponds to an individual electronic structure in a SEM images, sorted by the width of such a structure.
%\cite{bekkers2018nilpotent}.
        }
        \label{Fig: Hausdorff distance Width}
\end{figure*}
The segmentation results of the overall framework depicted in Fig.~\ref{fig: flowchart} relying on both geodesic tracking models $d_{\mathrm{c}}$ and $d_{\mathrm{proj}}$, are stable when varying the width of the electronic structures in the SEM images, see Figs.~\ref{Fig: MASD distance Width},~\ref{Fig: Hausdorff distance Width}, where we also observe that $d_{\mathrm{c}}$ performs better than $d_{\mathrm{proj}}$.

\end{appendices}

{\small
\bibliography{sn-bibliography}
}% common bib file
\end{document}